\documentclass[12pt]{amsart}
\usepackage[dvipsnames]{xcolor}
\usepackage{fullpage,graphicx,psfrag,amsmath,amsfonts, amssymb,enumitem, hyperref, forest}
\usetikzlibrary{positioning}

\hypersetup{
	colorlinks=true,
	linkcolor=NavyBlue,
	urlcolor = RoyalBlue,
	citecolor = PineGreen
}
\usepackage{makecell}

\numberwithin{equation}{section}

\newcommand{\R}{\mathbb{R}}
\newcommand{\C}{\mathbb{C}}
\newcommand{\Z}{\mathbb{Z}}
\newcommand{\N}{\mathbb{N}}

\newcommand{\A}{\mathbb{A}}
  
\newcommand{\diam}{{\mbox{diam}}}

\newcommand{\Cov}{\mathop{\bf Cov{}}}

\newcommand{\dist}{\mathop{\bf dist{}}}

\newcommand{\eg}{{\it e.g.}}
\newcommand{\ie}{{\it i.e.}}

\newtheorem{theorem}{Theorem}[section]
\newtheorem{prop}[theorem]{Proposition}
\newtheorem{fact}[theorem]{Fact}

\newtheorem{lemma}[theorem]{Lemma}
\newtheorem{definition}[theorem]{Definition}
\newtheorem{problem}{Problem}

\renewcommand*{\P}{\mathbb P}

\renewcommand*{\Re}{\mathrm{Re}}
\renewcommand*{\Im}{\mathrm{Im}}

\newcommand{\mclS}[2]{\hyperref[eq:super-solutions]{\mathcal{S}^{#1}_{#2}}}
\newcommand{\mclSshift}[2]{\hyperref[eq:shifted-supersolutions]{\mathcal{\tilde{S}}^{#1}_{#2}}}

\newcommand{\lss}[2]{\hyperref[eq:least-super-solution]{w^{#1}_{#2}}}
\newcommand{\lssshift}[2]{\hyperref[eq:appendix-least-supersolution]{\tilde{w}^{#1}_{#2}}}

\newcommand{\cluster}[2]{\hyperref[eq:non-co-set]{\mathit{\Lambda}^{#1}_{#2}}}

\newcommand{\regcluster}[2]{\hyperref[eq:regular-cluster]{\Lambda^{#1}_{#2}}}

\newcommand{\odometer}[2]{\hyperref[eq:limit-odometer]{v^{#1}_{#2}}}

\newcommand{\Mrho}[2]{\hyperref[eq:alternate-size]{M_{#1}(#2)}}
\newcommand{\SGrho}[2]{\hyperref[eq:alternate-size-2]{SG_{#1}(#2)}}
\newcommand{\SGrhofull}[4]{\hyperref[eq:alternate-size-2]{SG_{#3 #1, #4 #1}(#2)}}

\title{Harmonic balls in Liouville quantum gravity}
\author{Ahmed Bou-Rabee}
\author{Ewain Gwynne}

\begin{document}
	
	\begin{abstract}
		Harmonic balls are domains which satisfy the mean-value property for harmonic functions. 
		We establish the existence and uniqueness of harmonic balls on Liouville quantum gravity (LQG) surfaces using the obstacle problem formulation of Hele-Shaw flow. We show that LQG harmonic balls are neither Lipschitz domains nor LQG metric balls, and that the boundaries of their complementary connected components are Jordan curves.
		
		We conjecture that LQG harmonic balls are the scaling limit of internal diffusion limited aggregation (IDLA) on random planar maps. 
		In a companion paper, we prove this in the special case of mated-CRT maps.
	\end{abstract}
	\maketitle
	\section{Introduction} \label{sec:intro}
	
	\subsection{Overview}
	\begin{figure}
		\includegraphics[width=0.25\textwidth]{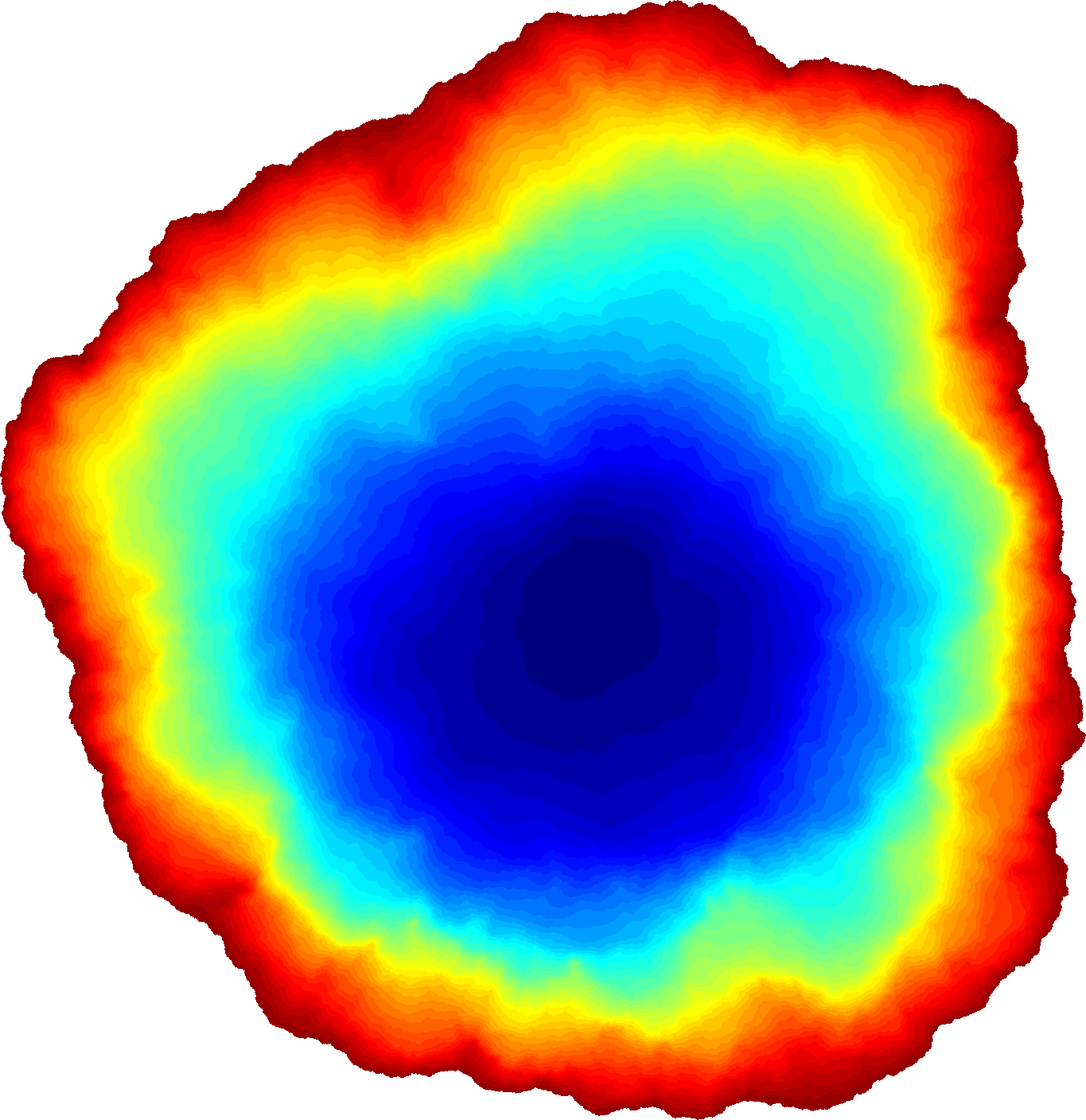} \quad
		\includegraphics[width=0.25\textwidth]{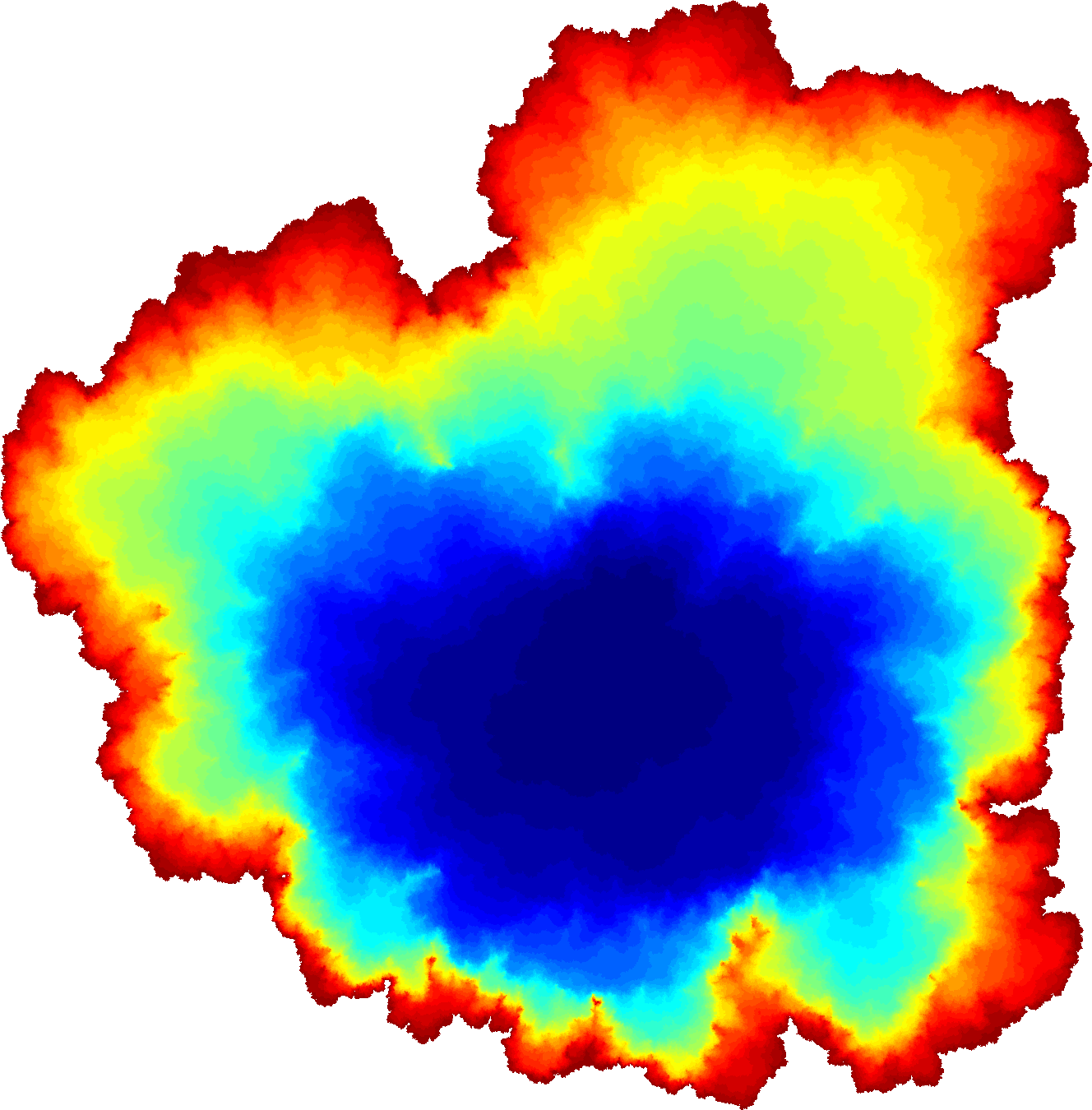} \quad
		\includegraphics[width=0.25\textwidth]{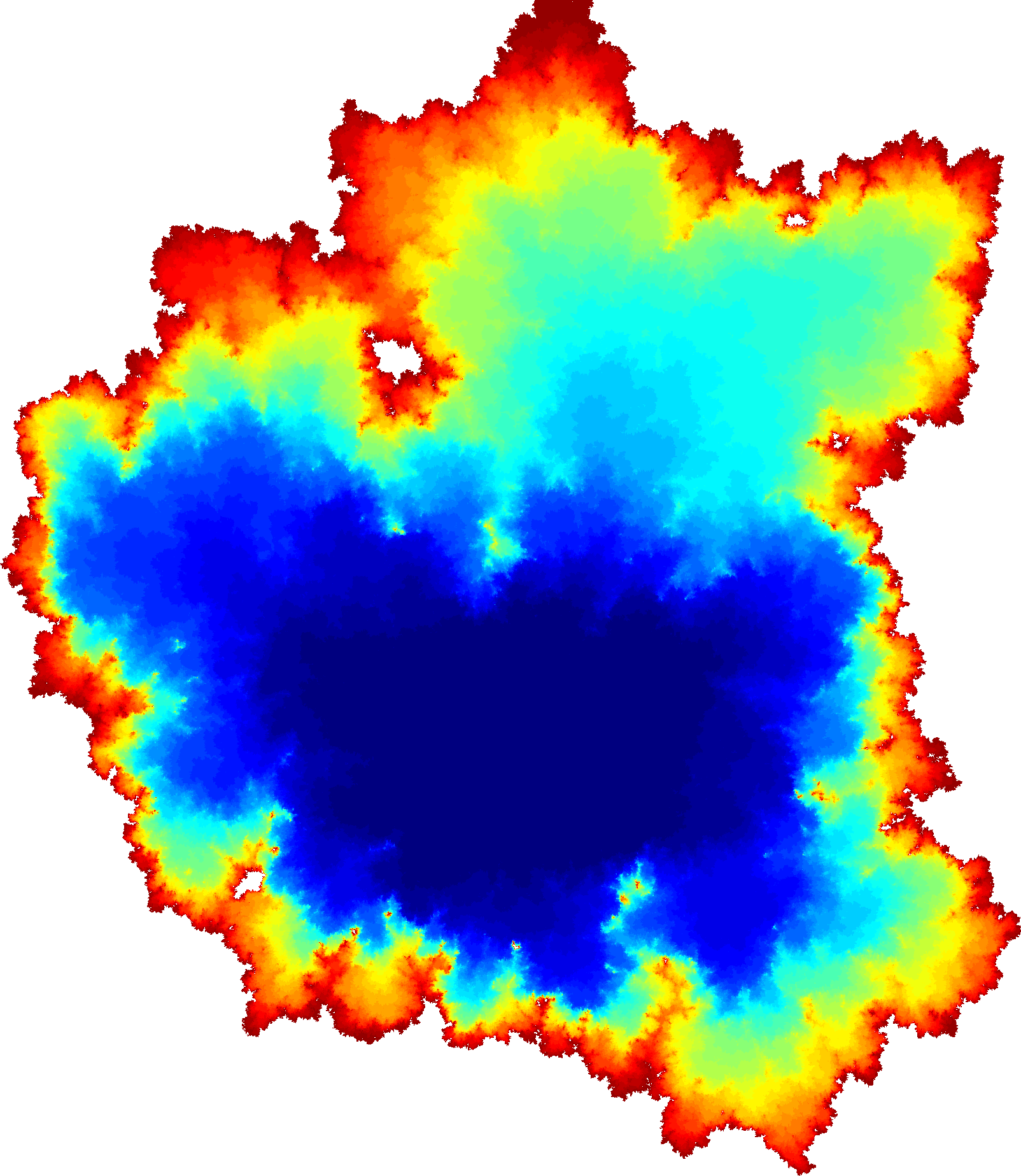}
		\caption{Simulations of $\gamma$-LQG harmonic balls with respect to the same GFF instance for $\gamma = 0.5, 1.0, 1.8$. The colors distinguish harmonic balls of different mass. From the figures, it appears that the complement of an LQG harmonic ball is not necessarily connected. We expect that this is the case for a `typical' LQG harmonic ball. }  \label{fig:lqg-harmonic-balls}
	\end{figure}

	Let $\mu$ be a locally finite Radon measure on $\C$ . A {\it harmonic ball for $\mu$ centered at~$z \in \C$} is an open set $\Lambda(z) \subset \C$ containing~$z$ which satisfies the mean-value property 
	\begin{equation} \label{eq:harmonic-ball}
	f(z) = \frac{1}{\mu(\Lambda(z))} \int_{\Lambda(z)} f(x) \mu(d x),
	\end{equation}
	for all functions $f$ which are harmonic in a neighborhood of the closure of $\Lambda(z)$. 
	
	\subsubsection{Liouville quantum gravity}
	In this article, we will construct and study harmonic balls in the setting of \textit{Liouville quantum gravity (LQG)}. LQG is a canonical one-parameter family of random fractal surfaces which were introduced by Polyakov in the 1980s in the context of bosonic string theory~\cite{polyakov-qg1}. One sense in which these surfaces are canonical is that they are known or conjectured to describe the scaling limits of various types of random planar maps (see Section~\ref{subsec:background} for more details).  
	
	Heuristically, for $\gamma \in (0,2)$ and a domain $U\subset\C$, a $\gamma$-LQG surface parameterized by $U$ is the two-dimensional Riemannian manifold with Riemannian metric tensor $e^{\gamma h} \, (dx^2+dy^2)$, where $dx^2+dy^2$ is the Euclidean metric tensor and $h$ is a variant of the Gaussian free field (GFF) on $U$. This metric tensor does not make literal sense since $h$ is a random generalized function, not a true function. Nevertheless, it is still possible to define the associated volume form, a.k.a.\ the \textit{$\gamma$-Liouville measure}. This is a random, locally finite Radon measure on $U$ which is informally given by
	\begin{equation} \label{eq:lqg-measure}
	\mu_h(dz) = e^{\gamma h(z)} \, dz  ,
	\end{equation}
	where $dz$ denotes Lebesgue measure. 
	
	The expression \eqref{eq:lqg-measure} can be made rigorous as part of a general theory 
	of regularized random measures called {\it Gaussian multiplicative chaos} \cite{kahane1985chaos, robert2010gaussian, rhodes2014gaussian}. This theory shows that the Liouville measure is a well-defined (random) Radon measure which is measurable with 
	respect to $h$ (the converse is also true \cite{berestycki2014equivalence}). However, the Liouville measure is quite irregular: $\mu_h$ is supported on the `thick' points of the GFF, a dense fractal set of Hausdorff dimension $2- \frac{\gamma^2}{2}$, and hence is mutually singular with respect to Lebesgue measure \cite{duplantier2011liouville, hu2010thick}.
	
	There is a vast literature on LQG: see~\cite{gwynne2020random,sheffield-icm} for introductory survey articles and~\cite{berestycki2021gaussian} for a more detailed introduction. However, only minimal prior knowledge of this literature is needed to read this paper. 
	The necessary background will be reviewed in Section~\ref{sec:gff-and-lqg}.
	
	\subsubsection{LQG harmonic balls}
	We will be interested in harmonic balls, as defined in~\eqref{eq:harmonic-ball}, in the case when $\mu$ is the $\gamma$-Liouville measure, $\mu_h$, for some $\gamma\in (0,2)$. 
	We call these domains {\it $\gamma$-LQG harmonic balls}. 
	One of our main motivations for studying LQG harmonic balls is that we expect them to be the scaling limits of internal diffusion limited aggregation (IDLA)  \cite{lawler1992internal} on random planar maps. 
	See Section~\ref{subsec:background} for further discussion. 
	
	We will construct $\gamma$-LQG harmonic balls via a certain partial differential equation involving $\mu_h$. In particular, this PDE {\it Hele-Shaw flow}\footnote{The family of $\gamma$-LQG harmonic balls we construct are weak solutions to a Hele-Shaw problem, see~\cite[Section 3.5]{gustafsson2006conformal} and~\cite{roos2016partial}.} (defined in Section~\ref{subsec:obstacle-def} below)
	describes the movement of a Newtonian fluid on a $\gamma$-LQG surface.
	The irregularity of $\mu_h$ precludes applying the classical theory of existence and uniqueness of harmonic balls \cite{sakai1984solutions, friedman1988free, gustafsson1990onquadraturedomains, caffarelli1998obstacle} in the LQG setting. 
	Much of the existing technology requires $\mu$ to be bounded from above and below by a multiple of the Lebesgue measure --- a constraint 
	too strict to be satisfied, even approximately, by $\mu_h$. Consequently, even the existence of $\gamma$-LQG harmonic balls is far from obvious.
	Indeed, nontrivial harmonic balls do not exist in general --- take, for instance,  $\mu$ to be a Dirac measure. 
	
	We address this difficulty by following a different path, inspired by arguments from discrete Laplacian growth \cite{duminil2013containing, jerison2013internal}. Roughly speaking, we show that it is unlikely for Brownian motion to avoid regions of large $\mu_h$-mass
	and then use this to show harmonic balls exist. This argument also leads to
	geometric requirements harmonic balls satisfy. We later use these requirements to show that `typical' $\gamma$-LQG harmonic balls are neither Lipschitz domains nor LQG metric balls.

	In contrast to the constructions of other objects associated with LQG, such as the LQG measure, the LQG metric, and Liouville Brownian motion, our construction of LQG harmonic balls does not use any approximation or regularization procedure. Rather, LQG harmonic balls are constructed directly as the solutions of an optimization problem involving the LQG measure (see Section~\ref{subsec:obstacle-def}).

	An LQG surface is a certain type of random fractal (albeit not one defined as a subset of $\mathbb R^d$ for some $d$). Analysis on fractals is a well-studied topic, see, \eg, \cite{kigami2001analysis, strichartz2006differential, strichartz2020analysis}.
	One program of research in this area is to construct the Laplacian on a fractal and then use this to develop a theory of elliptic PDE on the fractal. 
	As discussed further in Section \ref{sec:gff-and-lqg}, Brownian motion, and hence the Laplacian, has been constructed on LQG surfaces \cite{garban2016liouville, berestycki2015diffusion}.
	The present paper may be thought of as an initial step in the study of PDE on LQG surfaces.

	\subsection{Statement of the main result}

	Fix the LQG parameter $\gamma\in (0,2)$. Let $h$ be a whole-plane Gaussian free field, or more generally a whole-plane Gaussian free field plus the function $-\boldsymbol{\alpha}_0 \log|\cdot|$, where $\boldsymbol{\alpha}_0 < Q := 2/\gamma+\gamma/2$. Let $\mu_h$ be the $\gamma$-LQG area measure associated with $h$.
	The precise definitions of $h$ and $\mu_h$ will be reviewed in Section \ref{sec:gff-and-lqg}.
	For now, the unfamiliar reader can think of $\mu_h$ as a random, non-atomic, locally finite Borel measure on $\C$ which assigns positive
	mass to every open subset of $\C$.

	Our main result concerns existence and uniqueness of a family of $\gamma$-LQG harmonic balls. To prove uniqueness, we enlarge the class of harmonic functions in the definition to include those of the form,
	\begin{equation} \label{eq:harmonic-fns}
	\tilde H_O(D) = \{ \int_O G_O(\cdot,y) \,\nu(dy) : \mbox{$\nu$ is a signed Radon measure with support in $O \backslash D$}\},
	\end{equation}
	where $O, D \subset \C$ are bounded open sets and $G_O$ is the Green's function for Brownian motion in the domain $O$
	(defined in Section \ref{subsec:green} below). 
	That is, $\Lambda(z)$ is a harmonic ball centered at~$z \in \C$, if \eqref{eq:harmonic-ball} is
	satisfied for all functions $f = f_1 + f_2$ where $f_1$ is harmonic in a neighborhood of the closure of $\Lambda(z)$
	and $f_2 \in \tilde H_O(\Lambda(z))$ for some $O$ with~$\Lambda(z) \subset \overline{O}$.

	We are now ready to state our main existence and uniqueness result. 
	
	\begin{theorem}[Existence and uniqueness of harmonic balls] \label{theorem:harmonic-balls}
		On an event of probability one, for each~$z \in \C$, there exists a unique family of harmonic balls $\{\Lambda_t(z)\}_{t > 0}$
		satisfying the following properties:
		\begin{enumerate}[label=(\alph*)]
			\item For each $t > 0$ and~$z\in\C$, $\mu_h(\Lambda_t(z)) = t$, $\mu_h(\partial \Lambda_t(x)) = 0$, and $\Lambda_t(z)$ is equal to the interior of its closure.
			\item The domains $\Lambda_t(z)$ are bounded, connected, contain~$\{z\}$, increase continuously in $t$ (in the Hausdorff topology), and satisfy $\cap_{t > 0} \Lambda_t(z) = \{z\}$.
		\end{enumerate}
	\end{theorem}
	In some literature on harmonic balls, \eg, \cite{sakai1984solutions} or \cite{hedenmalm2002hele}, uniqueness is only proven up to sets of zero mass. In our setting, we get exact uniqueness thanks to the requirement that $\Lambda_t(z)$ is equal to the interior of its closure.
	
	We also show that typical harmonic balls are `novel'; that is, they are 
	neither Euclidean balls nor LQG metric balls. We also show that the boundaries of their complementary connected components are Jordan curves. Compare Figures \ref{fig:lqg-harmonic-balls} and \ref{fig:lqg-metric-balls}. 
	The precise definition of the LQG metric ball will be given in Section \ref{sec:gff-and-lqg}; for now the reader may think of 
	it as the natural notion of metric ball on an LQG surface.

	\begin{theorem}[Novelty of harmonic balls] \label{theorem:harmonic-balls-properties}
		The following is true on an event of probability 1. For Lebesgue-a.e.\ t, $\Lambda_t(0)$, constructed in Theorem \ref{theorem:harmonic-balls}, 
		is neither a Lipschitz domain nor an LQG-metric ball. Moreover, for each $t>0$ the boundaries of the connected components of $\C\setminus \overline{\Lambda}_t(0)$ are Jordan curves. 
	\end{theorem}

	We remark that the harmonic balls given by Theorem \ref{theorem:harmonic-balls} are locally determined by $\mu_h$, in the sense of the following statement. 	
	\begin{prop} \label{prop:locally-determined}
		For each fixed $t\geq 0$ and~$x \in \C$, the closed LQG harmonic ball $\overline{\Lambda}_t(x)$ is a local set for $h$ in the sense of~\cite[Lemma 3.9]{ss-contour}, \ie, for each deterministic open set $U\subset\C$, the event $\{\overline{\Lambda}_t(x) \subset U\}$ is measurable with respect to $\sigma(h|_U)$.
	\end{prop}

	\begin{figure}
		\includegraphics[width=0.25\textwidth]{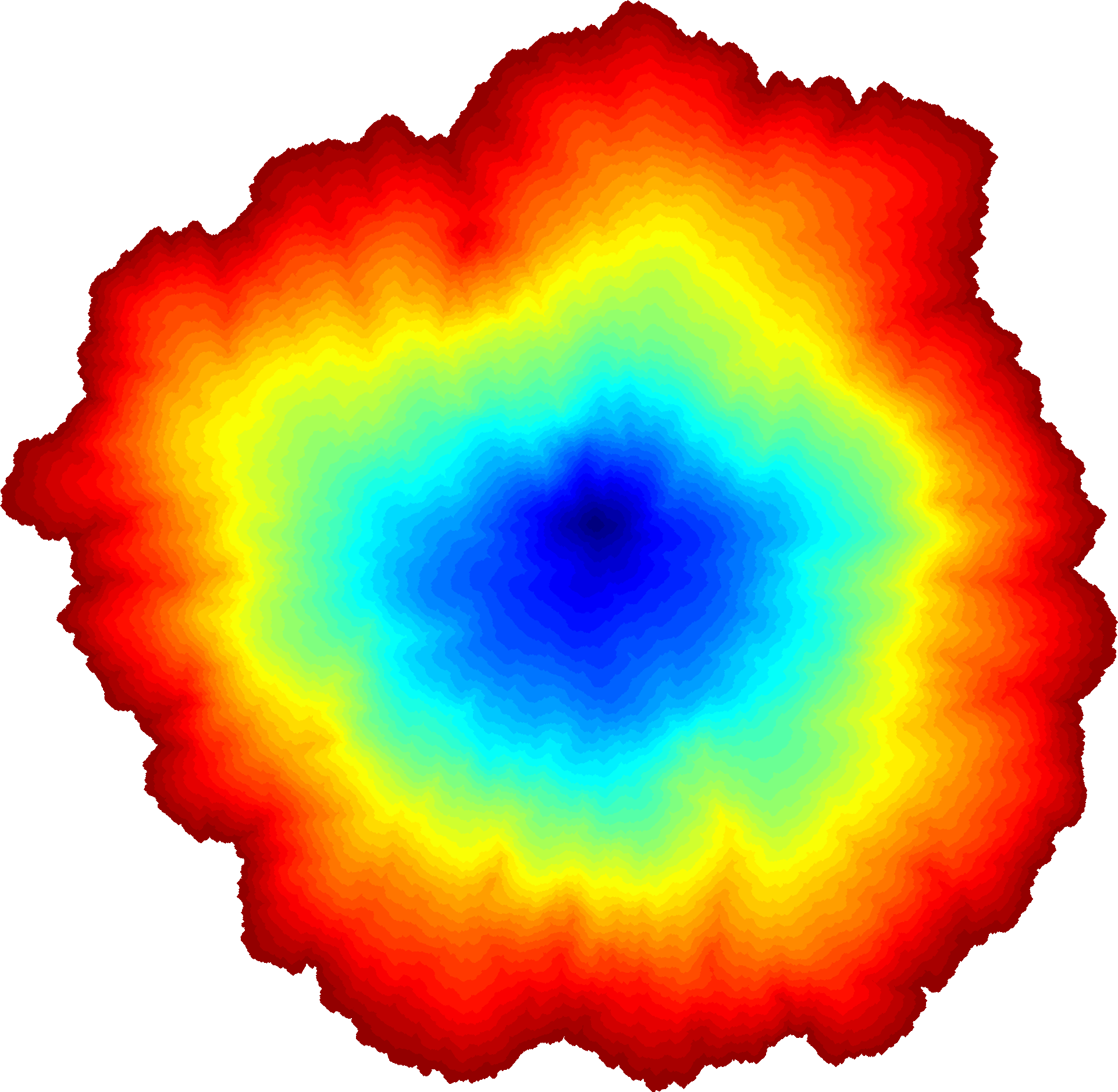} \quad
		\includegraphics[width=0.25\textwidth]{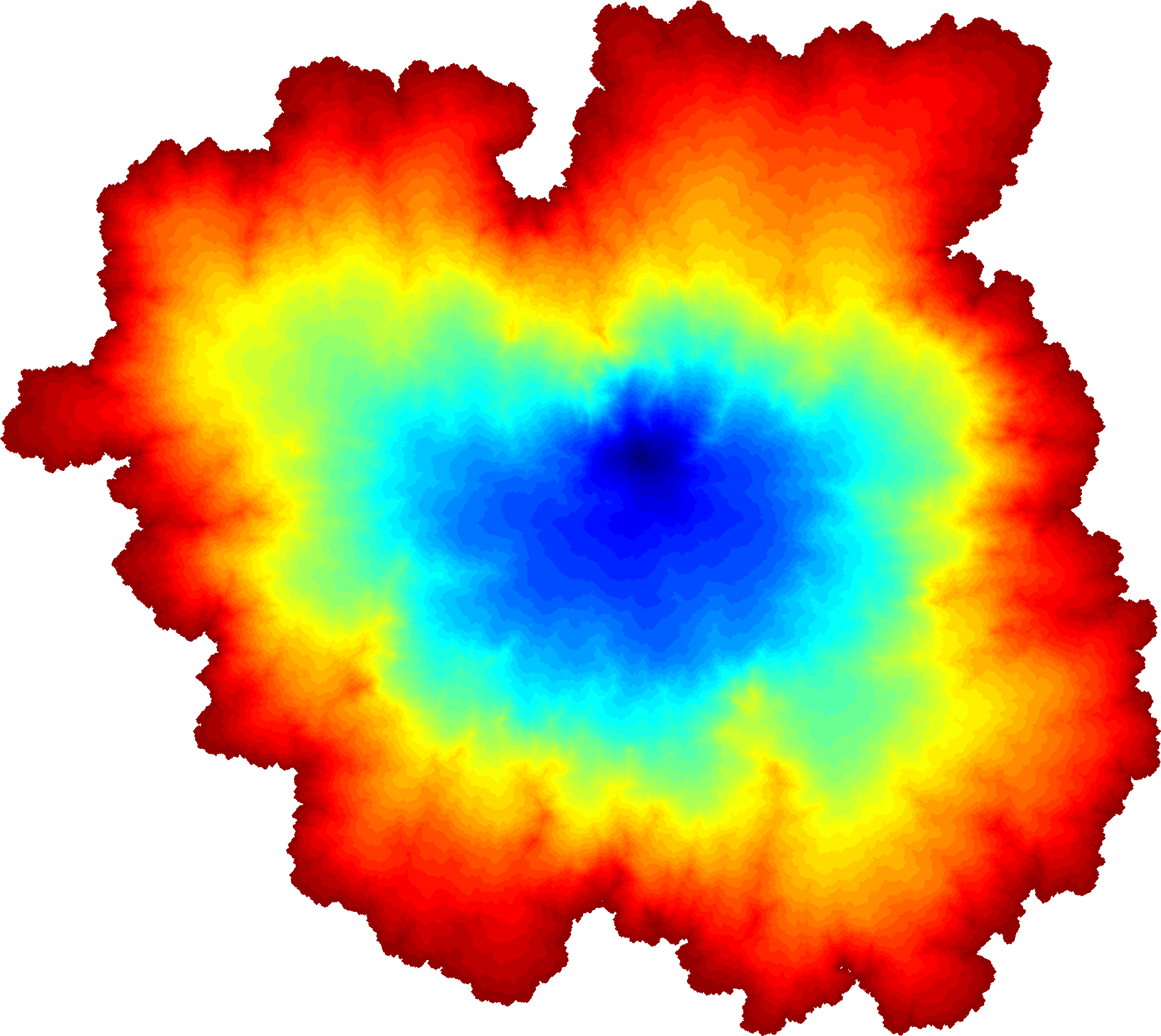} \quad
		\includegraphics[width=0.25\textwidth]{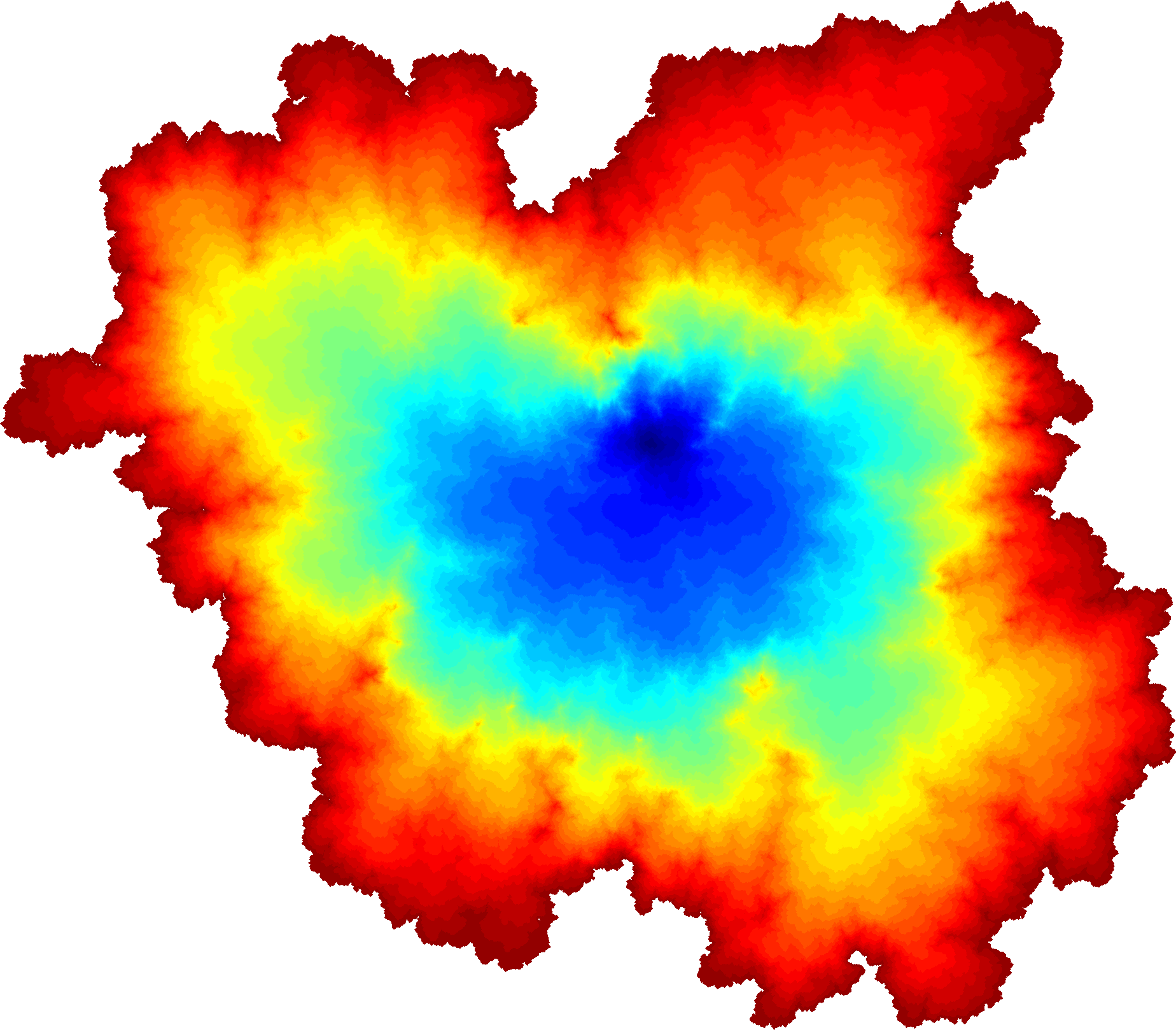}
		\caption{Simulations of $\gamma$-LQG metric balls with respect to the same GFF instance as Figure \ref{fig:lqg-harmonic-balls} for $\gamma = 0.5, 1.0, 1.8$.} \label{fig:lqg-metric-balls}
	\end{figure}

	\subsection{Background} \label{subsec:background}
	We provide some context and motivation for the study of harmonic balls on LQG surfaces. 
	
	\subsubsection{Harmonic balls}
	The term `harmonic ball' was coined by Shahgholian-Sj{\"o}din in \cite{shahgholian2013harmonic}
	and is a special case of {\it quadrature domains} for harmonic functions.
	A quadrature domain is a subset of $\mathbb{C}$ for which the integral of a harmonic function can be expressed 
	as a sum of simpler functionals (such as point evaluations). 
	Quadrature domains have a long history and are closely related to classical balayage (sweeping) \cite{doob1984classical}, Hele-Shaw 
	flow \cite{gustafsson2006conformal}, obstacle problems \cite{sakai1984solutions}, and Laplacian growth \cite{zidarov1990inverse, levine2017laplacian}.
	We direct the interested reader to \cite{gustafsson2005quadrature,gustafsson2004lectures}
	and the introductions of the theses of Roos \cite{roos2016partial} and Sj{\"o}din \cite{sjodin2005topics} for
	excellent expositions.
	
	Of particular relevance to our work are the papers of Hedenmalm-Shimorin \cite{hedenmalm2002hele} and Gustafsson-Roos \cite{gustafsson2018partial} which construct and analyze harmonic balls on Riemannian manifolds. Hedenmalm-Shimorin, building upon the work of Sakai \cite{sakai1991regularity}, show that harmonic balls on sufficiently smooth hyperbolic surfaces have boundaries which are the unions of a finite number of real analytic simple curves.
	Gustafsson-Roos show that harmonic balls and geodesic balls coincide on Riemannian surfaces 
	if and only if the Gaussian curvature of the manifold is constant. We emphasize that while some of the 
	basic constructions in these works may be adapted to our setting, an LQG surface is not a Riemannian
	manifold in the literal sense and so the results do not apply.

	\subsubsection{Internal DLA}
	One of our main motivations for studying LQG harmonic balls stems from a connection with internal diffusion limited aggregation (IDLA) on random planar maps. 
	IDLA was introduced as a toy model for chemical corrosion in \cite{meakin1986formation} and is a special case of a growth model studied by Diaconis-Fulton in \cite{diaconis1991growth}. IDLA is a random aggregation model defined as follows: start with $n$ walkers at the origin in $\Z^2$ and let each walker evolve according to a simple random walk until it reach a site in $\Z^2$ not occupied by any previous walker. This rule
	generates a growing sequence of sets $A_n \subset \Z^2$ indexed by the number of walkers $n \in \N$. 
	
	In a foundational work, Lawler-Bramson-Griffeath proved that $A_n$, suitably rescaled, converges to a Euclidean ball in $\R^2$
	as $n$ goes to infinity \cite{lawler1992internal, lawler1995subdiffusive}. Later, Asselah-Gaudillère \cite{asselah2013logarithmic, asselah2013sublogarithmic} and independently Jerison-Levine-Sheffield established logarithmic fluctuations of $A_n$ around its limit \cite{jerison2012logarithmic, jerison2013internal, jerison2014internal}. 
	
	Implicit in the proof of Lawler-Bramson-Griffeath is that harmonic balls are Euclidean balls when $\mu$ is the Lebesgue measure ---
	a simple proof of this was given by {\"U}lk{\"u} Kuran in 1972 \cite{kuran1972mean}. Interestingly, the connection between IDLA and quadrature domains generalizes. Levine-Peres showed in \cite{levine2010scaling} that the scaling limit of IDLA for any initial condition (for example multiple point sources) is given by a corresponding quadrature domain. 
	
	IDLA has also been studied on several other graphs including: Cayley graphs of groups with polynomial \cite{blachere2001internal} and  exponential \cite{blachere2007internal, huss2008internal} growth, supercritical percolation clusters \cite{shellef2010idla, duminil2013containing}, Sierpinski gasket graphs \cite{chen2017internal}, and cylinders \cite{jerison2014chapter, levine2019long, silvestri2020internal}--- see \cite{sava2021fractals} for a thorough survey of IDLA. In each of these cases the limit shape is either a Euclidean ball or a metric ball. On the other hand, Asselah-Rahmani showed that IDLA on the comb lattice has a limit shape which is neither a Euclidean ball nor a metric ball but rather a domain which satisfies a certain mean-value property \cite{asselah2016fluctuations} (see also \cite{huss-sava-comb}). 
	In a similar vein, Lucas has shown that IDLA with biased random walkers on $\Z^d$ converges under parabolic scaling to a domain which satisfies the mean-value property for caloric functions \cite{lucas2014limiting}.  
	
	\subsubsection{Random planar maps}
	A \textit{planar map} is a graph embedded in $\C$ in such a way that no two edges cross, viewed modulo orientation preserving homeomorphisms $\C \rightarrow \C$. 
	Various types of random planar maps are expected, and in some cases proven, to converge to $\gamma$-LQG surfaces. For example, uniform random planar maps (including uniform triangulations, quadrangulations, etc.) converge to $\sqrt{8/3}$-LQG surfaces in the Gromov-Hausdorff sense~\cite{legall-uniqueness,miermont-brownian-map,lqg-tbm1,lqg-tbm2} and, at least in the case of triangulations, when embedded into $\C$ via the so-called \textit{Cardy embedding}~\cite{hs-cardy-embedding}. Similar convergence results are expected to hold for various types of non-uniform random planar maps toward $\gamma$-LQG with $\gamma\not=\sqrt{8/3}$. For example, random planar maps sampled with probability proportional to the number of spanning trees they admit are expected to converge to $\sqrt 2$-LQG. We refer to~\cite{ghs-mating-survey} for a survey of work relating random planar maps and LQG.

	The aforementioned results on IDLA suggest that the scaling limits of IDLA on random planar maps are described by harmonic balls on LQG surfaces. 
	Random walks on (reasonably embedded) random planar maps are also expected to converge to (time changes of) Brownian motions on LQG surfaces
	--- this has recently been proven for a one-parameter family of random planar maps called mated-CRT maps in \cite{gms-tutte,berestycki2020random}.   
	In a companion work \cite{bou2022idla}, we verify that the scaling limit of IDLA on mated-CRT maps is given by LQG harmonic balls. It is still an open problem to prove this for other random planar map models, e.g., uniform random planar maps (see Problem~\ref{prob:planar-map}).

	\subsection{Open problems} \label{subsec:open-problems}
	We collect some questions suggested by this work. The first has been mentioned previously
	and is arguably the most important question here. 
	\begin{problem} \label{prob:planar-map}
		Show that the scaling limit of IDLA on random planar maps in the appropriate $\gamma$-LQG universality class, other than mated-CRT maps, is described by $\gamma$-LQG harmonic balls.
		For example, on a uniform planar map show that the scaling limit of IDLA is a $\sqrt{8/3}$-LQG harmonic ball.
	\end{problem}
	
	Possible topologies of convergence in Problem~\ref{prob:planar-map} include a version of the Gromov-Hausdorff distance for metric spaces decorated by compact sets; or convergence of the IDLA clusters w.r.t.\ the Hausdorff distance when the random planar map is embedded into $\C$ appropriately.
	There are also some purely continuum directions one could pursue --- the following is an example.

	\begin{problem} \label{prob:boundary-dim}
		Compute the Hausdorff dimension of the boundary of an LQG harmonic ball, with respect to the Euclidean metric and with respect to the LQG metric.
	\end{problem}
	
	We expect that the Euclidean and LQG dimensions of the harmonic ball boundary are each strictly greater than one. We note that the Hausdorff dimensions of the boundary of an LQG metric ball with respect to the Euclidean and LQG metrics have been computed in \cite{gwynne-dimension-metric-ball, MR4408506}. 
	
	It is also of interest to determine the analogue of LQG metric geodesics in 
	the setting of harmonic balls. In particular, we are interested in extending the theory of `Hele-Shaw geodesics', in the sense of 
	\cite{hedenmalm2002hele}, to our setting. As mentioned previously, Hedenmalm-Shimorin in \cite{hedenmalm2002hele} investigated harmonic balls on smooth Riemannian surfaces
	and showed that their boundaries are piece-wise smooth curves. Because of this smoothness, they were able to define
	{\it Hele-Shaw geodesics} as a family of curves originating from a fixed point which are orthogonal to the boundary of a harmonic ball at any point. 
	One may think of these geodesics as describing the trajectory of a single fluid particle started at a fixed point on a Riemannian surface. 
	As LQG harmonic balls do not have smooth boundaries, it is unclear how to adapt this to our setting, but 
	a weaker version of these objects may exist. 
	\begin{problem}
		Construct and analyze the analogue of `Hele-Shaw geodesics' \cite{hedenmalm2002hele} on LQG surfaces. 
	\end{problem}

	A helpful intermediate step would be to show some additional regularity of harmonic balls. For example, the harmonic balls
	we construct are monotone in $t$ but we are unable to show {\it strict} monotonicity in $t$. 
	\begin{problem}
		Prove or disprove that the family of harmonic balls $\{\Lambda_t(0)\}_{t>0}$, given by Theorem \ref{theorem:harmonic-balls},  is strictly monotone in $t$, \ie, $\overline{\Lambda}_s(0) \subset \Lambda_t(0)$ whenever $s < t$. 
	\end{problem}

	\subsection{Paper and proof outline}

	We start in Section \ref{sec:gff-and-lqg} by reviewing the definition of the GFF, LQG, and some results about Liouville Brownian motion.  We then introduce the fundamental obstacle problem  which we use to construct candidate harmonic balls,  {\it clusters}, in Section \ref{sec:harmonic-ball-construction}
	and establish some basic properties of the constructed clusters in Section \ref{sec:basic-properties-cluster}.	Roughly, for each~$t > 0$, the cluster~$\Lambda_t := \Lambda_t(0)$ is defined as the support of the solution to an obstacle problem,~$v_t : B_1 \to \R$, specifically, 
	\[
	\Lambda_t = \{x \in B_1 : v_t > 0 \} \, \, \mbox{where} \, \,  v_t = \inf\{ w \in C(\overline{B_1}) :  \Delta w \leq \mu_h - t \delta_0 \mbox{ in $B_1$}  \mbox{ and }  w \geq  0 \mbox{ in $\overline{B_1}$}\} \, . 
	\]
	Since the obstacle problem is restricted to the unit ball, it is easy to show existence and uniqueness of solutions. In particular, 
	the results of Sections \ref{sec:harmonic-ball-construction} and \ref{sec:basic-properties-cluster} are straightforward extensions of those appearing in the obstacle problem literature --- the only property of the LQG measure which is used there is that it is a Radon measure with certain volume growth bounds, Lemma \ref{lemma:volume-growth} below. 
	
	While it is relatively easy to show the existence of clusters, it is not immediate that clusters are harmonic balls. As we will see in Lemma~\ref{lemma:harmonic-ball}
	below,~$\Lambda_t$ is a harmonic ball only if~$\overline{\Lambda_t} \subset B_1$ and~$\mu_h(\partial \Lambda_t) = 0$. It is not clear a priori that these hold for any~$t > 0$. These properties, namely, that the clusters do not attain a large Euclidean diameter in an arbitrarily small amount of time and that the boundaries of the clusters have zero LQG mass, do not follow from standard arguments and are not true for the obstacle problem with an arbitrary Radon measure.
	Proving this thus requires input from the theory of LQG and is thus the main goal of Sections~\ref{sec:non-degenerate} through~\ref{sec:boundary-measure-zero}.
	
	In Section \ref{sec:non-degenerate} we outline a strategy for showing that the clusters are harmonic balls; that is, they grow continuously and their boundaries have zero LQG mass.   Our approach for verifying these properties is completely new and relies on a novel {\it Harnack-type estimate}, Proposition \ref{prop:harnack-type-property},
	which clusters must satisfy. In particular, this estimate forces clusters to have `no thin-tentacles', as in \cite{jerison2013internal}. Roughly speaking, our Harnack-type estimate says that there is a constant $\alpha > 0$ such that if $A \subset \C$ is an annulus on which the LQG measure $\mu_h$ is reasonably well-behaved, then if the cluster $\Lambda_t$ crosses between the inner and outer boundaries of $A$ we must have $\mu_h(A \cap \Lambda_t) \geq \alpha \mu_h(A)$.
	Our proof of the Harnack-type estimate in Section \ref{sec:harnack-type-estimate} combines potential theoretic techniques 
	with methods from LQG theory. See the beginning of Section \ref{sec:harnack-type-estimate} for an outline of the argument.
	
	In Sections \ref{sec:upper-bound} and \ref{sec:boundary-measure-zero}, respectively, we use the Harnack-type estimate to prove that the clusters grow continuously in time, Proposition \ref{prop:continuity-of-clusters}, and that boundaries of the clusters have measure zero,
	Proposition \ref{prop:boundary-measure-zero}. As demonstrated in Theorem \ref{theorem:full-theorem-minus-uniqueness}, these properties are enough to ensure that 
	clusters are LQG harmonic balls. 
	
	Having constructed LQG harmonic balls, we show, by adapting ideas from the obstacle problem literature \cite{sakai2006quadrature}, their uniqueness,
	Proposition \ref{prop:uniqueness}, in Section \ref{sec:uniqueness}.
	
	The Harnack-type estimate imposes strong geometric constraints on LQG harmonic balls. For instance, 
	it disallows LQG harmonic balls from `crossing' annuli too many times, Lemma \ref{lemma:cannot-cross}. We use this in Section \ref{sec:boundary-curves} to show 
	that the boundaries of complementary connected components of LQG harmonic balls are Jordan curves, Proposition \ref{prop:boundary-simple-loop}. 
	
	These geometric constraints may also be translated into a strong relationship between LQG harmonic balls and the underlying LQG area measure,
	Lemma \ref{lemma:harmonic-comparison-chain}.
	Since the LQG measure is quite variable, this imposes an irregularity on LQG harmonic balls. We use this to show 
	that LQG harmonic balls do not satisfy the cone condition, Lemma \ref{lemma:no-exterior-cone-condition}, in Section \ref{subsec:not-lipschitz}.
	Consequently, LQG harmonic balls cannot be Lipschitz domains. 
	
	Another feature of the Harnack-type estimate is that it precludes LQG harmonic balls from having `approximate pinch points' which have small Euclidean diameter but which come close to disconnecting sets of large LQG mass from the origin within the cluster. On the other hand, an LQG metric ball has such approximate pinch points, as we show in Section \ref{subsec:not-lqg-metric-ball}.
	This shows that LQG harmonic balls are not LQG metric balls. 
	A key technical input in the proof is Proposition \ref{prop:mass-diam-gen}, which shows that a region in the plane can have small LQG diameter but large LQG mass simultaneously with positive probability.

	\subsection{Notation and conventions}
	\begin{itemize}
		\item Inequalities/equalities between functions/scalars are interpreted pointwise.
		\item Differential inequalities/equalities are interpreted in the distributional sense. 
		\item For a set $A \subset \C$, $\partial A$ denotes its topological boundary, $\overline{A} = A \cup \partial A$ its closure, 
		and $\mathrm{int}(A)$ its interior. 
		\item For two sets $A, B \subset \C$, say that $A \Subset B$ if $\overline{A} \subset B$. 
		\item $B_r(x)$ denotes the open ball of Euclidean radius $r > 0$ centered at $x \in \C$, when $x$ is omitted, the ball is centered at 0. 
		\item For a set $A \subset \C$, we denote the $r$-neighborhood of $A$ by $B_r(A) = A + B_r$.
		\item For $0 < r_1 < r_2$, denote an open annulus centered at $z$ by
		\begin{equation} \label{eq:annulus}
		\A_{r_1, r_2}(z) = B_{r_2}(z) \backslash \overline{B}_{r_1}(z)
		\end{equation}
		and $\A_{r_1, r_2} := \A_{r_1,r_2}(0)$. 
		\item Let $\{E^{r}\}_{r > 0}$ be a one-parameter family of events. 
		We say that $E^r$ occurs with polynomially high probability as $r \to 0$
		if there exists $p>0$ such that $\P[E^r] \geq 1 - O(r^p)$. 
		\item For two sets $A, B \subset \C$, $\dist(A,B) = \inf_{a \in A, b \in B} \dist(a,b)$ where $\dist$ denotes the Euclidean distance between two points. 
	\end{itemize}
	
	\subsection*{Acknowledgments}
	We thank the anonymous referees for helpful comments. We also thank Charlie Smart and Bill Cooperman for useful discussions. 
	A.B. was partially supported by NSF grant DMS-2202940 and a Stevanovich fellowship. 
	E.G. was partially supported by a Clay research fellowship.
	
	\section{Preliminaries}\label{sec:gff-and-lqg}

	In this section we review the definitions and basic properties of the Gaussian Free Field (GFF),
	the Liouville Quantum Gravity (LQG) area measure, and the LQG metric.  We present just enough exposition for the purposes of this paper; 
	the book \cite{berestycki2021gaussian} and surveys \cite{gwynne2020random,ding2021introduction,sheffield-icm} provide more details.
	
	\subsection{Gaussian free field}
	The {\it whole-plane} GFF $h^{\C}$ is the centered Gaussian random generalized 
	function on $\C$ with covariances
	\begin{equation} \label{eq:cov-kernel}
	\Cov(h^\C(z), h^\C(w)) := \log \frac{\max(|z|,1) \max(|w|,1)}{|z-w|} ,\quad\forall z,w\in\C .
	\end{equation} 
	The GFF $h^{\C}$ is not well-defined pointwise since the covariance kernel in~\eqref{eq:cov-kernel} diverges to $\infty$ as $z\to w$. 
	Nevertheless, for $z\in\C$ and $r>0$, one can define the average of $h^{\C}$ over the circle of radius $r$ centered at $z$, which we denote by $h^{\C}_r(z)$~\cite[Section 3.1]{duplantier2011liouville}. 
	
	The whole plane GFF is usually defined modulo additive constant. Our choice of covariance in~\eqref{eq:cov-kernel} corresponds to fixing this additive constant so that $h^\C_1(0) = 0$ (see, e.g.,~\cite[Section 2.1.1]{vargas2017lecture}).
	
	The law of the whole-plane GFF, viewed modulo additive constant, is invariant under complex affine transformations of $\C$. This translates into the following
	invariance property for $h^{\C}$,
	\begin{equation} \label{eq:gff-scaling}
	h^{\C} \overset{d}{=} h^{\C}(a\cdot +b) - h^{\C}_{|a|}(b), \quad \forall a \in \C \backslash \{0\}, \quad \forall b \in \C. 
	\end{equation}

	Fix $\gamma \in (0,2)$ and $\boldsymbol{\alpha}_0 \in (-\infty, Q)$, where
	\begin{equation} \label{eq:Q}
	Q := \frac{2}{\gamma} + \frac{\gamma}{2}.
	\end{equation}
	Throughout this paper we take $h$ to be the whole-plane GFF with an $\boldsymbol{\alpha}_0$ log singularity at the origin.

	Specifically, let $h^{\C}$ denote the whole-plane GFF normalized so that its circle average over the unit disk is
	zero and set 
	\begin{equation} \label{eq:gff}
	h = h^{\C} - \boldsymbol{\alpha}_0 \log |\cdot|.
	\end{equation}
	It is immediate from~\eqref{eq:gff-scaling} that
	\begin{equation} \label{eq:h-coordinate-change}
	h \overset{d}{=} h(a\cdot) - h_{|a|}, \quad \forall a \in \C \backslash \{0\}.
	\end{equation}

	\subsection{Liouville Quantum Gravity} \label{subsec:lqg}

	Let $\mu_h$ denote the {\it $\gamma$-LQG area (Liouville) measure} associated to $h$. One of the (many) possible ways of defining $\mu_h$ is as the a.s.\ weak limit
	\begin{equation} \label{eq:measure-lim}
	\mu_h = \lim_{\epsilon \rightarrow 0} \epsilon^{\gamma^2/2} e^{\gamma h_\epsilon(z)} \,dz
	\end{equation}
	where $dz$ denotes Lebesgue measure and $h_\epsilon(z)$ is the circle average~\cite{duplantier2011liouville, sheffield2016field}.
	In fact, the measure $\mu_{\tilde h}$ exists for any random generalized function $\tilde h$ of the form  $h + f$ where $f$ is a possibly random continuous function. 	
	\begin{fact}[LQG measure] \label{fact:lqg-measure}
		The LQG area measure $\mu_h$ satisfies the following properties.
		\begin{enumerate}[label=\Roman*.]
			\item {\bf Radon measure.} A.s., $\mu_h$ is a non-atomic Radon measure.
			\item {\bf Locality.} For every deterministic open set $U \subset \C$,
			$\mu_h(U)$ is given by a measurable function of 
			$h \vert_{U}$.
			\item {\bf Weyl scaling.} A.s., $e^{\gamma f} \cdot \mu_h = \mu_{h + f}$
			for every continuous function $f: \C \to \R$. 
			\item {\bf Conformal covariance.} A.s., the following is true. Let $U , \tilde U \subset \C$ be open and let $\phi$ be a conformal
			map from $\tilde U$ to $U$. Then, with $Q$ as in~\eqref{eq:Q}, 
			\begin{equation} \label{eq:measure-coord}
			\mu_{h \circ \phi + Q \log |\phi'|}(A) = \mu_h(\phi(A)) \quad \mbox{for all Borel measurable $A \subset \tilde U$}.
			\end{equation}
		\end{enumerate}
	\end{fact}
	
	The first three properties in Fact~\ref{fact:lqg-measure} are immediate from the definition~\eqref{eq:measure-lim}. The conformal covariance property was proven to hold a.s.\ for a fixed conformal map in~\cite[Proposition 2.1]{duplantier2011liouville} and extended to all conformal maps simultaneously in~\cite{sheffield2016field}.
	
	It was shown in \cite{ding2020tightness, gwynne2021existence} that one can define also the {\it LQG metric} $D_h$, 
	which is the limit of regularized versions of the Riemannian distance function associated with the Riemannian metric tensor $e^{\gamma h}(d x^2 + dy^2)$.
	Like the LQG measure, the LQG metric is a fractal-type object. It induces the same topology on $\C$ as the Euclidean metric, but the Hausdorff dimension of the metric space $(\C , D_h)$ is a.s.\ given by a deterministic number $d_\gamma > 2$~\cite[Corollary 1.7]{gp-kpz}. The value of $d_\gamma$ is not known explicitly except that $d_{\sqrt{8/3}}=4$~\cite{le2007topological}. 
	
	In order to state an analog of Fact~\ref{fact:lqg-measure} for the LQG metric, we make the following definitions. For a Euclidean-continuous path $P$ in $\C$, we write $\mathrm{len}(P; D_h)$ for its length with respect to $D_h$.  
	For an open set $U \subset \C$, the {\it internal metric of $D_h$ on $U$}
	is defined by 
	\begin{equation} \label{eq:internal-metric}
	D_h(z,w; U) = \inf \{ \mathrm{len}(P; D_h) : \mbox{$P$ is a path from $z$ to $w$ in $U$}\}, \quad \mbox{$\forall z,w, \in U$}.
	\end{equation}
	As in the case of the measure, the metric $D_{\tilde h}$ exists whenever $\tilde h = h + f$, where $f$ is a possibly random continuous function. 
	
	\begin{fact}[LQG metric] \label{fact:lqg-metric}
		The LQG metric $D_h$ has the following properties.
		\begin{enumerate}[label=\Roman*.]
			\item {\bf Euclidean topology and length metric.} A.s., $D_h$ induces the same topology on $\C$ as the Euclidean metric and is a length metric, that is, $D_h(z,w)$ is the infimum of the $D_h$-length of paths from $z$ to $w$.
			\item {\bf Locality.} For every deterministic open set $U \subset \C$, the $D_h$-internal metric on $U$ is given by a measurable function of 
			$h \vert_{U}$. 
			\item {\bf Weyl scaling.} Let 
			\begin{equation} \label{eq:xi}
			\xi = \frac{\gamma}{d_{\gamma}} ,
			\end{equation}
			where $d_\gamma$ is the Hausdorff dimension of the $\gamma$-LQG metric as above. 
			Almost surely, for every continuous function $f: \C \to \R$,
			\[
			D_{h + f}(u, v) = \inf_{P: u \to v} \int_0^{\mathrm{len}(P; D_h)} e^{\xi f(P(t))} d t, \quad \forall u,v \in \C.
			\]
			
			\item {\bf Coordinate change for scaling and translation.} Let $r > 0$ and $z \in C$. Almost surely, with $Q$ as in~\eqref{eq:Q},  
			\[
			D_h(r u + z, r v + z) = D_{h(r\cdot +z) + Q \log r}(u,v), \quad \forall u,v, \in C.
			\]
		\end{enumerate}
	\end{fact}
	
	The properties listed in Fact~\ref{fact:lqg-metric} were verified for the LQG metric in~\cite{ding2020tightness,dubedat2020weak,gwynne2021existence}. In fact, it is shown in~\cite{gwynne2021existence} that these properties uniquely characterize $D_h$.

	In what follows, for sets $A,B\subset \C$, we write
	\begin{equation} \label{eq:set-dist}
	D_h(A,B) = \inf_{x\in A, y\in B} D_h(x,y) .
	\end{equation} 
	For disjoint compact sets $K_1,K_2\subset \C$, a \emph{$D_h$-geodesic} from $K_1$ to $K_2$ is a path from $K_1$ to $K_2$ of minimal $D_h$-length. It is easily seen from the length metric property and a compactness argument that $D_h$-geodesics always exist (see, e.g.,~\cite[Corollary 2.5.20]{bbi-metric-geometry}).

	\subsection{Green's function} \label{subsec:green}
	Let $G_{O}:\overline{O} \times \overline{O} \to \R \cup \{\infty\}$ denote the Green's function for standard Brownian motion killed upon exiting a bounded open set $O \subset \C$. 
	We make use of the following standard properties of the Green's function of a (sufficiently nice) set. 
	\begin{prop} \label{prop:green-properties}
		The Green's function of a ball, $B_R$ of radius $R > 0$, has the following properties for every $x \in B_R$.
		\begin{itemize}
			\item Fundamental solution: $\Delta G_{B_R}(x,\cdot) = -\delta_x(\cdot)$ on $B_R$.
			\item Positive: $G_{B_R}(x, \cdot) > 0$ on $B_R$.
			\item Zero boundary: $G_{B_R}(x, \cdot) = 0$ on $\partial B_R$.
			\item Smooth away from the pole: $G_{B_R}(x, \cdot)$ is infinitely differentiable away from $x$. 
		\end{itemize}
	\end{prop}

	\subsection{Liouville Potential Theory} \label{subsec:lqg-appendix}
	
	In this section we collect well known potential theoretic estimates on the LQG measure. 
	We first note bounds on the LQG mass of annuli and balls.

	\begin{lemma} \label{lemma:volume-growth}
		For each $\beta^+ \in (0, (2 -\gamma)^2/2)$ and $\beta^- > (2+\gamma)^2/2$
		it holds with polynomially high probability as $\epsilon \to 0$ that  
		\begin{equation} \label{eq:ball-volume}
		\epsilon^{\beta^-} \leq \mu_h(B_\epsilon(z)) \leq \epsilon^{\beta^+} ,\quad\forall z \in B_1 .
		\end{equation}
		Furthermore, for each $0 < r_1 < r_2$ there exists constants $C_1, C_2$ so that for all $z \in B_1$
		\begin{equation} \label{eq:ann-volume}
		C_2 \epsilon^{\beta^-} \leq \mu_h(\A_{r_1 \epsilon, r_2 \epsilon}(z)) \leq C_1 \epsilon^{\beta^+}
		\end{equation}
		with polynomially high probability as $\epsilon \to 0$, where here we use the notation for annuli from~\eqref{eq:annulus}.
	\end{lemma}
	\begin{proof}
		Exactly the same argument as in \cite[Lemma A.1]{berestycki2020random} shows that~\eqref{eq:ball-volume} holds with polynomially high probability as $\epsilon \to 0$. The estimate~\eqref{eq:ann-volume} follows from~\eqref{eq:ball-volume} and the fact that
		\[
		B_{\frac{r_2-r_1}{4} \epsilon}(z + \frac{r_2\epsilon+r_1\epsilon}{2}  e_1) \subset \A_{r_1 \epsilon, r_2 \epsilon}(z) \subset B_{r_2 \epsilon}(z)
		\]
		where $e_1 = (1,0)$. 
	\end{proof}

	\textit{Liouville Brownian motion} (LBM) is the natural diffusion associated with $\gamma$-LQG. Roughly speaking, LBM is obtained from ordinary Brownian motion (sampled independently from $h$) by changing time so that the process has `constant $\gamma$-LQG speed'. LBM was constructed in \cite{berestycki2015diffusion, garban2016liouville}. It was shown in \cite{berestycki2020random} to describe the scaling limit of random walk on a certain family of random planar maps. 
	
	The volume growth bounds given by Lemma \ref{lemma:volume-growth} lead to control 
	on the expected exit time of LBM from balls. 
	
	\begin{prop} \label{prop:lbm-exit-time}
		Let $O$ denote a smooth bounded open set. 
		The expected exit time of LBM from $O$ started at $x \in O$
		is finite and H\"{o}lder continuous in $x$.  More generally, any $q$ of the form,
		\[
		q(x) = \int_{O} G_{O}(x, y) f (y) d \mu_h(y) \qquad \mbox{for some $f \in L^{\infty}(\overline O)$}
		\]
		is finite and H\"{o}lder continuous in $\overline{O}$.
	\end{prop}
	\begin{proof}		
		Finiteness follows immediately from Lemma \ref{lemma:volume-growth}.
		H\"{o}lder continuity uses the embedding of Campanato spaces into H\"{o}lder spaces together with Lemma \ref{lemma:volume-growth}. 
		See, for example, Section 16.2 (or the remark after Proposition 13.5) in \cite{ponce2016elliptic}. 
	\end{proof}
	
	The bounds also lead to continuity of the LBM heat kernel using the main result of \cite{kigami2019}. 
	Continuity of the LBM heat kernel (for other versions of the GFF) was previously established by \cite{andres2016continuity}	and \cite{maillard2016liouville}.

	Let $K$ be a square in $\C$ and for $x \in \C$ let $\{ \mathcal{B}_t^x\}_{t > 0}$ denote $\gamma$-LBM with respect to the field $h$ started from $x$ with Neumann (reflecting) boundary conditions on $K$. The heat kernel $p^K_t$ of reflected LBM in $K$ is the 
	function $p^K_t(x,y) : (0,\infty) \times K \times K \to [0,\infty)$ such that 
	\begin{equation} \label{eq:defining-property-of-kernel}
	\P[ \mathcal{B}_t^x \in dy | h] = p^K_t(x,y) dy.
	\end{equation}

	\begin{prop} \label{prop:lbm-continuity}
		Let $K$ be a square in $\C$. Almost surely, the heat kernel $p^K_t(x,y)$ associated to $\gamma$-LBM with Neumann boundary conditions on $K$ exists, is finite, jointly continuous, and strictly positive for all $(t,x,y) \in (0,\infty) \times K \times K$. 
	\end{prop}
	\begin{proof}
		This is \cite[Theorem 13.1]{kigami2019}   with input given by Lemma \ref{lemma:volume-growth}.
		Strictly speaking, \cite[Theorem 13.1]{kigami2019}   concerns the transition density of reflecting $\gamma$-LBM in the unit square. A scaling argument shows that \cite[Theorem 13.1]{kigami2019} applies
		to the transition density of reflecting $\gamma$-LBM in any fixed square.
	\end{proof}

	\section{Construction of candidate harmonic balls via Hele-Shaw flow} \label{sec:harmonic-ball-construction}

	In this section we construct domains which we will later show are $\gamma$-LQG harmonic balls. Specifically, we construct a family of sets $\{{\cluster{}{t}(z)}\}_{t > 0}$ via an obstacle problem involving the Green's function for the ball. This family of sets models the flow of a Newtonian fluid injected at a constant rate into an LQG surface, restricted to a ball on the surface. The movement of this fluid is called~\emph{Hele-Shaw flow}. As exposited in Chapter 3 of \cite{gustafsson2006conformal}, 
	one way of defining Hele-Shaw flow mathematically is via the obstacle problem construction below.  The construction itself is fairly standard see, \eg, \cite{gustafsson1990onquadraturedomains, hedenmalm2002hele,shahgholian2013harmonic} and originates from the work of Sakai \cite{sakai1984solutions}. 
	
	While the construction is standard, since the obstacle problem is restricted to a ball, it is not obvious that the construction gives $\gamma$-LQG harmonic balls. We will later show, using LQG specific arguments, the existence of $T > 0$ so that $\{{\cluster{}{t}}\}_{0 < t < T}$ are a family of harmonic balls satisfying the conditions in Theorem 
	\ref{theorem:harmonic-balls}. We then use scale invariance and compatibility to extend this construction to all $t > 0$.

	\subsection{Definition of the obstacle problem} \label{subsec:obstacle-def} 
	We construct candidate harmonic balls via a technique similar to the Perron method involving the measure $\mu_h$ and the Green's function for the ball.
	For each $t, r > 0$ and~$z \in B_r$,  the set of {\it supersolutions}
	is 
	\begin{equation} \label{eq:super-solutions}
	\mclS{B_r; z}{t}  = \{ w \in C(\overline{B_r}) :  \Delta w \leq \mu_h \mbox{ in $B_r$}  \mbox{ and }  w \geq  -t G_{B_r}(z, \cdot) \mbox{ in $\overline{B_r}$}\},
	\end{equation}
	where $C(\overline{B_r})$ denotes the set of continuous functions on the closed ball.
	The {\it least supersolution} is defined as the pointwise infimum of all functions in $\mclS{B_r}{t}$, 
	\begin{equation} \label{eq:least-super-solution}
	\lss{B_r; z}{t} = \inf\{ w \in \mclS{B_r;z }{t} \}
	\end{equation}
	and the {\it cluster} as
	\begin{equation} \label{eq:non-co-set}
	\cluster{B_r; z}{t} = \{ x \in B_r : \lss{B_r; z}{t}(x) > -t G_{B_r; z}(z, x) \}.
	\end{equation}
	We also consider the {\it odometer}	
	\begin{equation} \label{eq:limit-odometer}
	\odometer{B_r; z}{t} = \lss{B_r; z}{t}  + t G_{B_r; z}(0, \cdot).
	\end{equation}
	Note that $\mclS{B_r; z}{t}$ is non-empty as it contains the zero function --- thus $\lss{B_r; z}{t}$ always exists.
	This equation in \eqref{eq:super-solutions} is known as an {\it obstacle problem} with obstacle given by the Green's function.
	When $B_r$ is the unit ball, we write, for example,~$\mclS{z}{t}$, and if additionally~$z = 0$, we write, for example,~$\mclS{}{t}$.

	We think of the above obstacle problem as modeling the flow of liquid on a rough surface. 
	A mass $t$ of fluid is injected at~$\{z\}$ and its growth is dictated by the infinitesimal capacity of the surface, namely, the measure $\mu_h$. The cluster ${\cluster{z}{t}}$ represents the settled fluid and $\odometer{z}{t}$ captures the `work' needed to spread the fluid. Specifically, the family of sets~$\{{\cluster{z}{t}}\}_{t > 0}$ is a weak solution 
	to a~\emph{restricted Hele-Shaw problem} involving the measure~$\mu_h$. The Hele-Shaw problem is restricted because 
	of the fact that~\eqref{eq:super-solutions} is only defined in the ball~$B_r$. Physically what this means is that the flow is stopped upon exiting~$B_r$. The obstacle problem~\eqref{eq:super-solutions}
	is a variational formulation of this restricted Hele-Shaw problem.  See~\cite[Section 3.5]{gustafsson2006conformal} and~\cite{roos2016partial} for an explicit description of the Hele-Shaw equation and how the obstacle problem relates to it. 
	
	From the physical picture described in the previous paragraph, one expects that if $t$ is larger than $\mu_h(B_r)$, then ${\cluster{}{t}}$ should fill the entire ball. Moreover, if $\mu_h$ is regular enough, then for $t$ small the clusters should be strictly contained in $B_r$. Further, clusters with closures which do not intersect the boundary of~$B_r$ should be compatible with clusters restricted to~$B_{r'}$ for~$r' < r$. We provide rigorous statements of these heuristics below.

	\subsection{Basic properties of the obstacle problem} \label{subsec:basic-properties}
	We assert existence and basic regularity of solutions to the obstacle problem. These results are standard but for completeness
	are proved in Appendix \ref{sec:obstacle-appendix}.

	We first note that the least supersolution is indeed a supersolution. 
	\begin{lemma}\label{lemma:basic-properties}
		On an event of probability 1, for all $t, r > 0$ and~$z \in B_r$, $\lss{B_r;z}{t}$ is finite, continuous, and an element of $\mclS{B_r;z}{t}$. 
	\end{lemma}

	The next lemma is a consequence of being the least supersolution. 
	\begin{lemma} \label{lemma:non-coincidence-open-harmonic}
		On an event of probability 1, for all $t,r > 0$, and~$z \in B_r$,  the cluster $\cluster{B_r;z}{t}$ is open and connected and
		\[
		\Delta \lss{B_r;z}{t} = \mu_h \vert_{\cluster{B_r;z}{t}} + \nu \vert_{\partial {\cluster{B_r;z}{t}}} \quad \mbox{on $B_r$},
		\]
		where $\nu$ is a Radon measure which is absolutely continuous with respect to $\mu_h$ on~$B_r$ and satisfies $0 \leq \nu \leq \mu_h$ on~$B_r$.
		In particular, 
		\[
		\Delta 	\odometer{B_r;z}{t} = -t \delta_z + \mu_h \vert_{{\cluster{B_r;z}{t}}} + \nu \vert_{\partial {\cluster{B_r;z}{t}}} \quad \mbox{on $B_r$}.
		\]
	\end{lemma}
	We will eventually show that on an event of probability one, $\mu_h(\partial {\cluster{B_r;z}{t}}) = 0$ for all~$t, r >0$ and~$z \in B_r$, which implies that $\nu = 0$. However, for the time being we need to allow for the possibility that there is some mass on $\partial {\cluster{B_r;z}{t}}$.

	We also have monotonicity of the clusters in $t$. 
	\begin{lemma} \label{lemma:monotonicity}
		On an event of probability 1, for all $t_1 \leq t_2$ and~$r > 0$,~$z \in B_r$, we have $\cluster{B_r}{t_1} \subseteq \cluster{B_r}{t_2}$.
	\end{lemma}
	
	Clusters also have a conservation of mass property. 
	\begin{lemma} \label{lemma:conservation-of-mass}
		On an event of probability 1, for all $r,t > 0$ and~$z \in B_r$, we have $\mu_h({\cluster{B_r;z}{t}}) \leq t$ and $\odometer{B_r;z}{t} = 0$ on $\partial B_r$.  Moreover, if $\overline{{{\cluster{B_r;z}{t}}}} \subset B_r$ and $\mu(\partial {\cluster{B_r;z}{t}}) = 0$, then $\mu_h({\cluster{B_r;z}{t}}) = t$. 
	\end{lemma}

	We conclude with a compatibility result for clusters across different domains. 
	\begin{lemma} \label{lemma:cluster-compatibility}
		The following holds for each $R > 0$ on an event of probability 1. 
		For all $s_1 \leq R$ and~$z \in B_{s_1}$ if, for some $s_2 \in [s_1, R]$, we have $\overline{\cluster{B_{s_2};z}{t}} \subset B_{s_1}$, 
		then  $\cluster{B_{s};z}{t} = \cluster{B_{s_2};z}{t}$ for all $s \in [s_1, R]$.
	\end{lemma}

	\section{Basic properties of clusters} \label{sec:basic-properties-cluster}
	
	In this section we note some basic properties of the clusters $\{{\cluster{}{t}}\}_{t > 0}$  and 
	odometers $\{\odometer{}{t}\}_{t > 0}$. As in Section \ref{sec:harmonic-ball-construction}, these results are fairly standard, \eg, \cite{sakai1984solutions, gustafsson1990onquadraturedomains, hedenmalm2002hele,shahgholian2013harmonic}, but (short) proofs are included for completeness.

	\subsection{Lower bound}
	We first show that each cluster contains a Euclidean ball of sufficiently small radius and eventually the family 
	coincides with the unit ball. 
	\begin{prop} \label{prop:lower-bound}
		On an event of probability 1, for each $t > 0$ and~$z \in B_1$,  there exists a random $\epsilon = \epsilon(t) > 0$ so that 
		\begin{equation}
		B_{\epsilon}(z) \subset {\cluster{z}{t}}.
		\end{equation}
		Moreover, for each $\delta \in (|z|,1)$, there exists a random $t(\delta)>0$ such that for all $t \geq t(\delta)$,
		\begin{equation}
		B_{1-\delta}(z) \subset  {\cluster{z}{t}} 
		\end{equation}
		and there exists a random $t^+ > 0$ so that for all $t \geq t^+$
		\begin{equation}
		\overline{{{\cluster{z}{t}}}} \cap \partial B_1 \neq \emptyset.
		\end{equation}
	\end{prop}
	
	Our proof uses the fact that the logarithm function blows up near the origin together with 
	the finiteness of the expected exit time of Liouville Brownian Motion from the unit ball. 
	
	\begin{proof}[Proof of Proposition \ref{prop:lower-bound}]
		
		Let $q_1(y)$ denote the expected exit time of Liouville Brownian Motion started at a point~$y$ from the unit ball, 		
		\[
		\begin{cases}
		\Delta q_1 = -\mu_h \quad &\mbox{in $B_1$} \\
		q_1 = 0 \quad &\mbox{on $\partial B_1$}.
		\end{cases}
		\]
		Let $t > 0$ and~$z \in B_1$ be given. 
		As $\Delta \lss{z}{t} \leq \mu_h$ (Lemma~\ref{lemma:basic-properties}), the function $\lss{z}{t} + q_{1}$ is superharmonic in $B_1$. Hence, as $\lss{z}{t} + q_1 \geq 0$ on $\partial B_1$, we have $\lss{z}{t} \geq -q_{1} $ 
		in $\overline{B_1}$. Since
		\[
		\lim_{\epsilon \to 0} \sup_{x \in B_{\epsilon}(z)} t \log |x-z| \to -\infty,
		\]
		we have that
		\begin{equation} \label{eq:lower-bound-green}
		\lss{z}{t}(x) \geq  -q_{1}(x) > t \log |x-z|, \quad \forall x \in B_{\epsilon}(z), \quad \forall \epsilon > 0 \mbox{ sufficiently small}.
		\end{equation}
		Indeed, by Proposition \ref{prop:lbm-exit-time}, $q_{1}$ is finite. By the definition~\eqref{eq:non-co-set} of ${\cluster{z}{t}}$ and the fact that $G_{B_1}(z,x) = O(-\log|z-x|)$, this shows that $B_{\epsilon}(z) \subset {\cluster{z}{t}}$. 	Similarly, for each $\delta \in (|z|,1)$, for all $t > t(\delta)$, \eqref{eq:lower-bound-green}
		is satisfied for all $x \in B_{1-\delta}(z)$. 	The last assertion follows by choosing $t^+ = \mu_h(B_1)$
		and using Lemma \ref{lemma:conservation-of-mass}.
	\end{proof}

	\subsection{H\"{o}lder continuity of the odometer}
	We observe that $\lss{}{t}$ is H\"{o}lder for a deterministic exponent depending only on $\gamma$.

	\begin{lemma} \label{odometer:cont}
		There exists a deterministic exponent $\alpha = \alpha(\gamma)$ so that on an event of probability 1, there
		exists a constant $C > 0$
		\[
		|\lss{z}{t}(x) - \lss{z}{t}(y)| \leq C |x-y|^{\alpha}
		\]
		for all~$t > 0$,~$z \in B_1$, and $x,y \in B_1$. 
	\end{lemma}
	
	\begin{proof}
		By Lemma \ref{lemma:non-coincidence-open-harmonic}, $\Delta \lss{z}{t} = \mu_h |_{{{\cluster{z}{t}}}} + \nu |_{\partial {\cluster{z}{t}}}$
		and $\nu$ is absolutely continuous with respect to $\mu_h$ in~$B_1$. Thus, the claim follows by Proposition  \ref{prop:lbm-exit-time}.
	\end{proof}
	
	\subsection{Non-degenerate clusters are subharmonic balls} \label{subsec:subharmonic}
	We prove that clusters which do not intersect the boundary of $B_1$ are harmonic balls. 
	In fact, we observe a stronger property --- each ${\cluster{z}{t}}$ strictly contained in $B_1$ 
	is a subharmonic ball. That is, subharmonic functions satisfy the sub-mean-value property 
	on such ${\cluster{z}{t}}$.
	
	Specifically, for~$z \in \C$, an open set $\Lambda(z)$ is a {\it subharmonic ball} centered at~$z \in \C$  with respect to a Radon measure $\mu$ if 
	\begin{equation} \label{eq:subharmonic-ball}
	\mu(\Lambda(z)) f(z)  \leq \int_{\Lambda(z)} f(x) \mu(d x)
	\end{equation}
	for all functions $f: \overline{O} \to \R$ of the form
	\begin{equation} \label{eq:sub-harmonic-potential}
	f(x) = \int_{O} G_{O}(x,y) d \nu(y) + q(x)
	\end{equation}
	where $O$ is an open set containing a neighborhood of the closure of $\Lambda(z)$,
	$\nu$ is a signed Radon measure, $\nu |_{\Lambda(z)} \leq 0$, with compact support in $O$,
	and $q:\overline{O} \to \R$ is a harmonic function on $O$.
	
	We note that every subharmonic ball is a harmonic ball in the sense described just above Theorem~\ref{theorem:harmonic-balls}. Indeed, the set of harmonic functions in the definition of a harmonic ball (as described above Theorem~\ref{theorem:harmonic-balls}) is the same as the set of functions $f$ of the form~\eqref{eq:sub-harmonic-potential} with $\nu|_{\Lambda(z)}=0$. Since this set of functions is closed under replacing $f$ with $-f$, the inequality~\eqref{eq:subharmonic-ball} gives both the sub-mean-value property and the super-mean-value property for functions in this set.
	
	\begin{lemma} \label{lemma:harmonic-ball}
		On an event of probability 1, for all $r,t > 0$ and~$z \in B_r$, if  $\overline{{{\cluster{B_r;z}{t}}}} \subset B_r$ and $\mu_h(\partial {\cluster{B_r;z}{t}}) = 0$, 
		then ${\cluster{B_r;z}{t}}$ is a subharmonic ball. 
	\end{lemma}
	
	We do not know a priori that the hypotheses $\overline{{{\cluster{B_r;z}{t}}}} \subset B_r$ and $\mu_h(\partial{\cluster{B_r;z}{t}}) = 0$ are satisfied for any value of $t, r>0$ with~$z \in B_r$. We will prove that these hypotheses are satisfied, at least when $t$ depending on~$r, z$ is small, in Sections~\ref{sec:upper-bound} and~\ref{sec:boundary-measure-zero}, respectively. In fact, we will show some uniformity in~$z$ of how small~$t$ needs to be.
	By the Riesz decomposition theorem, see, \eg, Section 4 in \cite{armitage2000classical}, functions of the form \eqref{eq:sub-harmonic-potential}
	include functions which are subharmonic in a neighborhood of $\overline{O}$.
	
	\begin{proof}[Proof of Lemma~\ref{lemma:harmonic-ball}]
		Let $r,t > 0$,~$z \in B_r$, an open set $O \supset \overline{{{\cluster{B_r; z}{t}}}}$, and $f, q, \nu$ as in \eqref{eq:sub-harmonic-potential} be given. By Lemma \ref{lemma:conservation-of-mass}, Lemma \ref{lemma:non-coincidence-open-harmonic}, and Proposition \ref{prop:green-properties} and our assumption that $\mu_h(\partial {\cluster{B_r;z}{t}}) = 0$, ${\cluster{B_r;z}{t}}$ is open and 
		\begin{equation} \label{eq:v-characterization}
		\begin{cases}
		\odometer{B_r;z}{t} = 0  &\mbox{on $\partial B_r$} \\
		\Delta \odometer{B_r;z}{t} = -t \delta_z + \mu|_{{{\cluster{B_r;z}{t}}}} &\mbox{on $B_r$}.
		\end{cases}
		\end{equation}
		and $\mu_h({\cluster{B_r;z}{t}}) = t$. As $\overline{{{\cluster{B_r;z}{t}}}} \subset O$, we can find a smooth domain $\Lambda'_t$ with ${\cluster{B_r;z}{t}} \subset \Lambda'_t \subset O$
		so that
		\begin{align*}
		0 &= \int_{{{\cluster{B_r;z}{t}}}} \Delta q(x) \odometer{B_r;z}{t}(x) dx \\
		&= \int_{\Lambda_t'} \Delta q(z) \odometer{B_r;z}{t}(x) dx  \qquad \mbox{(since $\odometer{B_r;z}{t} = 0$ on $B_r \backslash {\cluster{B_r;z }{t}}$)} \\
		&= \int_{\Lambda_t'} q(x) \Delta \odometer{B_r;z}{t}(x)dx  \qquad \mbox{(integration by parts)} \\ 
		&= -t q(z) + \int_{{{\cluster{B_r;z}{t}}}} q(x) \mu_h(dx) \qquad \mbox{(by \eqref{eq:v-characterization}) } .
		\end{align*}
		Moreover, 
		\begin{align*}
		&t (f-q)(z) - \int_{{{\cluster{B_r;z}{t}}}}(f-q)(y) d \mu_h(y)  \\
		&=\int_{{{\cluster{B_r;z}{t}}}} \int_{{{\cluster{B_r;z}{t}}}} (G_{O}(z,x) - G_{O}(y,x)) d \nu(x) d \mu_h(y)  \qquad \mbox{(definition of $f$)} \\
		&= \int_{{{\cluster{B_r;z}{t}}}} \int_{{{\cluster{B_r;z}{t}}}} (G_{O}(z,x) - G_{O}(y,x)) d \mu_h(y) d \nu(x) \qquad \mbox{(by Fubini)} \\
		&= \int_{{{\cluster{B_r;z}{t}}}} \odometer{B_r;z}{t}(z) d \nu(x)  \qquad \mbox{(by \eqref{eq:v-characterization}) } \\
		&\leq 0 \qquad \mbox{(since $\odometer{B_r;z}{t} \geq 0$ and $\nu \leq 0$)}.
		\end{align*}		
		We conclude by combining the above two expressions.
	\end{proof}

	\section{Non-degeneracy of the flow} \label{sec:non-degenerate}

	In this section, we set up the proof of our main result Theorem \ref{theorem:harmonic-balls}
	by dividing it into several intermediate results which will be proven in Sections~\ref{sec:harnack-type-estimate} through~\ref{sec:uniqueness}. We then show how these intermediate results imply the claim. 
	In the last subsection we observe that the clusters which we construct are locally determined in the sense of Proposition \ref{prop:locally-determined}. 
	
	\subsection{Properties of the restricted flow}
	
	We first show in Section \ref{sec:upper-bound} that clusters do not immediately exit the unit ball.
	\begin{prop} \label{prop:upper-bound}
		On an event of probability 1, there exists a (random) $T = T(\gamma, h) > 0$ so that for each~$z \in B_{1/2}$ 
		and all~$0 < t < T$
		\begin{equation}
		\overline{{{\cluster{z}{t}}}} \subset B_1.
		\end{equation}
	\end{prop}
	
	We show in the second part of Section \ref{sec:upper-bound} that the family is continuous. 
	\begin{prop} \label{prop:continuity}
		On an event of probability 1, for all~$z \in B_{1/2}$ the cluster centered at~$z$ decreases to~$\{z\}$,  
		\begin{equation} \label{eq:contains-origin}
		\bigcap_{t > 0} \overline{{{\cluster{z}{t}}}} = \{z\}
		\end{equation} 
		and continuously increase in $t$:  for each $t > 0$, for all $\epsilon > 0$ sufficiently small, there exists $\delta(z) > 0$ so that for all 
		$t' \in [t, t + \delta(z)]$, 
		\begin{equation} \label{eq:continuity}
		\overline{{\cluster{z}{t'}}} \subset {\cluster{z}{t}} + B_{\epsilon}(z).
		\end{equation}
	\end{prop}
	In Section \ref{sec:boundary-measure-zero}, we show that each cluster has zero boundary area measure.   
	\begin{prop} \label{prop:boundary-measure}
		On an event of probability 1,  for all~$z \in B_1$ and~$t > 0$ such that $\overline{{{\cluster{z}{t}}}} \subset B_r(z)$. 
		\begin{equation}
		\mu_h (\partial {\cluster{z}{t}}) = 0 \, . 
		\end{equation}
	\end{prop}

	In order to ensure exact uniqueness, the family of harmonic balls appearing in our final theorem differ from the above clusters
	via a set of $\mu_h$-measure zero, 
	\begin{equation} \label{eq:regular-cluster}
	{\regcluster{z}{t}} := \mathrm{int}(\overline{{{\cluster{z}{t}}}}), \quad \forall t >0 \, , z \in B_1.
	\end{equation}
	Indeed, by definition, ${\cluster{}{t}} \subset \mathrm{int}(\overline{{{\cluster{}{t}}}})$ and as $\mu_h(\partial {\cluster{}{t}}) = 0$, $\mu_h(\mathrm{int}(\overline{{{\cluster{}{t}}}}) \backslash {\cluster{}{t}}) = 0$. This shows that ${\regcluster{}{t}}$ is a subharmonic ball and Propositions \ref{prop:upper-bound} through \ref{prop:boundary-measure} 
	hold with ${\regcluster{}{t}}$ in place of ${\cluster{}{t}}$. Thus, we may combine Proposition \ref{prop:lower-bound}, Lemma \ref{lemma:harmonic-ball}, and Propositions \ref{prop:upper-bound} through \ref{prop:boundary-measure} into the following statement.

	\begin{prop} \label{prop:conditional}
		On an event of probability 1, for all~$r > 0 $ and~$z \in B_{1/2}$, there exists a family of clusters $\{{\regcluster{z}{t}}\}_{0 < t < T}$ strictly contained in $\overline{B_1}$
		satisfying the conditions of Theorem \ref{theorem:harmonic-balls} for $0 < t<T$ and $\overline{{\regcluster{z}{T}}} \cap \partial B_1 \neq \emptyset$.
	\end{prop}

	In the next two subsections we use the compatibility property Lemma \ref{lemma:cluster-compatibility} together with a certain scale invariance of clusters
	to extend the construction in Proposition \ref{prop:conditional} to the entire plane. 
	\begin{theorem} \label{theorem:full-theorem-minus-uniqueness}
		On an event of probability 1, for all~$z \in \C$, there exists a family of clusters $\{{\cluster{}{t}}(z)\}_{t > 0}$ 
		satisfying the conditions of Theorem \ref{theorem:harmonic-balls}. 
		Moreover, each cluster is a subharmonic ball related to the clusters of \eqref{eq:non-co-set} in the following way: 
		if for some~$s > 0, z \in \C$, if ${\regcluster{}{t}(z)} \Subset B_s$, then ${\regcluster{}{t}(z)} = \mathrm{int}(\overline{\cluster{B_r;z}{t}})$ for all $r \geq s$.
	\end{theorem}
	
	In Section \ref{sec:uniqueness}, we prove that the family given by Theorem \ref{theorem:full-theorem-minus-uniqueness} is the unique such family,
	completing the proof of Theorem \ref{theorem:harmonic-balls}. 
	\begin{prop} \label{prop:uniqueness}
		Let $\{{\regcluster{}{t}(z)}\}_{t > 0, z \in \C}$ be given by Theorem \ref{theorem:full-theorem-minus-uniqueness}. 
		On an event of probability 1, if for some~$z \in \C$,  $\{A_t(z)\}_{t > 0}$ is a family 
		of harmonic balls satisfying the assumptions in Theorem \ref{theorem:harmonic-balls}, 
		then $A_t(z)={\regcluster{}{t}(z)}$ for all $t> 0$.
	\end{prop}

	\subsection{Scale invariance}

	We now give the relevant scale invariance property which we then use to prove Theorem \ref{theorem:full-theorem-minus-uniqueness}. 
	Specifically, we show that the law of a cluster stopped upon exiting a ball of arbitrary radius coincides with the law of a rescaled cluster which is stopped upon exiting the unit ball. 
	\begin{lemma} \label{lemma:scale-invariance-flow}
		For each $k > 0$ and~$z \in B_1$, the laws of $\{ \cluster{z}{t}\}_{t > 0}$ and $\{ k^{-1} \cluster{B_{k}; k z}{A_k t}\}_{t > 0}$ coincide, where  
		$A_k := e^{\gamma (Q \log k + h_k(0))}$.
	\end{lemma}
	
	\begin{proof}
		Let $t, k > 0$ and~$z \in B_1$ be given.
		Consider the continuous function $\tilde{w}_t \in C(\overline{B_1})$ defined by 
		\begin{equation} \label{eq:definition-of-tildewt}
		\tilde{w}_{t} := A_k^{-1} \lss{B_{k}; k z}{A_k t}(k\cdot) \qquad \mbox{on $\overline{B_1}$}
		\end{equation}
		and the fields
		\begin{equation}
		\tilde{h} = h(k \cdot) + Q \log k, \quad \mbox{where} \quad Q = \frac{2}{\gamma} + \frac{\gamma}{2}.
		\end{equation}
		and
		\begin{equation}
		h' = h(k \cdot) - h_k(0)
		\end{equation}
		where $h_k(0)$ is the average of $h$ on the circle of radius $k$ around 0. By~\eqref{eq:h-coordinate-change}, we have $h' \overset{d}{=} h$. 
		
		We claim that
		\begin{equation}
		\tilde{w}_{t} = \tilde{u}_t := \inf\{ w \in \mathcal{S}^{h'}_t\}
		\end{equation}
		where $\tilde{u}_t$ is the pointwise infimum over the family
		\begin{equation} \label{eq:super-solution-other-h}
		\mathcal{S}^{h'}_t  = \{ w \in C(\overline{B_1}) :  \Delta w \leq \mu_{h'} \mbox{ in $B_1$}  \mbox{ and }  w \geq  -t G_{B_1}(z, \cdot) \mbox{ in $\overline{B_1}$}\},
		\end{equation}
		defined in the same manner as $\mclS{}{}$ but with $h'$ instead of $h$. Write $\mathcal{S}^h_{t}(B_{k})$ when $B_1$ in \eqref{eq:super-solution-other-h} is replaced by $B_{k}$ and $h'$ by $h$. 
		
		We first show that $\tilde{w}_t \in \mathcal{S}^{h'}_t$. By the fact that $\Delta \lss{}{t} \leq \mu_h$ (Lemma~\ref{lemma:basic-properties}), the LQG coordinate change formula, and Weyl scaling (Fact \ref{fact:lqg-measure}), 
		\begin{equation}
		\Delta \tilde{w}_{t} \leq A_k^{-1} \mu_h(k\cdot) = \mu_{h'}  \quad  \mbox{on $\overline{B_1}$}.
		\end{equation}
		Also, since $\lss{B_{k};k z}{A_k t} \in \mathcal{S}^h_{A_k t}(B_{k})$ and $G_{B_{k}}(k z,  k x) = G_{B_{1}}(z, x)$, 
		\begin{equation}
		\tilde{w}_{t}(x) \geq -t A^{-1}_{k} A_k G_{B_{r k}(z)}(k z, k x) = - t G_{B_1}(z, x) \quad  \mbox{on $\overline{B_1}$}.
		\end{equation}
		Hence, $\tilde{w}_t \in \mathcal{S}^{h'}_t$.
		Similarly, $A_k \tilde{u}_t(\cdot/k) \in \mathcal{S}^h_{A_k t}(B_k)$, which shows 
		$\tilde{u}_t = \tilde{w}_t$.
		Indeed, $A_k \tilde{u}_t(\cdot/k) \leq \lss{B_{k}}{A_k t}(\cdot)$ implies $\tilde{u}_t(\cdot/k) \leq A_k^{-1} \lss{B_{k}}{A_k t}(\cdot) = \tilde{w}_{t}(\cdot/k)$ by \eqref{eq:definition-of-tildewt}.
		Hence, as $h'$ has the same law as $h$,
		\begin{equation}
		\cluster{z}{t} \overset{d}{=} \{ x \in B_1 : \tilde{u}_t(x) > -t G_{B_1}(z, x) \} = k^{-1} \cluster{B_{k }; z k}{A_k t};
		\end{equation}
		the last equality uses $\tilde{u}_t = \tilde{w}_t$.
	\end{proof}
	
	%

	\subsection{Proof of Theorem \ref{theorem:full-theorem-minus-uniqueness} assuming Proposition \ref{prop:conditional}}
	By combining Lemma \ref{lemma:scale-invariance-flow} together with Proposition \ref{prop:conditional} and a union bound, 
	on an event of probability 1, for each $k \in \N$, there exists $T^{(k)} > 0$ so that for all~$x \in B_{1/2}$ the family of sets
	\begin{equation}
	\{\cluster{B_k; k x}{t}\}_{0 < t < A_k T^{(k)}}
	\end{equation}
	is compactly embedded in $\overline{B_{k/2}}$, and $\mathrm{int}(\overline{\cluster{B_k; k x}{t}})$ satisfies the properties of Theorem \ref{theorem:harmonic-balls} for $t < A_k T^{(k)}$,
	and $\overline{\cluster{B_k; k x}{A_k T^{(k)}}} \cap \partial B_k \neq \emptyset$. 
	
	Let~$z \in \C$ be given, select~$k_0 = 3 |z|$ (so that there exists~$x \in B_{1/2}$ with~$k_0 x = z$) and define 
	\begin{equation}
	{\cluster{}{t}}(z)  = \begin{cases}
	\cluster{B_{k_0}; z}{t} \qquad &\mbox{for $t < A_{k_0} T^{(k_0)}$} \\
	\cluster{B_k; z}{t} \qquad &\mbox{if $A_{k-1} T^{(k-1)} \leq t < A_{k} T^{(k)}$ for some $k \in [k_0 + 1, \infty) \cap \N$}.
	\end{cases}
	\end{equation}
	and ${\regcluster{}{t}(z)} = \mathrm{int}(\overline{\cluster{}{t}(z)})$.
	By the compatibility property, Lemma \ref{lemma:cluster-compatibility}, the times $A_k T^{(k)}$ are increasing in~$k$
	and hence the construction is well-defined. 
	Compatibility also implies that the family $\{{\regcluster{}{t}}\}_{t > 0}$ satisfies the properties of Theorem \ref{theorem:harmonic-balls}. 
	
	It remains to show that for each $t > 0$ there exists a $k \in \N$ so that $t < A_{k} T^{(k)}$.
	That is, we must show that
	\begin{equation} \label{eq:strictly-increasing-times}
	A_{k} T^{(k)} \to \infty \qquad \mbox{with probability 1}.
	\end{equation}
	Indeed, if this were the case, this would give us a complete family $\{\regcluster{}{}(z)\}_{t > 0}$ satisfying the properties of Theorem \ref{theorem:harmonic-balls}, and we have uniqueness of such a family by Proposition \ref{prop:uniqueness}.

	First note that by Lemma \ref{lemma:scale-invariance-flow}, $T^{(k)} {\buildrel d \over =} T^{(1)}$ for each $k$.
	Since $T^{(1)}$ is strictly positive, for each $p \in (0,1)$, there exists 
	$c_p > 0$ so that 
	\[
	\P[T^{(k)} > c_p] \geq p
	\]
	for all $k \in \N$. In particular, 
	\begin{equation} \label{eq:positive-io}
	\P[\cap_{m=1}^{\infty} \cup_{k=m}^{\infty} T^{(k)} > c_p] \geq p.
	\end{equation}
	We claim that also
	\begin{equation} \label{eq:geometric-bm}
	\lim_{k \to \infty} A_k = \infty. 
	\end{equation}
	Indeed, the process, 
	\[
	t \to h_{e^{t}}(0)
	\]
	has a continuous modification which is a standard two-sided Brownian motion \cite[Section 3.1]{duplantier2011liouville}. 
	Thus, 
	\[
	t \to A_{e^{t}} = e^{\gamma (Q t + h_{e^{t}}(0))} 
	\]
	is a geometric Brownian motion with percentage drift $\gamma Q + \gamma^2/2$
	and percentage volatility $\gamma$ --- this implies \eqref{eq:geometric-bm}. Combining \eqref{eq:geometric-bm} with \eqref{eq:positive-io} and using that $T^{(k)}$ is increasing in $k$ shows that with probability at least $p$, 
	\[
	A_k T^{(k)} \to \infty.
	\]
	Since this holds for any $p \in (0,1)$, we have \eqref{eq:strictly-increasing-times}, 
	completing the proof. \qed

	\subsection{Harmonic balls are local}
	In this subsection we prove that the harmonic balls given by Theorem \ref{theorem:full-theorem-minus-uniqueness} are local; that is, 
	we prove Proposition \ref{prop:locally-determined}. 
	
	Before doing so, we note that we have constructed clusters and stated Lemma \ref{lemma:cluster-compatibility} for clusters restricted to domains which are balls.
	However, the definition of $\cluster{B_r; z}{t}$ and the proof of Lemma \ref{lemma:cluster-compatibility} extend essentially verbatim to the case when $B_r$ is replaced by any bounded open set containing the origin.
	
	\begin{lemma} \label{lemma:strong-compatibility}
		Let $U$ be a deterministic bounded open set. 
		For all $t > 0$ and~$z \in \C$, we have $\overline{{\regcluster{}{t}(z)}} = \overline{\cluster{U; z}{t}}$ if either ${\regcluster{}{t}(z)} \Subset U$ or $\cluster{U; z}{t} \Subset U$.
	\end{lemma}
	\begin{proof}
		This is immediate from the proof of Lemma \ref{lemma:cluster-compatibility} together with Theorem \ref{theorem:full-theorem-minus-uniqueness}. 
	\end{proof}

	\begin{proof}[Proof of Proposition \ref{prop:locally-determined}] 
		
		Let the deterministic open set $U$, base point $z \in \C$,~$t > 0$, and cluster ${\regcluster{}{t}(z)}$ be given. As $\P[\overline{{\regcluster{}{t}(z)}} \subset U]  = 0$ if $z\notin U$
		we suppose $z \in U$. As we may approximate $U$ by an increasing sequence of bounded open sets, we further suppose $U$ is bounded. 
		
		The cluster $\cluster{U; z}{t}$ depends only on $\mu_h |_{U}$ and hence, by locality (Fact \ref{fact:lqg-measure}), only on $h |_{U}$. 
		Therefore, it suffices to observe from Lemma \ref{lemma:strong-compatibility} that 
		\begin{equation} \label{eq:cluster-coincides}
		\overline{{\regcluster{}{t}(z)}} \subset U  \iff  \overline{\cluster{U; z}{t}} \subset U 
		\end{equation}
		This completes the proof.
	\end{proof}

	\section{Harnack-type estimate} \label{sec:harnack-type-estimate}
	Recall the notation for Euclidean annuli from~\eqref{eq:annulus}. The main goal of this section is to prove the following Harnack-type estimate for clusters: for every~$z \in B_{1/2}$,
	\begin{equation} \label{eq:harnack-type}
	\mu_h({\cluster{z}{t}} \cap \A_{\rho/2, \rho}(x_0)) \leq \alpha \mu_h(\A_{\rho/2, \rho}(x_0)) \implies {\cluster{z}{t}} \cap B_{\rho/2}(x_0) = \emptyset
	\end{equation}
	for all $t > 0$ where $\overline{B_{\rho}}(x_0) \subset B_1 \backslash \{z\}$ and $\alpha \in (0,1)$ is some fixed, small constant. See Figure \ref{fig:harnack-type} for a visualization of this condition. 
	\begin{figure}
		\begin{center}
			\includegraphics[width=0.5\textwidth]{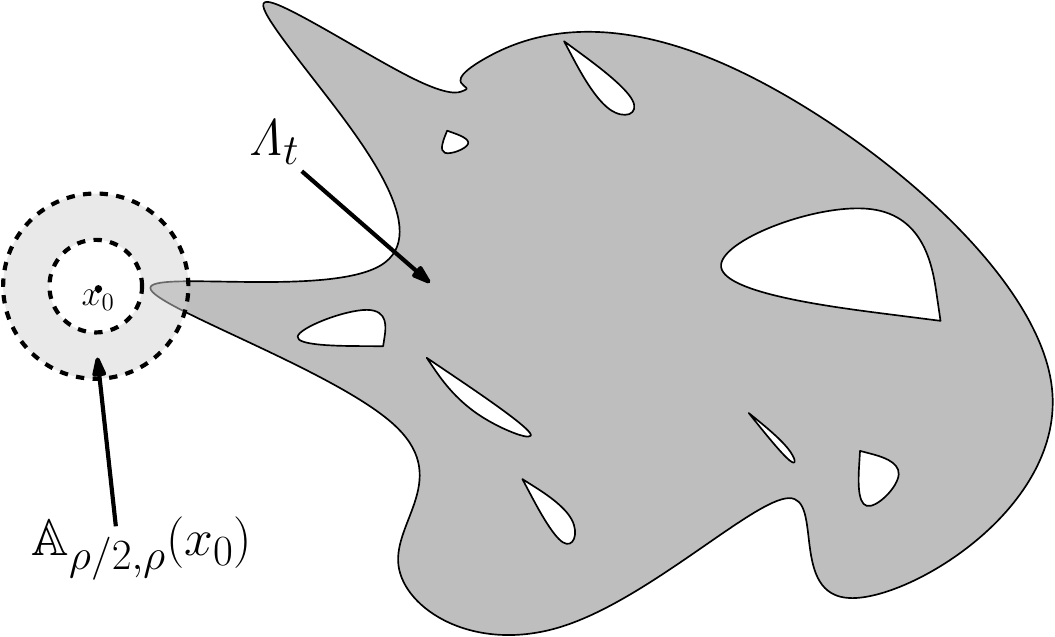}
		\end{center}
		\caption{An example of the Harnack-type property. The cluster ${\cluster{}{t}}$ is in gray with a solid boundary
			and an annulus $\A_{\rho/2,\rho}(x_0)$ for which $\overline E_\rho(x_0)$ occurs is displayed in light-gray with a dashed-line boundary. Proposition~\ref{prop:harnack-type-property} asserts that if $\mu_h( {\cluster{}{t}} \cap \A_{\rho/2,\rho}(x_0) )$ is small, then ${\cluster{}{t}} \cap B_{\rho/2}(x_0) = \emptyset$.  } \label{fig:harnack-type}
	\end{figure}
	
	Due to the variable nature of the Liouville measure, we cannot show this
	holds for every annulus but rather for `most' annuli. 	Specifically, we show the following.
	
	\begin{prop} \label{prop:harnack-type-property}	
		There exists $\alpha = \alpha(\gamma) > 0$ and events
		$\overline E_\rho(x_0)$ for $\rho \in (0,1)$ and $x_0 \in \C$
		with the following properties: for each~$z \in B_{1}$, if $\overline E_\rho(x_0)$ occurs, $\overline{B_{\rho}}(x_0) \subset B_1 \backslash \{z\}$, and $\mu_h({\cluster{z}{t}} \cap \A_{\rho/2, \rho}(x_0)) \leq \alpha \mu_h(\A_{\rho/2, \rho}(x_0))$, then ${\cluster{z}{t}} \cap B_{\rho/2}(x_0) = \emptyset$.
		
		Moreover,  there is a universal constant $c > 0$ so that
		with polynomially high probability as $\epsilon \to 0$ for each $x_0 \in (B_{1 + \epsilon} \backslash B_{10 \sqrt{\epsilon}}) \cap \frac{\epsilon}{100} \Z^2$ there are at least $c \log \epsilon^{-1/2}$ radii $\rho \in [\epsilon, \epsilon^{1/2}] \cap
		\{8^{-n}\}_{n \in \N}$ for which $\overline E_{\rho}(x_0)$ occurs. 
	\end{prop}

	In fact, we prove the following stronger statement. 
	
	\begin{prop} \label{prop:harnack-type-stronger-property}
		Assume we are in the setting of Proposition \ref{prop:harnack-type-property} and let $\rho \in (0,1)$ and $x_0 \in \C$. If $\overline E_\rho(x_0)$ occurs, $\overline{B_{\rho}}(x_0) \subset B_1 \backslash \{z\}$, and $\widetilde\Lambda_t$ is a connected component of ${\cluster{z}{t}} \cap B_{\rho}(x_0)$ 
		then the following occurs.  If $\mu_h(\widetilde \Lambda_t \cap \A_{\rho/2, \rho}(x_0)) \leq \alpha \mu_h(\A_{\rho/2, \rho}(x_0))$, then $\widetilde \Lambda_t \cap \overline{B_{\rho/2}(x_0)} = \emptyset$. 
	\end{prop}
	
	We note that Proposition~\ref{prop:harnack-type-stronger-property} implies the first part of Proposition~\ref{prop:harnack-type-property}. Indeed, if $\mu_h({\cluster{z}{t}} \cap \A_{\rho/2, \rho}(x_0)) \leq \alpha \mu_h(\A_{\rho/2, \rho}(x_0))$, then each connected component $\widetilde\Lambda_t$ as in Proposition~\ref{prop:harnack-type-stronger-property} satisfies $\mu_h(\widetilde \Lambda_t \cap \A_{\rho/2, \rho}(x_0)) \leq \alpha \mu_h(\A_{\rho/2, \rho}(x_0))$. So, if $\overline E_\rho(x_0)$ occurs, then Proposition~\ref{prop:harnack-type-stronger-property} implies that none of these connected components intersect $\overline B_{\rho/2}(x_0)$.  
	
	Proposition~\ref{prop:harnack-type-property} is sufficient for most of our applications, but Proposition~\ref{prop:harnack-type-stronger-property} is needed in Section \ref{sec:boundary-curves} to show that the boundaries
	of the complementary connected components of ${\cluster{}{t}}$ are curves. 
	
	The proof of the Harnack-type estimate is inspired by the `no thin tentacles' argument of Jerison-Levine-Sheffield \cite{jerison2012logarithmic}
	and the IDLA bound of \cite{duminil2013containing}. 
	The idea is as follows. If $A$ is an annulus, then with high probability for every set $Y\subset A$ such that $\mu_h(Y)$ is much smaller than $\mu_h(A)$, it is unlikely for a Brownian motion to cross between the inner and outer boundaries of $A$ without exiting $Y$ (Lemma~\ref{lemma:annulus-hit}). Hence, if $\mu_h({\cluster{}{t}} \cap A) / \mu_h(A)$ is small, then a Brownian motion is unlikely to cross between the inner and outer boundaries of $A$ before exiting ${\cluster{}{t}}$. Since the cluster ${\cluster{}{t}}$ is `grown according to harmonic measure', one can show that if $A'$ is a smaller annulus which is disconnected from 0 by $A$, then $\mu_h({\cluster{}{t}} \cap A') / \mu_h(A')$ is even smaller than $\mu_h({\cluster{}{t}} \cap A) / \mu_h(A)$. See Section~\ref{subsec:odometer-cluster-comparison} for precise statements to this effect. Iterating this across several nested annuli (with decreasing aspect ratios) allows us to prevent ${\cluster{}{t}}$ from intersecting an appropriate Euclidean ball.

	We make the above argument precise via a combination of potential theory and LQG arguments.
	We simultaneously study the odometer and the mass of the cluster. Specifically, we study the decay of the odometer and mass in disjoint shells of geometrically decreasing radii. We show that whenever an annulus is `very good' (as defined in Section~\ref{subsec:good-and-very-good-annuli}) and the mass is small in that annulus, then the odometer has to decrease by a geometric amount from one shell to the next (Lemma~\ref{lemma:harmonic-odometer-comparison}). If the odometer is small, then the mass is also small (Lemma~\ref{lemma:odometer-annuli}). This together with the prevalence of `very good' annuli established below forces the odometer (and mass) to go to zero. 
	
	We start by introducing notions of `good' and `very good' annuli
	in Section \ref{subsec:good-and-very-good-annuli}. We show in Sections \ref{subsec:good-annuli-are-prevalent}
	and \ref{subsec:very-good-annuli-are-prevalent} that there are many `good' and `very good' annuli.
	We then establish harmonic comparison lemmas which allow us to compare the size of the odometer and the LQG mass of the cluster
	in Section \ref{subsec:odometer-cluster-comparison}. In Section \ref{subsec:very-good-annuli-satisfy-harnack-type-estimate}
	we use these harmonic comparison lemmas to show that very good annuli satisfy the Harnack-type estimate. This result
	combined with the prevalence of very good annuli leads to the proof of Proposition \ref{prop:harnack-type-stronger-property}.

	\subsection{Good and very good annuli} \label{subsec:good-and-very-good-annuli}
	
	In this subsection we define what it means for an annulus to be good and very good.  
	Roughly, for a good annulus it is difficult for Brownian motion to stay within a set 
	of relatively small LQG measure until it exits the annulus. Very good annuli
	contain lots of good annuli and satisfy additional regularity properties. 
	We will later see that very good annuli satisfy \eqref{eq:harnack-type}.  
	
	\subsubsection{Good annuli}
	We start with defining good annuli. For $z \in \C$, $r > 0$ and parameters $a,b \in (0,1)$,
	let $E_r(z) = E_r(z; a,b)$ be the event that the following is true. For each Borel set $Y \subset \A_{3r, 5r}(z)$
	such that $\mu_h(Y) \leq a \mu_h(\A_{3 r, 5r}(z))$ we have 
	\begin{equation} \label{eq:good-annulus-harmonic}
	\sup_{u \in \partial B_{4 r}(z)} P\left[ \mbox{$\mathcal{B}^u$ exits $\A_{3 r, 5 r}(z)$ before exiting $Y$}  \, |\, h  \right] \leq b ,
	\end{equation}
	where $\mathcal B^u$ denotes standard planar Brownian motion started from $u$. 
	We note that $E_r(z) \in \sigma(h |_{\A_{3r,5r}(z)})$.
	The annuli $\A_{3 r, 5r}(z)$ for which $E_r(z)$ holds are {\it good}.
	
	\subsubsection{Alternative measures of the LQG size of an annulus}
	It will be convenient to go back and forth between Liouville measure and two other notions of size when using 
	\eqref{eq:good-annulus-harmonic}.  To that end, let $\beta^-$ be the growth lower bound exponent appearing in Lemma \ref{lemma:volume-growth} and define 
	\begin{equation} \label{eq:alternate-size}
	\Mrho{\rho}{x_0} = \inf_{z \in \A_{\rho/2, \rho}(x_0)} \inf_{r \in (0, \rho/4)} \frac{ \mu_h(B_r(z))}{ (r/\rho)^{\beta^-}} .
	\end{equation}
	By Lemma \ref{lemma:volume-growth}, a.s.\ $0 < \Mrho{\rho}{x_0} < \infty$ for each $\rho \in (0,1)$ and each $x_0 \in \C$. 
	We also define
	\begin{equation} \label{eq:alternate-size-2}
	\SGrho{\rho}{x_0} = \sup_{x \in \A_{\rho/4, 2 \rho}(x_0)} \left( \int_{\A_{\rho/4, 2 \rho}(x_0)} G_{B_{2 \rho}(x_0)}(x, y) d \mu_h(y) \right),
	\end{equation}
	where $G_{A}: \overline{A} \times \overline{A} \to \overline \R$ is the Green's function for the set $A$. 
	By Proposition~\ref{prop:lbm-exit-time}, a.s.\ $\SGrho{\rho}{x_0}$ is finite and positive for each $\rho\in (0,1)$ and each $x_0 \in \C$. 
	For later use, we also define a version of  $\SGrho{\rho}{x_0}$ with a variable aspect ratio,
	\begin{equation} \label{eq:alternate-size-2-full}
	\SGrhofull{\rho}{x_0}{s_1}{s_2} = \sup_{x \in \A_{s_1 \rho , s_2 \rho}(x_0)} \left( \int_{\A_{s_1 \rho, s_2 \rho}(x_0)} G_{B_{s_2 \rho}(x_0)}(x, y) d \mu_h(y) \right).
	\end{equation}
	
	\subsubsection{Very good annuli}
	For $x_0 \in \C, \rho > 0$, and parameters $N_0 \geq 1,  a \in (0,1),b \in(0,1), C^{\pm}_1 > 0, C^{\pm}_2 > 0, C_3 > 0, $ let $\overline E_\rho(x_0) = \overline E_\rho(x_0; N_0, a, b,C^{\pm}_1, C^{\pm}_2, C_3)$ be the event that the following are true:
	\begin{enumerate}[label=(VG-\roman*)]
		\item $C^-_1  \leq \frac{\Mrho{\rho}{x_0}}{\mu_h(A_{\rho/4, 2 \rho}(x_0))} \leq C^+_1$.  \label{enum:very-good-criteria-1}
		\item $C^-_2  \leq \frac{\SGrho{\rho}{x_0}}{\mu_h(A_{\rho/4, 2 \rho}(x_0))} \leq C^+_2$.   \label{enum:very-good-criteria-2}
		\item For each $\epsilon \in \{2^{-n}\}_{n \geq N_0}$
		and each $z \in \frac{\epsilon \rho}{100} \Z^2 \cap \A_{\rho/2, \rho}(x_0)$,
		there are at least $C_3 \log_7 \epsilon^{-1/2}$ radii $r \in [\epsilon \rho, \epsilon^{1/2} \rho] \cap \{\rho 7^{-k}\}_{k \geq 1}$
		for which $E_r(z; a, b)$ occurs.   \label{enum:very-good-criteria-3}
	\end{enumerate}
	The annuli for which $\overline E_{\rho}(x_0)$ occur are {\it very good}. 
	Our goal in the next two sections is to show that with high probability there are many very good annuli surrounding each point in $B_1$.

	\subsection{There are many good annuli} \label{subsec:good-annuli-are-prevalent}

	We start by showing that good annuli are prevalent, that is we prove the following.

	\begin{lemma} \label{lemma:good-annuli}
		Fix $b \in (0,1)$. There exists $a = a(b,\gamma) > 0$ and a universal constant $c > 0$ such that the following holds
		with polynomially high probability as $\epsilon \to 0$. 
		For each $z \in (B_{1 + \epsilon} \backslash B_{10 \sqrt{\epsilon}}) \cap \frac{\epsilon}{100} \Z^2$ there are at least $c \log_7 \epsilon^{-1/2}$ radii $r \in [\epsilon, \epsilon^{1/2}] \cap
		\{7^{-n}\}_{n \in \N}$ for which $E_r(z)$ occurs, where $E_r(z) = E_r(z; a, b)$ is as in \eqref{eq:good-annulus-harmonic}. 
	\end{lemma}
	
	We note the similarity between Lemma~\ref{lemma:good-annuli} and condition \ref{enum:very-good-criteria-3} in the definition of very good annuli.
	To prove Lemma \ref{lemma:good-annuli} we will first show that for each $z \in \C$ and $r > 0$, the event $E_r(z)$ occurs 
	with high probability provided $a$ is chosen to be sufficiently small depending on $b$ (Lemma \ref{lemma:annulus-hit}).
	We will then use the near-independence of the GFF across disjoint concentric annuli (Lemma~\ref{lemma:annulus-iterate}) 
	to show that for each fixed $z \in \C$, it holds with very high probability when $\epsilon$ is small 
	that there are many radii $r \in [\epsilon, \epsilon^{1/2}] \cap \{7^{-n}\}_{n \in \N}$ for which $E_r(z)$ occurs. 
	Finally, we will take a union bound over all $z \in B_{1 + \epsilon} \cap \frac{\epsilon}{100} \Z^2$.

	\begin{lemma} \label{lemma:annulus-hit}
		Let $b \in (0,1)$ and $p \in (0,1)$. There exists $a = a(b,p,\gamma) \in (0,1)$ such that the event $E_r(z) = E_r(z;a,b)$ of~\eqref{eq:good-annulus-harmonic} satisfies 
		\[
		P\left[ E_r(z) \right] \geq p ,\quad \forall r > 0,\quad \forall z \mbox{ such that $\dist(\A_{3 r, 5 r}(z), \{0\}) \geq r/100$}.
		\]
	\end{lemma}
	
	\begin{proof}

		We first show that it suffices to prove the lemma for $h^{\C}$, \ie, the GFF without a log-singularity ($\boldsymbol{\alpha}_0 = 0$ in \eqref{eq:gff}).
		We then prove the lemma for the case $\boldsymbol{\alpha}_0 = 0$.
		
		{\it Step 1: Reduction to $\boldsymbol{\alpha}_0 = 0$.} \\
		Recall that $h = h^\C -\boldsymbol{\alpha}_0 \log |\cdot|$, where $h^\C$ is a whole-plane GFF. 
		We first explain why it is sufficient to prove the lemma with $h^\C$ in place of $h$. 
		Suppose the statement of the lemma holds for $h^{\C}$ with $a' \in (0,1)$ in place of $a$. 
		
		Fix $r > 0$ and then  $z$ such that the annulus $\A_{3 r, 5 r}(z)$ lies at Euclidean distance at least $r/100$ from the origin. 
		By Weyl scaling (Fact~\ref{fact:lqg-measure}), there are constants $C_1,C_2>  0$ depending only on $\gamma$ such that 
		\begin{equation}
		\mu_h(Y) \geq C_1 |z|^{-\boldsymbol{\alpha}_0 \gamma} \mu_{h^{\C}}(Y), \quad \forall Y \subset \A_{3 r, 5 r}(z)
		\end{equation}
		and
		\begin{equation}
		\mu_h(Y) \leq C_2 |z|^{-\boldsymbol{\alpha}_0 \gamma} \mu_{h^{\C}}(Y), \quad \forall Y \subset \A_{3 r, 5 r}(z) .
		\end{equation} 
		Thus, for all $Y \subset \A_{3 r, 5 r}(z)$
		\[
		\mu_{h^\C}(Y) \leq a' \mu_{h^\C}(\A_{3 r, 5 r}(z)) \implies  \mu_h(Y) \leq \frac{C_2}{C_1} \times a' \mu_{h}(\A_{3 r, 5 r}(z))
		\]
		and hence $\P[E_r(z)] \geq p$ for $a := \frac{C_2}{C_1} \times a'$.

		{\it Step 2: Case when $\boldsymbol{\alpha}_0 = 0$.} \\
		For the rest of the proof we assume that $\boldsymbol{\alpha}_0 = 0$. The law of $h = h^\C$ is both scale and translation invariant modulo additive constant. By the Weyl scaling property of the measure $\mu_h$ (Fact~\ref{fact:lqg-measure}), the event $E_r(z)$ is a.s.\ determined by $h$ viewed modulo additive constant. From this and the LQG coordinate change formula for $\mu_h$, we infer that the probability of $E_r(z)$ does not depend on $r$ or $z$. Hence, it suffices to find $a  \in (0,1)$ as in the lemma statement such that $\P[E_1(0)] \geq p$. 
		
		To this end, for $u \in \C$ let $\mathcal{B}^u$ denote $\gamma$-Liouville Brownian motion with respect to the field $h$, started from $u$ with 
		reflecting boundary conditions in a square $K$ centered at the origin with side length 100. 
		
		By~\cite[Proposition 2.19]{garban2016liouville}, the conditional law of $\mathcal B^u$ stopped when exiting $\A_{3, 5}(0)$ depends continuously on $u$.
		Although the proof in \cite{garban2016liouville} is for the whole-plane massive GFF, as explained in ~\cite[Section 9]{garban2016liouville}, Proposition 2.19 in ~\cite{garban2016liouville} 
		extends to the massless GFF in a finite domain. Moreover, ordinary LBM and reflected LBM coincide until the first exit from $\A_{3,5}(0)$.
		Hence, by the compactness of the circle $\partial B_4(0)$, we may therefore find a random $T = T(h) > 0$ such that a.s.\
		\begin{equation} \label{eq:annulus-hit-time}
		\sup_{u \in \partial B_4(0)}  P\left[ \text{$\mathcal B^u$ exits $\A_{3,5}(0)$ before time $T$} \,|\, h \right] \leq  \frac{b}{2} .
		\end{equation} 
		
		For $t>0$, let $p^K_t(u,\cdot)$ be the time $t$ Liouville heat kernel for $\mathcal B^u$, so that $p^K_t(u,\cdot)\, d\mu_h$ is the law of $\mathcal B^u_t$. 
		By Proposition \ref{prop:lbm-continuity} a.s.\ $p^K_t(u,v)$ is a continuous function of $(t,u,v)$ and  $p^K_t(u,v) > 0$ for all $t > 0$ and all $u,v\in K$. 
		Again using the compactness of $\partial B_4(0)$, we infer that with $T$ as in~\eqref{eq:annulus-hit-time}, there exists a random $C = C(h)  > 0 $ such that a.s.\
		\begin{equation}
		\label{eq:heat-kernel-max}
		\sup_{u\in\partial B_4(0)} \sup_{v\in B_5(0)} p_{T}(u,v) \leq C  .
		\end{equation} 
		
		From~\eqref{eq:heat-kernel-max}, we get that for each Borel set $Y \subset  \A_{3,5}(0)$, 
		\begin{equation}
		\label{eq:heat-kernel-int}
		\sup_{u\in \partial B_4(0)} P\left[ \mathcal B^u_{T} \in Y \,|\, h \right] \leq C \mu_h(Y) .
		\end{equation}
		Hence, if $\mu_h(Y) \leq [ C \mu_h(\A_{3,5}(0) ) ]^{-1} (b/2)  \times \mu_h(\A_{3,5}(0) )$, then $P\left[ \mathcal B^u_{T} \in Y \,|\, h \right] \leq b/2$ for each $u\in \partial B_4(0)$. Combining this with~\eqref{eq:annulus-hit-time} shows that for every such Borel set $Y$,
		\begin{equation}
		\sup_{u \in \partial B_4(0)}  P\left[ \text{$\mathcal B^u$ exits $\A_{3,5}(0)$ before exiting $Y$} \,|\, h \right] \leq b .
		\end{equation}
		
		That is, a.s.\ $E_1(0)$ occurs with $a$ replaced by the random variable $[ C \mu_h(\A_{3,5}(0) ) ]^{-1} (b/2)$. This random variable is a.s.\ positive, so we can find a deterministic $a \in (0,1)$ such that 
		\begin{equation}
		P\left[ [ C \mu_h(\A_{3,5}(0) ) ]^{-1} (b/2) \geq a \right] \geq p . 
		\end{equation}
		Hence, for this choice of $a$ we have $\P[E_1(0)] \geq p$, as required.
	\end{proof}

	The following lemma is a consequence of the fact that the restrictions of the GFF to disjoint concentric annuli, viewed modulo additive constant, are nearly independent. See~\cite[Lemma 3.1]{gwynne2020local} for a slightly more general statement.

	\begin{lemma}[\cite{gwynne2020local}] \label{lemma:annulus-iterate}
		Fix $0 < s_1<s_2 < 1$. Let $\{r_k\}_{k\in\N}$ be a decreasing sequence of positive numbers such that $r_{k+1} / r_k \leq s_1$ for each $k\in\N$ and let $\{E_{r_k} \}_{k\in\N}$ be events such that $E_{r_k}$ is measurable with respect to $h |_{\A_{s_1 r_k , s_2 r_k}(0)  } $, viewed modulo additive constant, for each $k\in\N$. 
		For $K\in\N$, let $N(K)$ be the number of $k\in [1,K] \cap \Z$ for which $E_{r_k}$ occurs. 
		For each $\alpha > 0$ and each $\beta \in (0,1)$, there exists $p = p(\alpha,\beta,s_1,s_2) \in (0,1)$ and $C = C(\alpha,\beta,s_1,s_2) > 0$ (independent of the particular choice of $\{r_k\}$ and $\{E_{r_k}\}$) such that if  
		\begin{equation} \label{eqn-annulus-iterate-prob}
		P\left[ E_{r_k}  \right] \geq p , \quad \forall k\in\N  ,
		\end{equation} 
		then 
		\begin{equation} \label{eqn-annulus-iterate}
		P\left[ N(K)  < \beta K\right] \leq C e^{-\alpha K} ,\quad\forall K \in \N. 
		\end{equation}  
	\end{lemma}
	
	We now prove the desired claim.

	\begin{proof}[Proof of Lemma~\ref{lemma:good-annuli}]
		The event $E_r(z)$ depends only on the measure $\mu_h|_{\A_{3 r,5r}(z)}$. Moreover, multiplying this measure by a constant does not change whether $E_r(z)$ occurs. 
		Therefore, $E_r(z)$ is a.s.\ determined by $h|_{\A_{3 r,5r}(z)}$ viewed modulo additive constant.
		
		We now apply Lemma~\ref{lemma:annulus-iterate} with $K = \lfloor \log_7  \epsilon^{-1/2} \rfloor$, the radii $r_1,\dots,r_K \in [ \epsilon, \epsilon^{1/2}] \cap \{7^{-n}\}_{n\in\N}$, the events $E_{r_k} = E_{r_k}(z)$, and appropriate universal constant choices of $\alpha$ and $\beta$. We find that there exist universal constants $p\in (0,1)$ and $c>0$ such that if $\P[E_r(z)] \geq p$
		for each $r > 0$ and each $z \in \C \backslash B_{10 r}$,  then for all $z \in \C \backslash B_{10 \sqrt{\epsilon}}$, 
		\begin{equation} \label{eqn-annulus-hit-iterate}
		P\left[ \text{$E_r(z)$ occurs for at least $c\log\epsilon^{-1/2}$ values of $r \in [ \epsilon, \epsilon^{1/2}] \cap \{7^{-n}\}_{n\in\N}$} \right] \geq 1 - O_\epsilon(\epsilon^3) 
		\end{equation} 
		with a universal implicit constant in the $O_\epsilon(\cdot)$. 
		
		By Lemma~\ref{lemma:annulus-hit}, there exists $a = a(b,\gamma) > 0$ such that for this choice of $a$, one has $\P[E_r(z)] \geq p$ for each $r > 0$ and each $z \in \C \backslash B_{10 r}$. Therefore, the estimate~\eqref{eqn-annulus-hit-iterate} holds for this choice of $a$. We now conclude by means of a union bound over all $z\in  (B_{1 + \epsilon} \backslash B_{10 \sqrt{\epsilon}}) \cap \frac{\epsilon}{100} \Z^2$. 
	\end{proof}

	\subsection{There are many very good annuli} \label{subsec:very-good-annuli-are-prevalent}
	In this section we prove that very good annuli are prevalent, following the same strategy as the last section. 
	
	\begin{lemma} \label{lemma:very-good-annuli}
		Fix $b \in (0,1)$. There exists $a = a(b,\gamma) > 0$, universal constants $c, C_3 > 0$, $C^{\pm}_1, C^{\pm}_2 > 0$ depending on $a, \gamma$, and $N_1 = N_1(b,\gamma) \geq 1$ such that for all $N_0  \geq N_1 $ the following holds with polynomially high probability as $\epsilon \to 0$. For each $x_0 \in (B_{1 + \epsilon} \backslash B_{10 \sqrt{\epsilon}}) \cap \frac{\epsilon}{100} \Z^2$ there are at least $c \log_8 \epsilon^{-1/2}$ radii $\rho \in [\epsilon, \epsilon^{1/2}] \cap
		\{8^{-n}\}_{n \in \N}$ for which $\overline E_\rho(x_0)$ occurs, where $\overline E_\rho(x_0) = \overline E_\rho(x_0; N_0, a, b,C^{\pm}_1, C^{\pm}_2, C_3)$ is as in \ref{enum:very-good-criteria-1}, 	\ref{enum:very-good-criteria-2}, 	\ref{enum:very-good-criteria-3}.
		
	\end{lemma}
	
	We start by showing the event $\overline E_\rho(x_0)$ occurs with high probability 
	with $a,b$ chosen as in Lemma \ref{lemma:good-annuli},  $N_0$ large, and $C^{\pm}_1$,$C^{\pm}_2$
	chosen appropriately.

	\begin{lemma} \label{lemma:very-good-annuli-hit}
		Let $b \in (0,1)$ and $p \in (0,1)$.
		There exists $a = a(b,p,\gamma) \in (0,1)$,  $C^{\pm}_1, C^{\pm}_2$ depending on $a, \gamma, p$
		and a universal constant $C_3>0$ such that the event $\overline E_\rho(x_0)$ of \ref{enum:very-good-criteria-1} \ref{enum:very-good-criteria-2} \ref{enum:very-good-criteria-3} satisfies 
		\[
		P\left[ \overline E_\rho(x_0) \right] \geq p ,\quad \forall \rho > 0,\quad \forall x_0 \mbox{ such that $\dist(\A_{\rho/4, 2 \rho}(x_0), \{0\}) \geq \rho/100$}
		\]
		for all $N_0 \geq N_1(b,p,\gamma) \geq 1$ sufficiently large. 
	\end{lemma}
	
	\begin{proof}
		Recall that $h = h^\C -\boldsymbol{\alpha}_0 \log |\cdot|$. The proof is similar to that of Lemma \ref{lemma:annulus-hit}. We show that 
		we can reduce to the case $\boldsymbol{\alpha}_0 = 0$ and then give a proof in that case. 
		
		{\it Step 1: Reduction to $\boldsymbol{\alpha}_0 = 0$.} \\
		Suppose the statement of the lemma holds for $h^{\C}$ with constants $\underline{C}^{\pm}_1$ in place of $C^{\pm}_1$.  Write $M^{h}_{\rho}(x_0)$ to indicate the dependence in the definition of $\Mrho{\rho}{x_0}$
		on the GFF.

		Fix $\rho > 0$ and then $x_0$ such that the annulus $\A_{\rho/4, 2\rho}(x_0)$ lies at Euclidean distance at least $\rho/100$ from the origin. 
		By Weyl scaling (Fact~\ref{fact:lqg-measure}), there are constants $A_1,A_2>  0$ depending only on $\gamma$ such that 
		\begin{equation}
		\frac{M^{h}_{\rho}(x_0)}{\mu_h(A_{\rho/4, 2 \rho}(x_0))} \geq A_1  \frac{M^{h^\C}_{\rho}(x_0)}{\mu_{h^\C}(A_{\rho/4, 2 \rho}(x_0))}
		\end{equation}
		and
		\begin{equation}
		\frac{M^{h}_{\rho}(x_0)}{\mu_h(A_{\rho/4, 2 \rho}(x_0))} \leq A_2  \frac{M^{h^\C}_{\rho}(x_0)}{\mu_{h^\C}(A_{\rho/4, 2 \rho}(x_0))}.
		\end{equation} 
		Thus, 
		\[
		\underline{C}^{-}_1 \leq  \frac{M^{h^\C}_{\rho}(x_0)}{\mu_{h^\C}(A_{\rho/4, 2 \rho}(x_0))} \leq \underline{C}^{+}_1 \implies A_1 \underline{C}^{-}_1 \leq \frac{M^{h}_{\rho}(x_0)}{\mu_h(A_{\rho/4, 2 \rho}(x_0))} \leq A_2 \underline{C}^{+}_1 
		\]	
		and hence \ref{enum:very-good-criteria-1} occurs for $C^+_1 := A_2 \underline{C}^{+}_1$ and $C^-_1 := A_1 \underline{C}^{-}_1$ if it occurs under $h^{\C}$.
		The argument for \ref{enum:very-good-criteria-2} is nearly identical. Step 1 of the proof of Lemma \ref{lemma:annulus-hit} also shows that we can 
		reduce to the case $\boldsymbol{\alpha}_0 = 0$ for \ref{enum:very-good-criteria-3}.

		{\it Step 2: Case when $\boldsymbol{\alpha}_0 = 0$.} \\
		As in Step 2 of the proof of Lemma \ref{lemma:annulus-hit} it suffices to find $a \in (0,1)$
		and  $C^{\pm}_1, C^{\pm}_2, C_3 > 0$ such that $\P[\overline E_1(0)] \geq p$.
		
		Note that by Lemma \ref{lemma:volume-growth}, $\Mrho{1}{0}$ is a strictly positive, finite random variable. 
		Since $\mu_h(\A_{1/4, 2}(0))$ and $\SGrho{1}{0}$ are also strictly positive and finite random variables, 
		there exists $C^{\pm}_1, C^{\pm}_2 > 0$ so that 
		\begin{equation}
		P\left[C^-_1  \leq \frac{\Mrho{1}{0}}{\mu_h(A_{1/4, 2}(0))} \leq C^+_1\right] \geq p_1
		\quad \text{and} \quad
		P\left[C^-_2  \leq \frac{\SGrho{1}{0}}{\mu_h(A_{1/4, 2}(0))} \leq C^+_2\right] \geq p_2.
		\end{equation}
		Also, by Lemma \ref{lemma:good-annuli}, for $N_0$ sufficiently large, with probability at least $p_3$, the event in \ref{enum:very-good-criteria-3} occurs with $\rho = 1$ and $C_3$ an appropriate universal constant.

		By adjusting our choices of parameters so that $(1-p_1) +  (1-p_2) + (1-p_3) \leq 1-p$, we may conclude via a union bound. 
	\end{proof}

	\begin{proof}[Proof of Lemma~\ref{lemma:very-good-annuli}]
		Given Lemma \ref{lemma:very-good-annuli-hit}, the  argument is identical to the proof of Lemma~\ref{lemma:good-annuli}.
		%
		%
		%
	\end{proof}

	\subsection{Harmonic comparison} \label{subsec:odometer-cluster-comparison}
	In this section we prove lemmas which let us compare the size of the odometer to the LQG mass of the cluster.
	Our first lemma allows us to show that the LQG mass of the cluster is small in annuli where the odometer is small. 
	\begin{lemma} \label{lemma:odometer-annuli}
		Fix $0 < s_1 < s_2 < s_3 < s_4$ and~$z \in B_{1}$.
		Let $\A_{s_1 r, s_4 r}(w) \subset B_1$ be an annulus not containing~$z$
		and let $\hat \Lambda_t$ be a union of connected components of $\A_{s_1 r, s_4 r}(w) \cap {\cluster{z}{t}}$.
		There exists a constant $C$, depending only on $s_{1},\ldots,s_4$,  
		so that for all such annuli
		\[
		\mu_h(\A_{s_2 r, s_3 r}(w) \cap \hat \Lambda_t) \leq C \sup_{x \in \hat \Lambda_t} \odometer{z}{t}(x).
		\]
	\end{lemma}
	
	\begin{proof}
		First note that there is a positive constant $c_1$ so that the annulus $\A_{s_2, s_3}$ can be covered by $c_1^{-1}$ 
		balls of radius $c_2 := \min(s_2-s_1, s_3 -s_2, s_4-s_3)/4$ centered at points in $\A_{s_2, s_3}$. Therefore, by scaling,
		this implies the annulus $\A_{s_2 r, s_3 r}(w)$ can be covered by $c_1^{-1}$ balls of radius $c_2 r$
		centered at points in $\A_{s_2 r, s_3 r}(w)$. By the pigeonhole principle, there is at least one such ball $B_{2 c_2 r}(x) \Subset \A_{s_1 r, s_4 r}(w)$ 
		with 
		\begin{equation} \label{eq:ball-lower-bound-mass}
		\mu_h(B_{c_2r}(x) \cap \hat \Lambda_t) \geq c_1 \mu_h(\A_{ s_2 r,  s_3 r}(w) \cap \hat \Lambda_t).
		\end{equation}
		
		Write $\odometer{z}{t} |_{\hat \Lambda_t}$ for $\odometer{z}{t}(\cdot) 1\{ \cdot \in\hat \Lambda_t\}$.
		As we will explain just below, one can deduce from Lemma~\ref{lemma:non-coincidence-open-harmonic} that
		\begin{equation} \label{eq:laplacian-of-indicator}
		\Delta( \odometer{z}{t} |_{\hat \Lambda_t} ) = 
		\begin{cases}
		\mu_h &\mbox{ on $\hat \Lambda_t$ } \\
		0 &\mbox{on $\A_{s_1 r, s_4 r}(w) \backslash \overline{\hat \Lambda_t}$} \\
		\geq 0 &\mbox{ on $\partial \hat \Lambda_t \cap \A_{s_1 r, s_4 r}(w)$}.
		\end{cases}
		\end{equation}
		Indeed, $\hat \Lambda_t$ is a union of connected components of ${\cluster{z}{t}} \cap \A_{s_1 r, s_4 r}(w)$, 
		an open set. The odometer, $\odometer{z}{t}$ is non-negative and continuous on $B_1 \setminus \{z\}$ and $\odometer{z}{t} = 0$ on $\partial \hat \Lambda_t \cap \A_{s_1 r, s_4 r}(w)$. Therefore $\odometer{z}{t} |_{\hat \Lambda_t}$ is continuous on $\hat \Lambda_t \cap \A_{s_1 r, s_4 r}(w)$
		and $\odometer{z}{t} |_{\hat \Lambda_t}$  satisfies the sub-mean-value property on $\partial \hat \Lambda_t \cap \A_{s_1 r, s_4 r}(w)$.
		As $\odometer{z}{t} |_{\hat \Lambda_t}$ coincides with  $\odometer{z}{t}$ on $\hat \Lambda_t$, this shows \eqref{eq:laplacian-of-indicator} by Lemma~\ref{lemma:non-coincidence-open-harmonic}.

		Let $G_A$ denote the Green's function for the domain $A$ with zero boundary conditions
		and let 
		\[
		u(\cdot) = \int_{B_{2 c_2 r}(x) \cap \hat \Lambda_t} G_{B_{2 c_2 r}(x)}(\cdot,y) \mu_h(dy).
		\]
		Observe that $\Delta (u+\odometer{z}{t} |_{\hat \Lambda_t} ) \geq 0$ on $B_{2 c_2 r}(x)$: indeed, 
		by the definition of $u$ we have $\Delta u = -\mu_h$ on $\hat \Lambda_t \cap B_{2 c_2 r}(x)$ and $\Delta u = 0$ elsewhere
		which, together with \eqref{eq:laplacian-of-indicator}, shows $u+\odometer{z}{t} |_{\hat \Lambda_t} $ is subharmonic on $B_{2 c_2 r}(x)$.

		Hence, by the maximum principle, on $B_{2 c_2 r}(x)$, 
		\[
		u+\odometer{z}{t} |_{\hat \Lambda_t}  \leq \sup_{\partial B_{2 c_2  r}(x)} (u + \odometer{z}{t} |_{\hat \Lambda_t} ) = \sup_{\partial B_{ 2 c_2 r }(x)} \odometer{z}{t} |_{\hat \Lambda_t}  \leq \sup_{\A_{s_1 r, s_4 r}(w)} \odometer{z}{t} |_{\hat \Lambda_t}  = \sup_{\hat \Lambda_t} \odometer{z}{t} 
		\]
		as $u = 0$ on $\partial B_{2 c_2 r}(x)$ and $B_{2 c_2 r}(x) \Subset \A_{s_1 r, s_4 r}(w)$. Thus, as $\odometer{z}{t} |_{\hat \Lambda_t}  \geq 0$, 
		\begin{equation} \label{eq:mass-ineq-upper-bound}
		0 \leq \left(\sup_{\hat \Lambda_t} \odometer{}{t} \right) - u ,\quad \text{on $B_{2 c_2 r}(x)$} .
		\end{equation} 
		
		We now estimate $u$ at the center of $B_{2 c_2 r}(x)$. By the definition of $u$ and then the scale invariance of the Green's function for a ball,
		\begin{align*}
		u(x) &= \int_{B_{2 c_2  r}(x) \cap \hat \Lambda_t} G_{B_{2 c_2 r}(x)}(x,y) \mu_h(dy) \\
		&\geq \int_{B_{c_2 r}(x) \cap \hat \Lambda_t} G_{B_{1}}(0, (2 c_2 r)^{-1}(y-x)) \mu_h(dy) \\
		&\geq C \mu_h(\A_{s_2 r, s_3 r}(w) \cap \hat \Lambda_t) \qquad \mbox{(by \eqref{eq:ball-lower-bound-mass})}
		\end{align*}
		where $C := c_1 \inf_{y \in B_{1/2}(0)} G_{B_{1}(0)}(0, y) > 0$ is independent of $r$.  We conclude the proof by combining this lower bound for $u(x)$ with \eqref{eq:mass-ineq-upper-bound}.
	\end{proof}

	We next show that if it is difficult for Brownian motion to get through a domain without
	exiting ${\cluster{}{t}}$, then the odometer must be small within the domain. 
	\begin{lemma} \label{lemma:harmonic-odometer-comparison} 
		Let~$z \in B_1$, and let $A$ denote an open set in $B_1$ not containing~$z$.
		For all $x \in A \cap {\cluster{z}{t}}$
		\begin{equation}
		\odometer{z}{t}(x) \leq \left( \sup_{\partial A \cap \overline \Lambda^x_t} \odometer{z}{t} \right)  \P[\mbox{$\mathcal{B}^x$ exits  $A$ before hitting  $A \backslash \Lambda^x_t$} | h] \, , 
		\end{equation}
		where $\mathcal{B}^x$ denotes an independent Brownian motion started at $x$ and $\Lambda^x_t$ 
		is the connected component of $A \cap {\cluster{z}{t}}$ containing $x$. 
	\end{lemma}
	
	\begin{proof}
		Write $f_A(x) =\P[\mbox{$\mathcal{B}^x$ exits  $A$ before hitting  $A \backslash {\cluster{z}{t}}$} | h]$. Observe that
		\begin{equation} \label{eq:expression-for-fA}
		\begin{cases}
		\Delta f_A = 0 \quad &\mbox{in $A \cap {\cluster{z}{t}}$} \\
		f_A = 1 \quad &\mbox{on $\partial A \cap {\cluster{z}{t}}$} \\
		f_A = 0 \quad &\mbox{on $A \cap \partial {\cluster{z}{t}}$}
		\end{cases}
		\end{equation}
		and by Lemma~\ref{lemma:non-coincidence-open-harmonic},
		\begin{equation}\label{eq:expression-for-odometer}
		\begin{cases}
		\Delta \odometer{z}{t} \geq 0 \quad &\mbox{in $A \cap {\cluster{z}{t}}$} \\
		\odometer{}{t} \leq \sup_{\partial A \cap {\cluster{}{t}}} \odometer{z}{t} \quad &\mbox{on $\partial A \cap {\cluster{z}{t}}$} \\
		\odometer{z}{t} = 0 \quad &\mbox{on $A \cap \partial {\cluster{z}{t}}$}.
		\end{cases}
		\end{equation}
		Now, fix $x \in A \cap {\cluster{z}{t}}$ and consider $\Lambda^x_t$, the connected component of $A \cap {\cluster{z}{t}}$ containing $x$.  
		As $A \cap \Lambda^x_t$ is a connected component of $A \cap {\cluster{z}{t}}$, \eqref{eq:expression-for-fA} and \eqref{eq:expression-for-odometer}
		show that 
		\[
		\odometer{z}{t}(\cdot) - \left(\sup_{\partial A \cap \overline \Lambda^x_t} \odometer{z}{t}\right) f_A(\cdot)
		\]
		is subharmonic in $A \cap \Lambda^x_t$, equal to 0 on $A \cap \partial \Lambda^x_t$, and less than or 
		equal to 0 on $\partial A \cap \overline \Lambda^x_t$. 
		Hence, by the maximum principle,
		\[
		\odometer{z}{t}(y) \leq  \left(\sup_{\partial A \cap \overline \Lambda^x_t} \odometer{z}{t}\right) f_A(y) , \quad \mbox{for $y \in A \cap \Lambda^x_t$} .
		\]
		Moreover, 
		\[
		f_A(y) = \P[\mbox{$\mathcal{B}^y$ exits  $A$ before hitting  $A \backslash \Lambda^x_t$} | h ] \quad \mbox{for $y \in A \cap \Lambda^x_t$}
		\]
		since there is no path in $A$ from $y$ to any point of $\Lambda_t\setminus \Lambda_t^x$.
		The previous two sentences imply  
		\[
		\odometer{z}{t}(y) \leq  \left(\sup_{\partial A \cap \overline \Lambda^x_t} \odometer{z}{t}\right) \P[\mbox{$\mathcal{B}^y$ exits  $A$ before hitting  $A \backslash \Lambda^x_t$} |h ]  \quad  \mbox{for $y \in A \cap \Lambda^x_t$}
		\]
		completing the proof. 	
	\end{proof}

	We next provide a weak upper bound on the growth of the odometer around its zeros.  
	\begin{lemma} \label{lemma:odometer-bound-lqg-mass}
		Fix $0 < s_1  < s_2 < s_3 < s_4$. 
		There exists a constant $c > 0$, depending only on $s_1, \ldots, s_4$,  so that, with $\SGrhofull{\rho}{x_0}{s_1}{s_4}$ as in~\eqref{eq:alternate-size-2-full}, for all~$z \in B_1$
		\[
		\sup_{x \in \A_{s_2 \rho,s_3 \rho}(x_0)} \odometer{z}{t}(x) \leq c  \SGrhofull{\rho}{x_0}{s_1}{s_4}
		\]
		for all $\A_{s_1 \rho,s_4 \rho}(x_0) \subset B_1 \backslash \{z\}$ such that ${{{\cluster{z}{t}}}}^c \cap  \A_{s_2 \rho,s_3 \rho}(x_0) \neq \emptyset$. 
	\end{lemma}
	
	\begin{proof}
		Fix an annulus $\A_{s_1 \rho,s_4 \rho}(x_0) \subset B_1 \setminus \{z\}$ and let $\lambda = \Delta \odometer{z}{t}$.
		Consider the positive function  
		\[
		q(w) = \int_{\A_{s_1 \rho,s_4 \rho}(x_0)} G_{B_{s_4 \rho}(x_0)}(w,y) d \lambda(y).
		\]
		for $w \in \A_{s_1 \rho,s_4 \rho}(x_0)$.
		Since $\lambda \leq \mu_h$ (Lemma~\ref{lemma:non-coincidence-open-harmonic}), 
		\begin{equation} \label{eq:q-inequality} 
		\sup_{w \in \A_{s_1 \rho,s_4 \rho}} q(w) \leq \SGrhofull{\rho}{x_0}{s_1}{s_4}.
		\end{equation}
		Hence the statement of the lemma will follow once we bound $\odometer{z}{t}$ by $q$. We do this via Harnack's inequality for positive harmonic functions.

		Our choice of $q$ ensures it is positive and 
		\[
		\Delta q = -\lambda \qquad \mbox{ on $\A_{s_1 \rho,s_4 \rho}(x_0)$}.
		\]
		This implies that the function $g: \A_{s_1 \rho,s_4 \rho}(x_0) \to \R$ defined by
		\[
		g := \odometer{z}{t} + q
		\]
		is harmonic and non-negative in $\A_{s_1 \rho,s_4 \rho}(x_0)$.
		Fix $s_2', s_3'$ so that 
		\[
		0 < s_1 < s_2' < s_2 < s_3 < s_3' < s_4.
		\]
		By Harnack's inequality for positive harmonic functions, 
		\begin{equation} \label{eq:g-harnack-step}
		\sup_{z \in \A_{s_2' \rho,s_3' \rho}} g(z) \leq c \inf_{y \in \A_{s_2 \rho,s_3 \rho}(x_0)} g(y)
		\end{equation}
		for a constant $c$ (depending only on the ratio of the domains on the left and right of \eqref{eq:g-harnack-step}). 
		Note that the assumption ${{{\cluster{z}{t}}}}^c \cap  \A_{s_2 \rho,s_3 \rho}(x_0) \neq \emptyset$ implies the existence of $z_0 \in \A_{s_2 \rho,s_3 \rho}(x_0)$ with $ \odometer{z}{t}(z_0) = 0$. 
		Hence, 
		\begin{align*}
		\sup_{z \in \A_{s_2' \rho,s_3' \rho}(x_0)} \odometer{z}{t}(z) &\leq \sup_{w \in \A_{s_2' \rho,s_3' \rho}(x_0)} g(w)  \qquad \mbox{ (since $q \geq 0$)} \\
		&\leq c \inf_{y \in \A_{s_2 \rho,s_3 \rho}(x_0)} g(y) \qquad \mbox{(by \eqref{eq:g-harnack-step})} \\
		&\leq c g(z_0)   \qquad \mbox{ (since $z_0 \in \A_{s_2 \rho,s_3 \rho}(x_0)$)} \\
		&= c q(z_0) \qquad \mbox{(since $\odometer{z}{t}(z_0) = 0$)} \\
		&\leq c \sup_{w \in \A_{s_2 \rho,s_3 \rho}(x_0)} q(w),
		\end{align*}
		which together with~\eqref{eq:q-inequality} completes the proof. 	
	\end{proof}
	
	\subsection{Very good annuli satisfy the Harnack-type estimate} \label{subsec:very-good-annuli-satisfy-harnack-type-estimate}
	In this section we prove that there are choices of parameters so that for every~$z \in B_{1}$, every annulus $\A_{\rho/2, \rho}(x_0)$ with $\overline{B}_{\rho}(x_0) \subset B_1 \backslash \{z\}$  which is very good also satisfies the Harnack-type property at~$z$.
	\begin{lemma} \label{lemma:very-good-annuli-satisfy-harnack-type-estimate}
		There exists a universal constant $C_3 > 0$ and $b = b(\gamma) \in (0,1)$ so that the following is true
		for every~$z \in B_{1}$.  For every choice of $C_1^\pm , C_2^\pm  > 0$ and $a \in (0,1)$ there exists 
		\begin{itemize}
			\item $N_1 \geq 200$ depending only on $a,\gamma,C_1^-, C_2^+$;
			\item for each $N_0 \geq N_1$, a parameter $\alpha \in (0,1)$ depending on $a, N_0, \gamma$, and $C^-_1$;
		\end{itemize}
		with the following property. If $\rho \in (0,1)$ and $x_0 \in B_1$ are such that $\overline{B}_{\rho}(x_0) \subset B_1 \backslash \{z\}$,  the event $\overline E_\rho(x_0)$ of \ref{enum:very-good-criteria-1},
		\ref{enum:very-good-criteria-2}, and \ref{enum:very-good-criteria-3} occurs, and $\widetilde\Lambda_t$ is a connected component of ${\cluster{z}{t}} \cap B_{\rho}(x_0)$ 
		for which  $\mu_h(\widetilde \Lambda_t \cap \A_{\rho/2, \rho}(x_0)) \leq \alpha \mu_h(\A_{\rho/2, \rho}(x_0))$, then $\widetilde\Lambda_t \cap \overline{B_{\rho/2}(x_0)} = \emptyset$.
	\end{lemma}

	\begin{figure}
		\begin{center}
			\begin{forest}
	[	{\hyperref[lemma:very-good-annuli]{$C_3$}}, name = C3, fill = lightgray
	[ 
	{\hyperref[eq:choice-of-b]{$b$}}, name = b, fill = lightgray
		[{\hyperref[lemma:very-good-annuli]{$a$}}, name = srca, fill = lightgray]  
		 { \draw (.west) node[left]{Lem. \ref{lemma:very-good-annuli}}; }
		[{\hyperref[lemma:very-good-annuli]{$C^{\pm}_1$}}, name = srcc, fill = lightgray]
		[{\hyperref[lemma:very-good-annuli]{$C^{\pm}_2$}}, fill = lightgray,
			[{\hyperref[eq:choice-of-N1]{$N_1 \leq N_0$}}, fill = lightgray, name = N
				[{\hyperref[eq:choice-of-alpha]{$\alpha$}}, name = tgt, fill = lightgray] { \draw (.east) node[right]{Eqn. \eqref{eq:choice-of-alpha}}; }
			] 	{ \draw (.east) node[right]{Eqn. \eqref{eq:choice-of-N1}}; }
		]  
		[	{\hyperref[eq:small-odometer-implies-low-mass]{$C_4$}}, name = C4, fill = lightgray, no edge]
		{ \draw (.east) node[right]{Eqn. \eqref{eq:small-odometer-implies-low-mass}}; }
	]{ \draw (.west) node[left]{Eqn. \eqref{eq:choice-of-b}}; } 	
	]
	\draw[-] (srcc) to[out=south,in=north west] (tgt);
	\draw[-] (srca) to[out=south,in=west] (tgt);
	\draw[-] (srca) to[out=south,in=west] (N);
	\draw[-] (srcc) to[out=south,in=north] (N);
	\draw[-] (C4) to[out=south,in=north] (N);
	{ \draw (.west) node[left]{Lem. \ref{lemma:very-good-annuli}}; }
\end{forest}
		\end{center}
		\caption{Illustration of the choice of constants and dependencies in the proof of Proposition \ref{prop:harnack-type-stronger-property}.
			A line between two constants indicates
			that the downwards constant is chosen in a way which depends directly on the upwards constant. Where the constant is chosen is written directly next to it. Note that $b = b(\gamma)$. Otherwise, dependence on $\gamma \in (0,2)$ and some universal constants is not indicated.  
			$N_0$ can be any number larger than $N_1$. This is done so that there is flexibility later (the proof of Proposition \ref{prop:boundary-measure-zero}) to choose the initial scale $N_0$ to be large.} \label{fig:constant-dependencies}
	\end{figure}
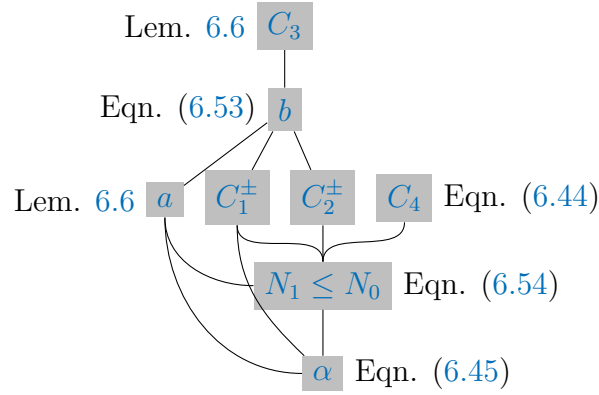

	This leads to the Harnack-type inequality. 
	\begin{proof}[Proof of Propositions \ref{prop:harnack-type-property} and \ref{prop:harnack-type-stronger-property}]
		Combine Lemma \ref{lemma:very-good-annuli} and Lemma \ref{lemma:very-good-annuli-satisfy-harnack-type-estimate}.  
		See Figure \ref{fig:constant-dependencies} for an illustration of how the constants are chosen. 
	\end{proof}

	The proof of Lemma \ref{lemma:very-good-annuli-satisfy-harnack-type-estimate} is purely deterministic. We now outline the proof (in the case~$z = 0$ for convenience) --- also see Figure \ref{fig:shell-decomp-harnack}. 
	As mentioned previously, we will use \ref{enum:very-good-criteria-1} and \ref{enum:very-good-criteria-2}  to switch between $\mu_h$, $\SGrho{\rho}{\cdot}$, and $\Mrho{\rho}{\cdot}$ when convenient. 
	\begin{enumerate}
		\item Set up the iteration by decomposing $\A_{\rho/2, \rho}(x_0)$ into a disjoint, sparse collection of shells $\{S_j\}_{j \geq 0}$
		contained in $\{\A_{\rho/2, d_j}(x_0)\}_{j \geq 0}$ for some infinite geometric sequence $d_j \downarrow d_{\infty} > \rho/2$ (Lemma \ref{lemma:shell-decomposition}).
		\item Show that if $\mu_h(\A_{\rho/2, d_j}(x_0) \cap {\cluster{}{t}})$ is very small, then the supremum of the odometer
		decreases by a multiplicative factor from $\A_{\rho/2, d_j}(x_0)$ to $\A_{\rho/2, d_{j+1}}(x_0)$ (Lemma \ref{lemma:low-mass-implies-small-odometer}).
		\begin{enumerate}
			\item As  $\mu_h(\A_{d_j, \rho}(x_0) \cap {\cluster{}{t}})$ is small, we may use \ref{enum:very-good-criteria-3} 
			to cover $S_j$ by a dense grid of points
			surrounded by a large number $N$ of concentric good annuli.
			\item For each good annulus, apply harmonic comparison, Lemma \ref{lemma:harmonic-odometer-comparison}, together
			with \eqref{eq:good-annulus-harmonic} to see that the supremum of the odometer 
			decreases by a factor of $b$ from one concentric annulus to the next.
			\item  Since the concentric annuli cover $S_j$ and $\odometer{}{t}$ is subharmonic, iterating shows that 
			the odometer decreases by a factor of $b^N$. 
		\end{enumerate}
		\item Show that if $\sup_{\A_{\rho/2, d_{j+1}}(x_0)} \odometer{}{t}$ is small, then $\mu_h(\A_{3r, 5r}(z) \cap {\cluster{}{t}})$ is small for each annulus $\A_{r, 7r}(z) \subset \A_{\rho/2, d_{j+1}}(x_0)$ (using Lemma \ref{lemma:odometer-annuli}).
		
		\item Start with a weak initial bound on the odometer, Lemma \ref{lemma:odometer-bound-lqg-mass}
		and the initial assumption that $\mu_h(\A_{\rho/2, \rho}(x_0) \cap {\cluster{}{t}}) \leq \alpha \mu_h(\A_{\rho/2,\rho}(x_0))$ and iterate the previous two steps to see that $\lim \left( \sup_{\A_{\rho/2, d_j}(x_0)}  \odometer{}{t} \right) \to 0$.
	\end{enumerate}

	\begin{figure}
		\begin{center}
			\includegraphics[width=0.5\textwidth]{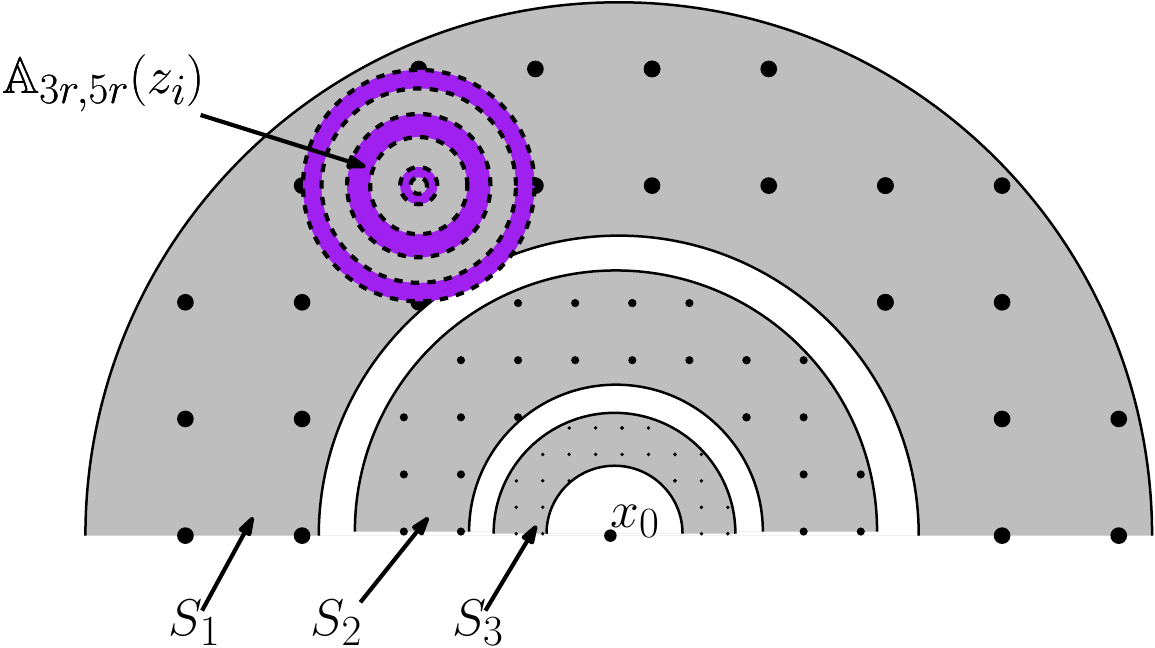}
			\caption{Visual explanation of the proof of Lemma \ref{lemma:very-good-annuli-satisfy-harnack-type-estimate}. 
				We display some of the shells in the shell decomposition, the grid of points in each shell surrounded by good annuli, and some of the good annuli $\A_{r_i, 7 r_i}(z)$ surrounding one of the grid points $z_i$. The shells are displayed in gray, the good annuli are in purple, and the grid of points are black dots. Aspect ratios of annuli are not shown to scale. } \label{fig:shell-decomp-harnack}
		\end{center}
	\end{figure}

	We start with the shell decomposition. 
	\begin{lemma} \label{lemma:shell-decomposition}
		Fix $\rho \in (0,1)$ and $x_0 \in B_1$. Define the collection of shells
		\begin{equation}
		S_j =  \partial B_{d_j}(x_0) + B_{l_j}  \qquad \mbox{for $j \geq 0$}
		\end{equation}
		where 
		\begin{equation}
		l_j = \rho \times 2^{-j/4} \times 2^{-50} \qquad \mbox{for $j \geq 0$}
		\end{equation}
		and
		\begin{equation}
		\begin{aligned}
		d_0 &= \frac34 \rho  \\
		d_j &= d_{j-1} - 32   l_{j-1} \qquad \mbox{for $j \geq 1$}.
		\end{aligned}
		\end{equation}
		For each $N_0 \geq 200$, the shells satisfy the following properties.
		\begin{itemize}
			\item The union of fattened shells is contained in an annulus:
			\begin{equation} \label{eq:contained-in-annulus}
			\bigcup_{j \geq 0} (S_j + B_{8 l_j}) \Subset \A_{\rho/2, \rho}(x_0).
			\end{equation}
			\item Shells are sufficiently far apart:
			\begin{equation} \label{eq:sparse-shells}
			(S_j + B_{8 l_j}) \cap \left( \cup_{j' \neq j} S_{j'} \right) = \emptyset \qquad \mbox{for $j \geq 0$}.
			\end{equation}
			\item Balls centered at points of a shell are contained in an annulus: for each $j \geq 0$, let
			\begin{equation} \label{eq:def-epsilon-j}
			\epsilon_j := 2^{-(N_0 + j)} \quad \text{and} \quad d^+_j := d_j + 8 l_j .
			\end{equation} 
			Then for each $z\in S_j$, 
			\begin{equation} \label{eq:annulus-contained-in-ball}
			B_{8 \epsilon^{1/2}_j \rho}(z) \subset \A_{\rho/2, d^+_j}(x_0) \subset \A_{\rho/2, \rho}(x_0).
			\end{equation}

		\end{itemize}
		
	\end{lemma}
	\begin{proof}
		The first two properties \eqref{eq:contained-in-annulus} and \eqref{eq:sparse-shells} are immediate from the definitions of $d_j$ and $l_j$.
		The second inclusion in \eqref{eq:annulus-contained-in-ball} follows from unpacking the definitions:
		\[
		d^+_j = d_j + 8 l_j \leq d_0 + 8 l_0 =  (3/4 + 8 \times 2^{-100}) \rho < \rho.
		\]
		We now check check the first inclusion in \eqref{eq:annulus-contained-in-ball}. First note
		\begin{equation}
		d_j - l_j \leq |z - x_0| \leq d_j + l_j, \quad \forall z \in S_j
		\end{equation}
		and
		\begin{equation} \label{eq:shell-ball-contained}
		d_j - l_j - 8 \epsilon^{1/2}_j \rho \leq |z'-x_0| \leq d_j + l_j + 8 \epsilon^{1/2}_j \rho, \quad \forall z' \in B_{8 \epsilon^{1/2}_j \rho}(z), z \in S_j.
		\end{equation}
		We claim that the first inclusion in \eqref{eq:annulus-contained-in-ball} follows from the following inequality
		which we verify below,
		\begin{equation} \label{eq:desired-inequality-shell-decomp}
		l_j + 8 \epsilon^{1/2}_j \rho \leq 4 l_j.
		\end{equation}
		Indeed, if \eqref{eq:desired-inequality-shell-decomp} holds, then by \eqref{eq:shell-ball-contained}
		and the definitions of $d^+_j$ and $S_j$, 
		\[
		B_{8 \epsilon^{1/2}_j \rho}(z) \subset S_j + B_{4 l_j} \subset \A_{\rho/2, d^+_j}, \quad \forall z \in S_j.
		\]
		It remains to check \eqref{eq:desired-inequality-shell-decomp}: 
		\begin{align*}
		l_j + 8 \epsilon^{1/2}_j \rho &= \rho \times 2^{-j/4} \times 2^{-50} + 8 \times  2^{-N_0/2-j/2} \times \rho \\
		&\leq \rho \times 2^{-j/4} \left(  2^{-50} + 8 \times 2^{-N_0/2}  \right) \qquad \mbox{(since $2^{-j/4} \geq 2^{-j/2}$ for $j \geq 0$)} \\
		&\leq 2 \times \rho \times 2^{-j/4} \times 2^{-50} \qquad \mbox{(since $N_0 \geq 200$)} \\
		&= 2 l_j.
		\end{align*}
	\end{proof}
	
	We next show that an upper bound for the amount of mass in $\A_{\rho/2, d^+_j}(x_0)$ implies an upper bound for the odometer in $\A_{\rho/2, d^+_{j+1}}(x_0)$. This will be a key input in the induction argument in the proof of Lemma~\ref{lemma:very-good-annuli-satisfy-harnack-type-estimate}.
	
	\begin{lemma} \label{lemma:low-mass-implies-small-odometer}
		Let~$z \in B_1$, and fix an annulus $\A_{\rho/2,\rho}(x_0)$ not containing~$z$ for which \ref{enum:very-good-criteria-3} occurs
		with parameters $a,b \in (0,1)$, $N_0 \geq 200$, and universal constant $C_3 > 0$. Let the shell decomposition $\{S_j, d^+_j\}_{j \geq 0}$ be given by Lemma \ref{lemma:shell-decomposition} and let $\widetilde \Lambda_t$ be a connected component of ${\cluster{z}{t}} \cap B_{\rho}(x_0)$. 
		For $j\geq 0$, let
		\begin{equation} \label{eq:k-j-parameter}
		K_j = a \times \Mrho{\rho}{x_0} \times \epsilon^{\beta^-}_j \qquad \mbox{($\epsilon_j$ from \eqref{eq:def-epsilon-j} and $\beta^-$ from Lemma \ref{lemma:volume-growth})} .
		\end{equation}
		Then for each $j\geq 0$,
		\begin{align} \label{eq:low-mass-small-odometer}
		& \sup_{\substack{r > 0 , z \in B_1 : \\ \A_{r, 7r}(z) \subset \A_{\rho/2, d^+_j}(x_0)}}  \mu_h(\widetilde \Lambda_t \cap \A_{3r, 5r}(z)) \leq K_j \notag \\
		&\qquad\qquad \implies 
		\sup_{\widetilde \Lambda_t \cap \A_{\rho/2, d^+_{j+1}(x_0)}} \odometer{}{t}  \leq  b^{C' \times C_3 \times (N_0 + j)}  \left( \sup_{\tilde \Lambda_t \cap \A_{\rho/2, d^+_j}(x_0)} \odometer{}{t} \right)
		\end{align}
		where $C' > 0$ is a universal constant.
	\end{lemma} 
	\begin{proof} 
		Let $j \geq 0$ be given and let $\widetilde \Lambda_t$ be a connected component of ${\cluster{z}{t}} \cap B_{\rho}(x_0)$. We first check that the assumption 
		\[
		\sup_{\substack{r > 0 , w \in B_1 : \\ \A_{r, 7r}(z) \subset \A_{\rho/2, d^+_j}(x_0)}}    \mu_h(\widetilde \Lambda_t \cap \A_{3 r, 5r}(w)) \leq K_j
		\]
		allows us to use \eqref{eq:good-annulus-harmonic} on sufficiently many annuli covering $S_j$.
		
		In particular, there exists a finite set of points $Z\subset  S_j \cap \frac{\epsilon_j \rho}{100} \Z^2$ so that $\cup_{z\in Z} B_{\rho \epsilon_j}(z)$ 
		covers $S_j$. Fix one such $w \in Z$,
		define
		\[
		N = \lfloor C_3 \log \epsilon_j^{-1/2} \rfloor,
		\]
		and let $r_1, \ldots, r_N \in [\epsilon_j \rho, \epsilon_j^{1/2} \rho] \cap \{\rho 7^{-n}\}_{n \in \N}$ be distinct radii $r_1 > r_2 > \cdots > r_N$ for which the event $E_r(w)$ occurs, as provided by \ref{enum:very-good-criteria-3}.

		{\it Step 1: Small mass in annuli. } \\ 
		For each $r \in \{ r_1, \ldots, r_N\}$,
		\begin{align*}
		\mu_h(\widetilde \Lambda_t \cap \A_{3r,5r }(z)) 
		&\leq K_j \qquad \mbox{($\A_{r, 7r}(z) \subset \A_{\rho/2, d_j^+}(x_0)$ by \eqref{eq:annulus-contained-in-ball} and since $r  \leq \epsilon_j^{1/2} \rho$)} \\
		&= a \times \Mrho{\rho}{x_0} \times \epsilon^{\beta^-}_j  \qquad \mbox{(definition of $K_j$)}\\
		&\leq a \times \frac{\mu_h(B_{\epsilon_j \rho }(z + 4 r e_1))}{\epsilon_j^{\beta^-}} \times  \epsilon_j^{\beta^-} \qquad \mbox{(definition of $\Mrho{\rho}{x_0}$ and \eqref{eq:annulus-contained-in-ball})}  \\
		&= a \times \mu_h(B_{\epsilon_j \rho}(z + 4 r e_1)) \\
		&\leq a \times \mu_h(\A_{3 r, 5 r}(z)) \qquad \mbox{(since $r \geq r_N \geq \epsilon_j \rho$)}.
		\end{align*} 
		Hence, we may use the estimate on the exit probability  \eqref{eq:good-annulus-harmonic} on each such annulus with $Y = \widetilde \Lambda_t \cap \A_{3r,5r }(z)$.
		In fact, we may use it with $Y$ set to be any connected component of  $\widetilde \Lambda_t \cap \A_{3r,5r }(z)$.
		
		{\it Step 2: Small mass in good annulus implies small odometer.} \\
		First note that as $\A_{3 r_1, 5 r_1}(w) \Subset \A_{\rho/2, d^+_j}(x_0)$, 
		\begin{equation} \label{eq:outer-odom-sup}
		\sup_{\partial \A_{3 r_1, 5 r_1}(w) \cap \widetilde \Lambda_t} \odometer{z}{t} \leq \sup_{\widetilde \Lambda_t \cap \A_{\rho/2, d^+_j}(x_0)} \odometer{z}{t}.
		\end{equation}
		For each $x \in \widetilde \Lambda_t \cap \partial B_{4 r_1}(w)$, let $\mathcal{B}^x$ denote an independent Brownian motion started at $x$ and $\Lambda^x_t$ 
		the connected component of $\A_{3r_1, 5r_1}(w) \cap {\cluster{z}{t}}$ containing $x$. 
		Note that for each such $x$, we have $\Lambda_t^x \subset \widetilde \Lambda_t$ (and hence $\mu_h(\Lambda_t^x) \leq \mu_h(\tilde \Lambda_t))$.
		We use this to see that for each $x \in \widetilde \Lambda_t \cap \partial B_{4 r_1}(w)$, 
		\begin{align*}
		&\odometer{z}{t}(x) \\
		&\leq \left( \sup_{\partial \A_{3r_1, 5r_1}(w) \cap \overline \Lambda^x_t} \odometer{z}{t} \right)  \P[\mbox{$\mathcal{B}^x$ exits  $\A_{3r_1 , 5r_1}(w)$ before  $\Lambda^x_t$} | h]
		\quad \mbox{(Lemma \ref{lemma:harmonic-odometer-comparison} with $A = \A_{3r_1, 5r_1}(w)$)} \\
		&\leq b \left( \sup_{\partial \A_{3r_1, 5r_1}(w) \cap \overline \Lambda^x_t} \odometer{z}{t} \right)  \quad \mbox{(\eqref{eq:good-annulus-harmonic} with $Y = \Lambda^x_t$)} \\
		&\leq b  \left( \sup_{  \A_{3 r_1, 5 r_1}(w) \cap  \widetilde \Lambda_t}   \odometer{z}{t} \right) \quad \mbox{($\Lambda_t^x \subset \widetilde \Lambda_t$)}.
		\end{align*}

		As $\widetilde \Lambda_t$ is a connected component 
		of ${\cluster{z}{t}} \cap \A_{\rho/2, \rho}(x_0)$,  $\widetilde \Lambda_t \cap B_{4 r_1}(w)$ is a union 
		of connected components of ${\cluster{z}{t}} \cap B_{4 r_1}(w)$. 
		The same argument in the proof of Lemma \ref{lemma:odometer-annuli} shows that $\odometer{z}{t} 1\{ \cdot \in \widetilde \Lambda_t\}$
		is subharmonic in $B_{\rho}(x_0)$. 
		
		This together with the maximum principle and the most recent indented inequality shows that
		\[
		\sup_{\widetilde \Lambda_t \cap B_{4 r_1}(w)} \odometer{z}{t} \leq 	\sup_{\widetilde \Lambda_t \cap \partial B_{4 r_1}(w)} \odometer{}{t} \leq  b \left( \sup_{ \A_{3 r_1, 5 r_1}(w) \cap \widetilde \Lambda_t} \odometer{z}{t} \right).
		\]
		Since the next annulus, $\A_{r_2, 7 r_2}(w) \subset B_{4 r_1}(w)$ ($r_2 \leq r_1/7$ by construction), this implies that 
		\[
		\sup_{  \A_{r_2, 7 r_2}(w) \cap \widetilde \Lambda_t } \odometer{z}{t} \leq  b  \left( \sup_{  \A_{3 r_1, 5 r_1}(w) \cap \widetilde \Lambda_t} \odometer{z}{t} \right) .
		\]

		{\it Step 3: Iterate. } \\
		We have shown in Step 1 that each annulus $\A_{r, 7 r}(w)$ for $r \in \{r_1, \ldots, r_N\}$ satisfies the conditions required to use \eqref{eq:good-annulus-harmonic}, so we may iterate Step 2 $(N-1)$ times, then apply~\eqref{eq:outer-odom-sup}, to get
		\begin{equation} \label{eq:first-iteration}
		\sup_{\widetilde \Lambda_t \cap \A_{r_N, 7 r_N}(w)} \odometer{z}{t} 
		\leq b^{N-1}  \left( \sup_{  \A_{3 r_1, 5 r_1}(w)  \cap \widetilde \Lambda_t} \odometer{z}{t} \right) 
		\leq b^{N-1} \left( \sup_{\widetilde \Lambda_t \cap \A_{\rho/2, d^+_j}(x_0)} \odometer{z}{t} \right).
		\end{equation}
		Since the estimate \eqref{eq:first-iteration} holds for all $w\in Z$, $r_N \geq \rho \epsilon_j$, and $\cup_{w\in Z} B_{\rho \epsilon_j}(w ) \supseteq S_j$, the maximum principle applied in $B_{r_N}(w)$ for each $w\in Z$ gives
		\begin{equation} \label{eq:odom-sup-S}
		\sup_{\widetilde \Lambda_t \cap S_j} \odometer{z}{t}  \leq b^{N-1} \left( \sup_{\widetilde \Lambda_t \cap \A_{\rho/2, d^+_j}(x_0)} \odometer{z}{t}  \right).
		\end{equation}
		As previously mentioned, $v_t 1\{ \cdot \in \widetilde \Lambda_t\}$ is subharmonic in $B_{\rho}(x_0)$. 
		Therefore, as $S_j$ disconnects $\partial B_{d^+_{j+1}}(x_0)$ from $\partial B_{\rho}(x_0)$, 
		\[
		\sup_{\widetilde \Lambda_t \cap \A_{\rho/2, d^+_{j+1}}(x_0)} \odometer{z}{t} \leq \sup_{\widetilde \Lambda_t \cap S_j} \odometer{z}{t}.
		\]
		This combined with~\eqref{eq:odom-sup-S} gives
		\begin{equation} \label{eq:odom-sup-end}
		\sup_{\widetilde \Lambda_t \cap \A_{\rho/2, d^+_{j+1}}(x_0)} \odometer{z}{t} \leq b^{N-1} \left( \sup_{\widetilde \Lambda_t \cap \A_{\rho/2, d^+_j}(x_0)} \odometer{z}{t}  \right) .
		\end{equation} 
		We now recall that $N = \lfloor C_3 \log \epsilon_j^{-1/2} \rfloor$ and $\epsilon_j = 2^{-(N_0 + j)}$~\eqref{eq:def-epsilon-j}. Hence the lemma statement follows from~\eqref{eq:odom-sup-end}.
	\end{proof}
	
	\begin{figure}
		\begin{center}
			\begin{tikzpicture}[auto, node distance=0.6cm,>=latex,block/.style={draw, fill=white, rectangle}, minimum height=1.2em, minimum width=6em]
\node[block, text width = 7 cm] (A1) {Decompose annulus $\A_{\rho/2,\rho}(x_0)$ into shells,  $\{S_j\}_{j \geq 0}$, given by Lemma \ref{lemma:shell-decomposition}};
\node[block, below=of A1, text width = 15 cm] (A2) {
	\makecell[c]{Assume cluster in annulus has small mass relative to the size of the annulus \eqref{eq:ind-assumption} \\$\implies$ Base case:
	small mass in shell $S_0$ \eqref{eq:step-4-initial-mass-bound}}
};
\node[block, below= of A2] (A3) {Base case: small odometer in shell $S_{1}$, \eqref{eq:base-case-odometer-size}};
\node[below = of A3, node distance=0cm] (A5) {}; 
\node[block, below= of A3, text width = 10cm, align = center, node distance = 1cm] (A6) {\makecell[c]{Small mass in shell $S_{j}$, \eqref{eq:step-4-induct-mass}  for $j$\\Small odometer in shell $S_{j+1}$, \eqref{eq:step-4-induct-odometer} for $(j+1)$}};
\node[block, below = of A6] (A7) {Small mass in shell $S_{j+1}$, \eqref{eq:step-4-induct-mass} for $(j+1)$}; 
\node[block, below = of A7] (A8) {Small odometer in shell $S_{j+2}$, \eqref{eq:step-4-induct-odometer} for $(j+2)$}; 
\node[ below = of A8] (A9) {}; 
\node[block, below = of A8] (A10) {Zero odometer on $\partial B_{\rho/2}$, \eqref{eq:zero-odometer-in-inner-ball}};
\draw[->] (A2) -- node[pos=0.5,right]{Lem.~\ref{lemma:low-mass-implies-small-odometer} and Lem.~\ref{lemma:odometer-bound-lqg-mass}}(A3); 
\draw[dashed, ->] (A3) -- node[pos=0.5,right]{Inductive hypothesis}(A6); 
\draw[->] (A6) -- node[pos=0.5,right]{Eqn. \eqref{eq:small-odometer-implies-low-mass}}(A7); 
\draw[->] (A7) -- node[pos=0.5,right]{Lem.~\ref{lemma:low-mass-implies-small-odometer}}(A8)
coordinate[midway] (aux); 
\draw[-] (A6.west) |- ($(A6.west) - (0.1,0)$)  |- ($ (A8.north)+(0,0.3) $)
-| (aux.west);
\draw[dashed, ->] (A8) -- node[pos=0.5,right]{Iterate and send $j \to \infty$}(A10); 
\end{tikzpicture}
		\end{center}
		\caption{Schematic outline of the proof of Lemma \ref{lemma:very-good-annuli-satisfy-harnack-type-estimate}}
	\end{figure}
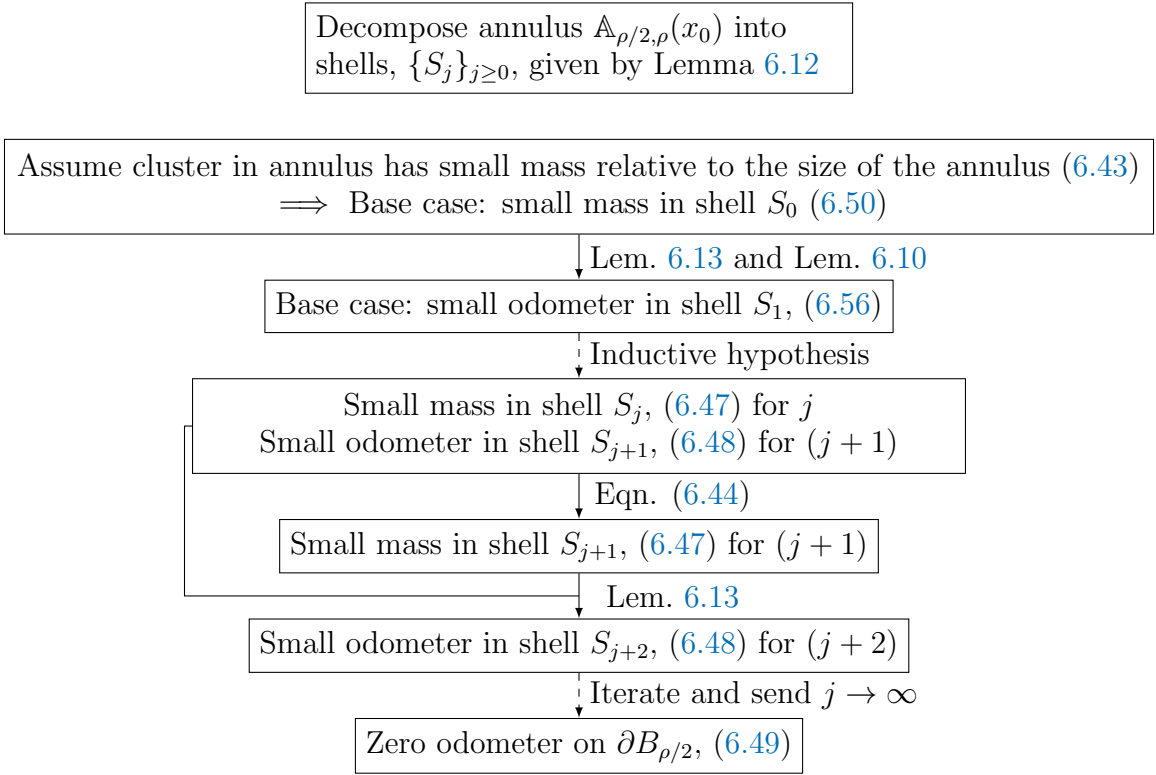

	We conclude with the proof of the desired claim.

	\begin{proof}[Proof of Lemma \ref{lemma:very-good-annuli-satisfy-harnack-type-estimate}]
		Let $C^{\pm}_1 > 0, C^{\pm}_2 >0$, corresponding to \ref{enum:very-good-criteria-1} and \ref{enum:very-good-criteria-2},
		and $a \in (0,1)$, $C_3 > 0$, corresponding to \ref{enum:very-good-criteria-3}, be given.  Fix $b = b(\gamma) \in (0,1)$, $N_0 \geq N_1(a,C^-_1, C^+_2) \geq 200$, and then $\alpha = \alpha(a, N_0, \gamma, C^-_1) \in (0,1)$ which will be specified 
		in \eqref{eq:choice-of-b}, \eqref{eq:choice-of-N1}, and \eqref{eq:choice-of-alpha} respectively.

		Fix~$z \in B_1$, $\rho \in (0,1)$ and $x_0 \in B_1$ such that $\overline E_\rho(x_0)$ occurs,~$\overline{B}_{\rho} \subset B_1 \setminus \{z\}$ and let $\widetilde \Lambda_t$ be a connected component of ${\cluster{z}{t}} \cap \A_{\rho/2, \rho}(x_0)$ with $\overline{\widetilde \Lambda_t} \cap \partial B_{\rho}(x_0) \neq \emptyset$. We assume that
		\begin{equation} \label{eq:ind-assumption}
		\mu_h(\widetilde{\Lambda}_t \cap \A_{\rho/2, \rho}(x_0)) \leq \alpha \mu_h(\A_{\rho/2, \rho}(x_0)) 
		\end{equation}
		and we seek to show that $\widetilde{\Lambda}_t \cap \partial B_{\rho/2}(x_0) = \emptyset$.
		
		Start by decomposing the annulus via the shell decomposition  $\{S_j, d^+_j\}_{j \geq 0}$ given by Lemma \ref{lemma:shell-decomposition}.
		We will iteratively apply Lemma \ref{lemma:low-mass-implies-small-odometer} and Lemma \ref{lemma:odometer-annuli}. 
		Suppose $r>0$ and $w \in B_1$ such that $\A_{r,7r}(w) \subset \A_{\rho/2, d^+_{j+1}}(x_0)$.
		As $\widetilde \Lambda_t$ is a connected component of $\A_{\rho/2,\rho}(x_0)$ and $\A_{r,7r}(w) \Subset \A_{\rho/2,\rho}(x_0)$, we have that $\A_{r, 7r}(w) \cap \widetilde \Lambda_t$ is a union of connected components of ${\cluster{z}{t}} \cap \A_{r, 7r}(w)$. Hence, we may use the following consequence of Lemma \ref{lemma:odometer-annuli}: 
		\begin{align} \label{eq:small-odometer-implies-low-mass}
		&\mu_h(\A_{3r, 5r}(w) \cap \widetilde \Lambda_t) \leq C_4 \sup_{\widetilde \Lambda_t \cap \A_{\rho/2, d^+_{j+1}}(x_0)} \odometer{z}{t} \notag\\
		&\qquad\qquad \forall j \geq 0 \mbox{ and all $r>0$ and $w\in B_1$ such that } \A_{r, 7r}(w) \subset \A_{\rho/2, d^+_{j+1}}(x_0),
		\end{align}
		where $C_4 > 0$ is a universal constant.

		{\it Step 1: Choose $\alpha$ and set up iteration.}  \\
		We set up the iteration of Lemma \ref{lemma:low-mass-implies-small-odometer} and Lemma \ref{lemma:odometer-annuli}.
		Start by choosing 
		\begin{equation} \label{eq:choice-of-alpha}
		\alpha := a \times 2^{-N_0 \beta^-} \times \frac{1}{C^-_1 \vee 1}
		\end{equation}
		so that our assumption~\eqref{eq:ind-assumption} implies
		\begin{equation} \label{eq:step-4-initial-assumption}
		\mu_h(\widetilde \Lambda_t \cap \A_{\rho/2, \rho}(x_0)) \leq a \times 2^{-N_0 \beta^-} \times \frac{1}{C^-_1 \vee 1}  \times \mu_h(\A_{\rho/2, \rho}(x_0)).
		\end{equation}
		Let $K_j$ be as in Lemma~\ref{lemma:low-mass-implies-small-odometer}. We will show 
		\begin{equation} \label{eq:step-4-induct-mass}
		\sup_{\substack{r > 0 , w \in B_1 : \\ \A_{r, 7r}(w) \subset \A_{\rho/2, d^+_j}(x_0)}}    \mu_h(\widetilde \Lambda_t \cap \A_{3 r, 5r}(w)) \leq K_j, \quad \forall j \geq 0
		\end{equation}
		and
		\begin{equation} \label{eq:step-4-induct-odometer}
		\sup_{\widetilde \Lambda_t \cap \A_{\rho/2, d^+_{j}(x_0)}} \odometer{z}{t} \leq \frac{1}{C_4 \vee 1} K_j, \quad \forall j \geq 1 \quad 
		\end{equation}
		where the universal constant $C_4$ is from \eqref{eq:small-odometer-implies-low-mass}.
		
		Once we show this, then we may take $j \to \infty$ in \eqref{eq:step-4-induct-odometer} and use that $\lim_{j\to \infty} K_j = 0$ to get that $\odometer{z}{t} $ is zero on $\widetilde\Lambda_t \cap \partial B_{d_\infty^+}(x_0)$, where $ d_\infty^+:= \lim_{j\to\infty} d_j^+ \in (\rho/2,\rho)$. 
		Since $\odometer{z}{t} 1\{ \cdot \in {\widetilde \Lambda_t}\}$ is subharmonic in $B_{\rho}(x_0)$, this implies that 
		\begin{equation} \label{eq:zero-odometer-in-inner-ball}
		\sup_{\widetilde \Lambda_t \cap \overline B_{\rho/2}(x_0)} \odometer{z}{t} = 0,
		\end{equation}
		implying the desired statement by the definition of ${\cluster{z}{t}}$.

		Hence, it remains to prove \eqref{eq:step-4-induct-mass} and \eqref{eq:step-4-induct-odometer}.
		Our strategy is to induct on $j \geq 0$ and show the following chain of implications:
		\[
		\{ \mbox{\eqref{eq:step-4-induct-mass} for $j$ and \eqref{eq:step-4-induct-odometer} for $(j+1)$} \} \implies 	\{ \mbox{\eqref{eq:step-4-induct-mass} for $(j+1)$ and \eqref{eq:step-4-induct-odometer} for $(j+2)$} \}.
		\]
		We start with the base case.
		
		{\it Step 2: Base case $j = 0$.} \\
		By \ref{enum:very-good-criteria-1} and the inequality \eqref{eq:step-4-initial-assumption} 
		\begin{equation} \label{eq:step-4-initial-mass-bound}
		\mu_h(\widetilde \Lambda_t \cap \A_{\rho/2, \rho}(x_0)) \leq a \times \Mrho{\rho}{x_0} \times 2^{-N_0 \beta^-}  = K_0,
		\end{equation}
		which is \eqref{eq:step-4-induct-mass} for $j = 0$.
		Since $\alpha \in (0,1)$, by \eqref{eq:ind-assumption},
		\[
		\A_{\rho/2, \rho}(x_0) \cap \widetilde{\Lambda}_t^c \neq \emptyset.
		\]
		Therefore, by Lemma 
		\ref{lemma:odometer-bound-lqg-mass}
		\begin{equation} \label{eq:step-4-odometer-upper-bound}
		\sup_{\A_{\rho/2, \rho}(x_0)} \odometer{z}{t} \leq c \times \SGrho{\rho}{x_0} 	\leq c \times C^+_2 \times \frac{1}{C^-_1 \vee 1} \times \Mrho{\rho}{x_0} 
		\end{equation}
		with the latter inequality following from \ref{enum:very-good-criteria-1} and \ref{enum:very-good-criteria-2}. The inequality \eqref{eq:step-4-initial-mass-bound} allows us to use Lemma \ref{lemma:low-mass-implies-small-odometer}
		to see that 
		\begin{equation} \label{eq:first-step-initial-bound}
		\sup_{\widetilde \Lambda_t \cap \A_{\rho/2, d^+_{1}}(x_0)} \odometer{z}{t} \leq b^{C' \times C_3 \times N_0} \left(	\sup_{\widetilde \Lambda_t \cap \A_{\rho/2, d^+_{0}}(x_0)} \odometer{z}{t} \right).
		\end{equation}
		Now, pick  
		\begin{equation} \label{eq:choice-of-b}
		b := 2^{-\frac{2 \beta^-}{(C' \times C_3) \vee 1}}  .
		\end{equation}
		We emphasize that $b$ depends only on $\gamma$ (through $\beta^-$) since $C_3$ is a universal constant.
		We also choose $N_1$ sufficiently large so that
		\begin{equation} \label{eq:choice-of-N1}
		2^{-\beta^- N_1} \leq a \times 2^{-\beta^-} \times \frac{\frac{1}{C_4 \vee 1}}{(c \times C^+_2 \times \frac{1}{C^-_1 \vee 1}) \vee 1}.
		\end{equation}
		Fix some $N_0 \geq N_1$. With these choices of $b$ and $N_0$, we have 
		\begin{equation} \label{eq:choice-of-b-and-m0}
		b^{C' \times C_3 \times N_0} \leq 
		2^{-N_0 \times 2 \beta^-} \leq
		\frac{ \frac{1}{C_4 \vee 1} \times a 2^{-\beta^-(N_0+1)}}{(c \times C^+_2 \times \frac{1}{C^-_1 \vee 1}) \vee 1}.
		\end{equation}
		Hence, by \eqref{eq:first-step-initial-bound} followed by \eqref{eq:step-4-odometer-upper-bound} and~\eqref{eq:choice-of-b-and-m0}, 
		
		\begin{equation} \label{eq:base-case-odometer-size}
		\sup_{\widetilde \Lambda_t \cap \A_{\rho/2, d^+_{1}(x_0)}} \odometer{z}{t} \leq b^{C' \times C_3 \times N_0} \left(	\sup_{\widetilde \Lambda_t \cap \A_{\rho/2, \rho}(x_0)} \odometer{z}{t} \right) \leq \frac{1}{C_4 \vee 1} a 2^{-\beta^-(N_0+1)} \Mrho{\rho}{x_0}	= \frac{1}{C_4 \vee 1} K_1,
		\end{equation}
		which is \eqref{eq:step-4-induct-odometer} for $j = 1$.

		{\it Step 3: Inductive step, $j \to (j+1)$.} \\
		If \eqref{eq:step-4-induct-odometer} holds for $(j+1)$, then 
		by \eqref{eq:small-odometer-implies-low-mass}, we have \eqref{eq:step-4-induct-mass} for $(j+1)$. 
		It remains to show that 
		\[
		\{\mbox{\eqref{eq:step-4-induct-mass} for $(j+1)$ and \eqref{eq:step-4-induct-odometer} for $(j+1)$} \} \implies 	\{\mbox{\eqref{eq:step-4-induct-odometer} for $(j+2)$} \}.
		\]
		This is similar to the argument of the base case however we will not need the full strength of the inductive step (unlike the base case). In particular we will use the very crude bound $b^{C' \times C_3 \times (N_0 + j)} \leq 2^{-\beta^-}$.
		
		By \eqref{eq:step-4-induct-mass} for $(j+1)$
		we may use Lemma \ref{lemma:low-mass-implies-small-odometer}. Hence, 
		\begin{align*}
		\sup_{	\A_{\rho/2, d^+_{j+2}}(x_0)} \odometer{z}{t} 
		&\leq b^{C' \times C_3 \times (N_0 + j + 1)} \sup_{	\A_{\rho/2, d^+_{j+1}}(x_0)} \odometer{z}{t} \quad \mbox{(by \eqref{eq:low-mass-small-odometer})} \\
		&\leq b^{C' \times C_3 \times (N_0 + j)} \frac{1}{C_4 \vee 1} K_{j+1} \quad \mbox{(by \eqref{eq:step-4-induct-odometer} for $(j+1)$)}  \\
		&\leq 2^{-\beta^- } \times \frac{1}{C_4 \vee 1} K_{j+1} \qquad \mbox{(by \eqref{eq:choice-of-b-and-m0})} \\
		&= \frac{1}{C_4 \vee 1} K_{j+2} \quad \mbox{(by \eqref{eq:k-j-parameter})}
		\end{align*}
		which is \eqref{eq:step-4-induct-odometer} for $(j+2)$, completing the proof.
	\end{proof}

	\section{Upper bound and continuity} \label{sec:upper-bound}
	
	In this section we prove, using the Harnack-type estimate Proposition \ref{prop:harnack-type-property}, that clusters do not immediately exit the unit ball and are in fact H\"{o}lder-continuous in the parameter $t$. 
	We start with the upper bound. 
	
	\begin{prop} \label{prop:upper-bound-polynomially-high-prob}
		For each $r \in (0,1)$, on an event which occurs with polynomially high probability as $T \to 0$, 
		\[
		\overline{{{\cluster{z}{t}}}} \subset B_{r/2}(z) ,\quad \forall t \leq T \, , \quad \forall z \in B_{1/2} \, . 
		\] 
	\end{prop}
	
	\begin{proof}
		Fix $r \in (0,1)$ and let $\overline E_{\rho}(x_0)$ and $\alpha \in (0,1)$ be the event and parameter from Proposition \ref{prop:harnack-type-property}.
		By Proposition~\ref{prop:harnack-type-property}, it holds with polynomially high probability as $\epsilon \to 0$ that for each $x\in (B_{1 + \epsilon} \backslash B_{10 \sqrt{\epsilon}}) \cap \frac{\epsilon}{100} \Z^2$, there exists $\rho_x \in [\epsilon,\epsilon^{1/2}]$ such that $\overline E_{\rho_x}(x)$ occurs. Henceforth assume that this is the case for some $\epsilon \in (0, 2^{-20} r^2)$.
		
		Let $\mathbf{X} \subset  (B_{1 + \epsilon} \backslash B_{10 \sqrt{\epsilon}}) \cap \frac{\epsilon}{100} \Z^2$ be a set such that
		\begin{equation}
		\bigcup_{x \in \mathbf{X}} \A_{\rho_x/2, \rho_x}(x) \Subset B_{1} \backslash B_{r/4}(z)
		\end{equation}
		and
		\begin{equation} \label{eq:cover-inner-ball}
		\partial B_{r/2}(z) \subset \bigcup_{x \in \mathbf{X}} B_{\rho_x/2}(x).
		\end{equation}
		This is possible since for each $x$, we have $\epsilon \leq \rho_x$ and $\rho_x^{1/2} \leq r/100$ --- see Figure \ref{fig:upper-bound}.

		By Lemma~\ref{lemma:volume-growth}, it holds with polynomially high probability as $\epsilon \to 0$ that
		\begin{equation} \label{eq:upper-bound-annuli}
		\mu_h(\A_{\rho_x/2,\rho_x}(x )) \geq \rho_x^{2 \beta^-}, \quad \forall x \in \mathbf{X}, \quad \mbox{$\beta^-$ from Lemma \ref{lemma:volume-growth}}.
		\end{equation}
		Henceforth assume that $\epsilon \in  (0, 2^{-20} r^2)$ is such that~\eqref{eq:upper-bound-annuli} holds.
		
		Now choose $T \leq \alpha \epsilon^{2 \beta^-}$. Then for each $x \in \mathbf{X}$,~$z \in B_{1/2}$, and each $t \leq T$,
		\begin{align*}
		\mu_h({\cluster{z}{t}} \cap \A_{\rho_x/2,\rho_x}(x)) &\leq \mu_h({\cluster{z}{t}}) \quad \mbox{(monotonicity, Lemma \ref{lemma:monotonicity})} \\
		&\leq  \alpha \epsilon^{2 \beta^-} \qquad \mbox{(choice of $T$ and Lemma \ref{lemma:conservation-of-mass})} \\
		&\leq \alpha \rho_x^{2 \beta^-} \qquad \mbox{($\rho_x \geq \epsilon$)} \\
		&\leq \alpha \mu_h(\A_{\rho_x/2,\rho_x}(x )) \qquad \mbox{(by \eqref{eq:upper-bound-annuli})}.
		\end{align*}
		Thus, we may apply  Proposition \ref{prop:harnack-type-property} and \eqref{eq:cover-inner-ball} to see that 
		\begin{equation}
		\sup_{\partial B_{r/2}(z)} \odometer{z}{t} = 0, \quad \forall t \leq T \, , \quad \forall z \in B_{1/2} \, . 
		\end{equation}
		Since $\sup_{\partial B_1} \odometer{z}{t} = 0$ by Lemma \ref{lemma:conservation-of-mass} and $\odometer{z}{t}$ is subharmonic away from~$B_{r/2}(z)$, this implies 
		\begin{equation} \label{eq:upper-bound-odometer}
		\sup_{B_1 \setminus B_{r/2}(z)} \odometer{z}{t} = 0, \quad \forall t \leq T \, , \quad \forall z \in B_{1/2} \, . 
		\end{equation}
		By the definition~\eqref{eq:non-co-set} of ${\cluster{}{t}}$, \eqref{eq:upper-bound-odometer} implies that ${\cluster{z}{t}} \subset B_{r/2}(z)$. 
		We conclude by recalling that both the condition in the first paragraph and~\eqref{eq:upper-bound-annuli} hold with polynomially high probability as $\epsilon\to 0$.
	\end{proof}

	\begin{figure}
		\begin{center}
			\includegraphics[width=0.4\textwidth]{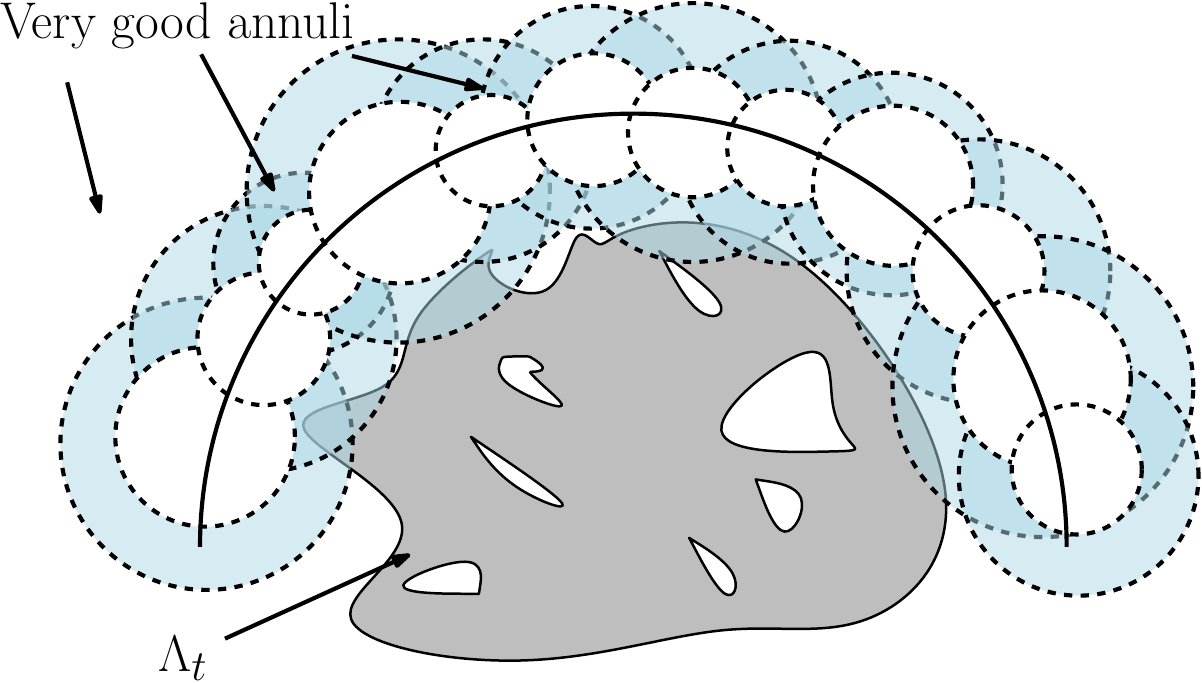}
			\caption{Covering of $\partial B_{r/2}$ by very good annuli as in the proof of Proposition \ref{prop:upper-bound-polynomially-high-prob}. The cluster ${\cluster{}{t}}$ is displayed in gray with a 
				solid black border and the good annuli are light gray with a dotted black border. } 
			\label{fig:upper-bound}
		\end{center}
	\end{figure}

	We next prove continuity of the clusters, using a similar argument as in the proof of Proposition~\ref{prop:upper-bound-polynomially-high-prob}.   
	
	\begin{prop} \label{prop:continuity-of-clusters}
		The following occurs on an event of probability 1. 
		For each $t > 0$ and~$z \in B_1$ such that ${\cluster{z}{t}} \Subset B_1$, for all $\epsilon > 0$ sufficiently small, depending on $t$,
		\[
		{\cluster{z}{t + \epsilon}} \Subset {\cluster{z}{t}} + B_{\delta}(z) ,\quad \text{for $\delta = C \epsilon^{1/(4\beta^-)}$}
		\]
		where $\beta^- > (2+\gamma)^2/2$ is from Lemma \ref{lemma:volume-growth} and  $C = C(\beta^-) >0$ is a deterministic constant. 
	\end{prop}
	
	\begin{proof}
		Let $\overline E_{\rho}(x_0)$ and $\alpha \in (0,1)$ be the event and parameter from Proposition \ref{prop:harnack-type-property}.
		By Proposition~\ref{prop:harnack-type-property}, it holds with polynomially high probability as $\delta \to 0$ that for each $x\in (B_{1 + \delta} \backslash B_{10 \sqrt{\delta}}) \cap \frac{\delta}{100} \Z^2$, there exists $\rho_x = \rho_x(\delta) \in [\delta,\delta^{1/2}]$ such that $\overline E_{\rho_x}(x)$ occurs. 
		By Lemma~\ref{lemma:volume-growth}, it also holds with polynomially high probability as $\delta \to 0$ that 
		\begin{equation} \label{eq:volume-growth-lower-bound}
		\mu_h\left( \A_{\rho_x/2, \rho_x}(x ) \right)  \geq \rho_x^{2 \beta^-}, \quad \forall x \in (B_{1 + \delta } \backslash B_{10 \sqrt{\delta }})  \cap \frac{\delta }{100} \Z^2 .
		\end{equation}
		By the Borel-Cantelli lemma, a.s.\ there exists $M_0$ sufficiently large such that the preceding two conditions hold for each $\delta \in \{2^{-n}\}_{n \geq M_0}$.

		Let $t > 0$, and $z \in B_1$ such that ${\cluster{z}{t}} \Subset B_1$. By Proposition \ref{prop:lower-bound}, there exists $\epsilon_0(t) > 0$ so that $B_{\epsilon_0}(z) \subset {\cluster{z}{t}}$.
		Hence, since ${\cluster{z}{t}} \Subset B_1$, by taking $M_0$ possibly larger (depending on $\epsilon_0$), we can arrange that
		\begin{equation} \label{eq:boundary-inside}
		({\cluster{z}{t}} + B_{20 \sqrt{\delta }}) \backslash ({\cluster{z}{t}} + B_{\sqrt{\delta }}) \Subset  B_1 \backslash B_{10 \sqrt{\delta }}(z), \quad \forall \delta  \in \{2^{-n}\}_{n \geq M_0} .
		\end{equation}
		By~\eqref{eq:boundary-inside} and the fact that $\rho_x \in [\delta , \delta^{1/2}]$, we obtain that for each $\delta  \in \{2^{-n}\}_{n \geq M_0}$ there exists $\mathbf{X} = \mathbf{X}(\delta) \subset (B_{1 + \delta } \backslash B_{10 \sqrt{\delta }})  \cap \frac{\delta }{100} \Z^2$ for which
		\begin{equation} \label{eq:cover-continuous}
		({\cluster{z}{t}} + B_{4 \sqrt{\delta }}) \backslash ({\cluster{z}{t}} + B_{3 \sqrt{\delta }}) \Subset \bigcup_{x \in \mathbf{X}} B_{\rho_x/2}(x ) \Subset B_{1- 10 \sqrt{\delta}}
		\end{equation}
		and 
		\begin{equation} \label{eq:cover-in-complement}
		\bigcup_{x \in \mathbf{X}} B_{\rho_x}(x) \Subset ({{{\cluster{z}{t}}}})^c
		\end{equation}
		for each $\delta  \in \{2^{-n}\}_{n \geq M_0}$.

		As ${\cluster{z}{t}} \Subset B_1$, by Lemma \ref{lemma:conservation-of-mass} we have $\mu_h({\cluster{z}{t}}) = t$.
		Thus, by monotonicity, Lemma \ref{lemma:monotonicity}, 
		\[
		\mu_h({\cluster{z}{t + \epsilon}} \backslash {\cluster{z}{t}}) \leq \epsilon, \quad \forall \epsilon > 0.
		\] 
		Therefore, by \eqref{eq:cover-in-complement}, whenever $\delta  \in \{2^{-n}\}_{n \geq M_0}$ we have
		\begin{equation} \label{eq:expand-cluster-mass}
		\mu_h(\A_{\rho_x/2,\rho_x}(x ) \cap {\cluster{z}{t+\epsilon}}) \leq \epsilon, \quad \forall x \in \mathbf{X}, \quad \forall \epsilon > 0.
		\end{equation}
		
		Now set 
		\[
		\epsilon = \epsilon(\delta) := \alpha \delta^{2 \beta^-},
		\]
		(where $\alpha$ is as in Proposition \ref{prop:harnack-type-property}). Then for each $\delta  \in \{2^{-n}\}_{n \geq M_0}$ and each $x \in \mathbf{X}$, we have
		\begin{align*}
		\mu_h(\A_{\rho_x/2,\rho_x}(x) \cap {\cluster{z}{t+\epsilon }}) &\leq \alpha \delta^{2 \beta^-} \quad \text{(by \eqref{eq:expand-cluster-mass})} \\
		&\leq \alpha \rho_x^{2 \beta^-} \quad \mbox{($\rho_x \geq \delta $)} \\
		&\leq \alpha \mu_h(\A_{\rho_x/2,\rho_x}(x)) \quad \mbox{(by \eqref{eq:volume-growth-lower-bound})}.
		\end{align*}
		Therefore, as the event $\overline E_{\rho_x }(x )$ occurs for each $x \in \mathbf{X}$, by Proposition \ref{prop:harnack-type-property},
		\[
		B_{\rho_x/2}(x ) \cap {\cluster{z}{t+\epsilon }} = \emptyset, \quad \forall x \in \mathbf{X}, \quad \forall \delta  \in \{2^{-n}\}_{n \geq M_0} .
		\]
		This implies by \eqref{eq:cover-continuous} that 
		\begin{equation} \label{eq:zero-on-shell}
		\sup_{({\cluster{z}{t}} + B_{4 \sqrt{\delta}}) \backslash ({\cluster{z}{t}} + B_{3 \sqrt{\delta}})} \odometer{z}{t+\epsilon } = 0, \quad \forall \delta  \in \{2^{-n}\}_{n \geq M_0}.
		\end{equation}
		Since $\sup_{\partial B_1} \odometer{z}{t+\epsilon } = 0$ by Lemma \ref{lemma:conservation-of-mass} and $\odometer{z}{t+\epsilon }$ is subharmonic away from~$z$, \eqref{eq:zero-on-shell} implies
		\begin{equation}
		\sup_{B_1 \backslash ({\cluster{z}{t}} + B_{3 \sqrt{\delta }})} \odometer{z}{t+\epsilon } = 0, \quad \forall \delta  \in \{2^{-n}\}_{n \geq M_0} .
		\end{equation}
		By the definition~\eqref{eq:non-co-set} of $\cluster{z}{t+\epsilon }$, this implies that
		\begin{equation}
		{\cluster{z}{t+\epsilon }} \Subset {\cluster{z}{t}} + B_{4\sqrt{\delta}}, \quad \forall \delta  \in \{2^{-n}\}_{n \geq M_0} .
		\end{equation}
		Recalling that $\epsilon = \alpha \delta^{2 \beta^-}$, this concludes the proof. 
	\end{proof}

	For completeness, we indicate how the above two results imply Proposition~\ref{prop:upper-bound} and Proposition \ref{prop:continuity}. 
	
	\begin{proof}[Proof of Propositions~\ref{prop:upper-bound} and~\ref{prop:continuity}]
		Proposition~\ref{prop:upper-bound} follows immediately from Proposition~\ref{prop:upper-bound-polynomially-high-prob}. 
		
		The fact that the clusters continuously increase in $t$, \eqref{eq:continuity}, is implied by Proposition \ref{prop:continuity-of-clusters}. 
		We show here that Proposition \ref{prop:upper-bound-polynomially-high-prob} implies that the clusters decrease to the center point, \eqref{eq:contains-origin}. 
		
		By Proposition \ref{prop:lower-bound}, for each $t > 0$ and~$z \in B_{1/2}$ there exists $\epsilon(t) > 0$ so that $B_{\epsilon(t)}(z) \subset {\cluster{z}{t}}$.
		By Proposition \ref{prop:upper-bound-polynomially-high-prob} and the Borel-Cantelli lemma, applied to a sequence of dyadic radii, $r_n := \{2^{-n}\}_{n \geq 2}$, 
		a.s.\ there exists a sequence of positive times $\{t_n\}_{n \geq 2}$ so that 
		\[
		\overline{{\cluster{z}{t_n}}} \Subset B_{r_n}(z), \quad \forall n \geq 2.
		\]
		The desired claim follows by combining the previous two sentences with monotonicity of the clusters in $t$, Lemma \ref{lemma:monotonicity}.
	\end{proof}

	\section{Boundary has measure zero} \label{sec:boundary-measure-zero}
	
	We show that the boundary of a cluster compactly embedded in the unit ball has LQG-measure zero. 
	\begin{prop} \label{prop:boundary-measure-zero}
		On an event of probability 1, 
		\[
		\mu_h(\partial {\cluster{z}{t}}) = 0 \qquad \mbox{for all $z \in B_{1}$ with ${\cluster{z}{t}} \Subset B_1$}.
		\]		
	\end{prop}
	Our strategy for doing so is to use the Lebesgue density theorem together with some of the intermediate results from Section \ref{sec:harnack-type-estimate}.

	The Lebesgue density theorem for general Radon measures on $\C$~\cite[Corollary 2.14]{mattila1999geometry} shows that a.s.\ for every Borel set $X$, the set of $\mu_h$-density points of $X$ has full $\mu_h$-mass, i.e., 
	\begin{equation} \label{eq:density-0}
	\lim_{r \to 0} \frac{\mu_h(B_r(z) \cap X)}{\mu_h(B_r(z))} = 1,\quad \text{for $\mu_h$-a.e.\ $z\in X$}. 
	\end{equation}

	We want to deduce Proposition~\ref{prop:boundary-measure-zero} from the density theorem for $\mu_h$ in Euclidean balls~\eqref{eq:density-0}
	together with the Harnack-type estimate from Section~\ref{sec:harnack-type-estimate}. 
	However, the results of Section~\ref{sec:harnack-type-estimate} are in terms of the LQG mass of the intersection of a cluster with an annulus. In order to compare the $\mu_h$-masses of balls and annuli, we require a doubling property for the $\mu_h$-masses of Euclidean balls, \ie, an up-to-constants comparison of the $\mu_h$-masses of $B_{2^{-n}}(z)$ and $B_{2^{-n-1}}(z)$ with the constant independent of $n$ and $z$. That is, we require an event of the form,
	\begin{equation} \label{eq:annulus-ball-comparison}
	G_{n}(z) = G_n(z; m) := \left\{\mu_{h}(\A_{2^{-n-1}, 2^{-n}}(z)) \geq m \mu_{h}(B_{2^{-n}}(z)) \right\}
	\end{equation}
	for $z \in \C$ and $m > 0$. 
	Due to the randomness of $\mu_h$, such an event does not hold uniformly over all choices of $n$ and $z$. Instead, we will show that for $\mu_h$-a.e.\ $z \in \C$, one has this estimate 
	{\it and} a Harnack-type property for `most' large values of $n$. For convenience, write
	\begin{equation} \label{eq:very-good-annuli-event}
	\overline E_{n}(z) := \overline E_{2^{-n}}(z)
	\end{equation}
	for $z \in \C$ and $n \in \N$ where $\overline E_{\rho}(z)$ is the very good event from Section \ref{subsec:good-and-very-good-annuli}.
	
	\begin{lemma} \label{lemma:mu_h-doubling-and-harnack}
		Let $\zeta > 0$ and $b \in (0,1)$. There exists $m > 0$ and parameters corresponding to the event $\overline E_{\rho}(z) = \overline E_{\rho}(z; N_0, a, b, C_1^{\pm}, C_2^{\pm}, C_3)$ (from Section \ref{subsec:good-and-very-good-annuli}) such that a.s.\ for $\mu_h$-a.e.\ $z\in  B_1$, it holds for each large enough $N\in \N$ (depending on $z$) that
		\begin{equation} \label{eq:mu_h-doubling}
		\#\left\{ n \in [N+1,2N] \cap \Z : \mbox{$\overline E_{n}(z)$ and $G_{n}(z)$ occur} \right\}  \geq (1-\zeta) N.
		\end{equation}  
	\end{lemma}
	
	Recall that $h^{\C} = h + \boldsymbol{\alpha}_0 \log |\cdot|$ is the whole plane GFF as defined in \eqref{eq:cov-kernel}. It is a standard fact from LQG theory that if $U\subset  \C$ is open and $Z$ is sampled from $\mu_{h^\C}|_U$, normalized to be a probability measure, then near $Z$ the field $h$ locally looks like $\tilde h - \gamma \log |\cdot-Z|$, where $\tilde h$ is a GFF sampled independently from $Z$ (see, e.g.,~\cite[Section 3.3]{duplantier2011liouville}). Hence, Lemma~\ref{lemma:mu_h-doubling-and-harnack} will turn out to be a consequence  of the following statement for a GFF with a logarithmic singularity at 0. 
	
	\begin{lemma} \label{lemma:doubling-alpha}
		Let $\zeta > 0$ and $b \in (0,1)$. There exists $m > 0$ and parameters corresponding to the event $\overline E_{\rho}(0) = \overline E_{\rho}(0; N_0, a, b, C_1^{\pm}, C_2^{\pm}, C_3)$ (from Section \ref{subsec:good-and-very-good-annuli}) such that a.s.\ for each large enough $N \in \N$, the condition~\eqref{eq:mu_h-doubling} holds.
	\end{lemma}
	\begin{proof}
		By the scale invariance of the law of $h$, viewed modulo additive constant~\eqref{eq:h-coordinate-change}, and the LQG coordinate change formula for $\mu_{h}$ (Fact \ref{fact:lqg-measure}), the law of  
		\[
		\mu_{h}(\A_{2^{-n-1}, 2^{-n}}(0))/\mu_{h}(B_{2^{-n}}(0))
		\]
		does not depend on $n$. 
		Furthermore, this random variable is a.s.\ finite and strictly larger than 0. Hence, 
		$   \P[G_n(0; m)]$ does not depend on $n$ and we can find $m = m(\alpha,\zeta,\gamma) > 0$ such that
		\begin{equation} \label{eq:doubling-prob}
		\P[G_n  ] \geq 1 - \zeta /8 , \quad\forall n\in\N ,\quad \text{where $G_n := G_n(0; m)$} .
		\end{equation} 
		
		By  Lemma \ref{lemma:very-good-annuli-hit} there is a choice of parameters so that 
		\begin{equation} \label{eq:harnack-prob}
		\P[\overline E_n] \geq 1 - \zeta /8 , \quad \forall n \geq 0 ,\quad \text{where $\overline E_n := \overline E_n(0)$} .
		\end{equation}
		Hence, by a union bound
		\begin{equation} \label{eq:doubling-and-harnack-prob}
		q := \P[\overline E_n \cap G_n] \geq 1- \zeta/4.
		\end{equation}
		By the scale invariance of the law of $h$ modulo additive constant and the fact that the occurrence of the event $G_n \cap \overline E_n$ does not depend on the choice of additive constant for $h$, the sequence of random variables
		\[
		\{1_{G_n \cap \overline E_n}\}_{n \geq 0}
		\]
		is stationary. 
		Hence, by the Birkhoff ergodic theorem, 
		\[
		\frac{1}{N} \sum_{n=1}^N  1_{G_n \cap \overline E_n}
		\]
		converges a.s.\ and in $L^1$ to a (possibly random) limit. The limiting random variable is measurable with respect to the $\sigma$-algebra $\bigcap_{\epsilon > 0} \sigma(h|_{B_\epsilon(0)})$, which is trivial (see~\cite[Lemma 7.2]{duplantier2014liouville} for a proof of the analogous tail triviality statement for a free-boundary GFF; the proof for a whole-plane GFF is similar). 
		Therefore, the limiting random variable is a.s.\ constant, and hence is a.s.\ equal to the number $q$ from~\eqref{eq:doubling-and-harnack-prob}. Consequently, a.s.\ 
		\begin{equation}
		\lim_{N \to \infty} \frac{1}{N} \sum_{n=1}^N  1_{G_n \cap \overline E_n} = q. 
		\end{equation}
		Hence
		\begin{equation}
		\lim_{N \to \infty} \frac{1}{N} \sum_{n=N+1}^{2N}  1_{G_n \cap \overline E_n}
		= \lim_{N \to\infty} \frac{1}{N} \sum_{n=1}^{2N}  1_{G_n \cap \overline E_n} - \lim_{N \to\infty} \frac{1}{N} \sum_{n=1}^{N}  1_{G_n \cap \overline E_n} 
		= 2q - q
		\geq 1-\zeta/2 .
		\end{equation}
		By the definition of $G_n$ and $\overline E_n$, this implies the lemma statement. 
	\end{proof}

	\begin{proof}[Proof of Lemma~\ref{lemma:mu_h-doubling-and-harnack}]
		Recall that $h = h^{\C} - \boldsymbol{\alpha}_0 \log|\cdot|$, where $h^{\C}$ is a whole-plane GFF normalized so that $h^{\C}_1(0) = 0$. 
		By Weyl scaling, we have $\mu_{h^{\C}} = |\cdot|^{\boldsymbol{\alpha}_0 \gamma} \mu_h$. 
		Conditional on $h$, let $Z$ be sampled from $|\cdot|^{\boldsymbol{\alpha}_0 \gamma} \mu_{h}|_{B_1}$, normalized to be a probability measure. 
		By~\cite[Lemma A.10]{duplantier2014liouville}, the law of the pair $(h, Z)$ is mutually absolutely continuous with respect to the law of the pair $(\tilde h , \tilde Z)$, where $\tilde Z$ is sampled from Lebesgue measure in $B_1$ independently from $h$ and $\tilde h = h  - \gamma\log|\cdot-\tilde Z|  + \gamma \log \max\{|\cdot|,1\}$, with $h$ and $\tilde h$ viewed as distributions modulo additive constant. 
		
		From the definitions of $G_n(\tilde Z)$ and $\overline E_n(\tilde Z)$ and the locality property of $\mu_{\tilde h}$ (Fact~\ref{fact:lqg-measure}), we have 
		\begin{equation} \label{eq:doubling-locality}
		G_n(\tilde Z) \cap \overline E_n(\tilde Z) \in \sigma\left( \tilde Z  , h |_{B_{2^{-n+1}}(\tilde Z)} \right).
		\end{equation}
		Almost surely, $\tilde Z \not= 0$. If $r < |\tilde Z|$, then the restriction of $\tilde h$ to $B_r(\tilde Z)$ is equal to the restriction of a whole-plane GFF to $B_r(\tilde Z)$ plus $- \gamma\log|\cdot-\tilde Z|$ plus the function $-\boldsymbol{\alpha}_0 \log|\cdot|  + \gamma \log\max\{|\cdot|,1\}$, which is smooth on $B_r(\tilde Z)$. By standard absolute continuity results for the GFF (see, e.g.,~\cite[Proposition 2.9]{miller2017imaginary}), the conditional law of $\tilde h|_{B_r(\tilde Z)}$ given $\tilde Z$ is absolutely continuous with respect  to the law of the corresponding restriction of a whole-plane GFF plus $-\gamma\log|\cdot-\tilde Z|$.  
		From this, \eqref{eq:doubling-locality}, Lemma~\ref{lemma:doubling-alpha} (with $\boldsymbol{\alpha}_0 =\gamma$), and the translation invariance of the law of the whole-plane GFF, viewed modulo additive constant, we get that if the parameters for $\overline E_n(\cdot)$ and  $G_{n}(\cdot)$ are chosen as in Lemma~\ref{lemma:doubling-alpha}, then a.s.\ for each large enough $N\in\N$, 
		\begin{equation} \label{eq:mu_h-doubling-reweight}
		\#\left\{ n \in [N+1,2N] \cap \Z  : \mbox{$\overline E_{n}(\tilde Z)$ and $G_{n}(\tilde Z)$ occur with $\tilde h$ in place of $h$} \right\}  \geq (1-\zeta) N.
		\end{equation}  
		By absolute continuity, the same is also true with $(h,Z)$ in place of $(\tilde h , \tilde Z)$. Since $Z$ is sampled from $|\cdot|^{\boldsymbol{\alpha}_0 \gamma} \mu_{h}|_{B_1}$, we get that a.s.\ the lemma statement holds for $\mu_h$-a.e.\ $z\in B_1$. 
	\end{proof}

	We conclude with a proof of the desired claim. 
	
	\begin{proof}[Proof of Proposition \ref{prop:boundary-measure-zero}]
		By Lemma \ref{lemma:very-good-annuli-satisfy-harnack-type-estimate} 
		and Lemma \ref{lemma:mu_h-doubling-and-harnack}
		we can choose parameters $b$, $N_0$, $a,C^{\pm}_1, C^{\pm}_2, C_3, c, m$
		so that for some fixed $\alpha \in (0,1)$, it holds for all~$z \in B_{1}$, $x_0 \in B_1 \setminus \{z\}$ with~$\overline{B}_{\rho}(x_0) \subset B_1 \setminus \{z\}$ that
		\begin{equation} \label{eq:good-annulus-implies-harnack}
		\mbox{$\overline E_{\rho}(x_0)$ occurs} \implies  \{ 		\mu_h({\cluster{z}{t}} \cap \A_{\rho/2, \rho}(x_0)) \leq \alpha \mu_h(\A_{\rho/2, \rho}(x_0)) \implies {\cluster{z}{t}} \cap B_{\rho/2}(x_0) = \emptyset \}
		\end{equation}
		and the implication of Lemma \ref{lemma:mu_h-doubling-and-harnack}
		holds with $\zeta = 1/2$. 
		
		We will now show that $\mu_h(\partial{\cluster{z}{t}}) = 0$ for every $t>0$ and~$z \in B_{1}$ such that ${\cluster{z}{t}} \Subset B_1$. By \eqref{eq:density-0}, a.s.\ for every $t>0$ we have
		\begin{equation} \label{eq:density-0-apply0}
		\lim_{\rho \to 0} \frac{\mu_h(B_{\rho}(w) \cap \partial {\cluster{z}{t}})}{\mu_h(B_{\rho}(w))} = 1, \quad \text{for $\mu_h$-a.e.\ $w\in \partial {\cluster{z}{t}}$} .
		\end{equation}		
		Since ${\cluster{z}{t}}$ is open, we have ${\cluster{z}{t}} \cap \partial {\cluster{z}{t}} = \emptyset$, so \eqref{eq:density-0-apply0} implies that
		\begin{equation} \label{eq:density-0-apply}
		\lim_{\rho \to 0} \frac{\mu_h(B_{\rho}(w) \cap {\cluster{z}{t}})}{\mu_h(B_{\rho}(w))} = 0 , \quad \text{for $\mu_h$-a.e.\ $w\in \partial {\cluster{z}{t}}$} .
		\end{equation}
		In particular, for $\mu_h$-a.e.\ $w\in \partial {\cluster{z}{t}}$ it holds for each large enough $n\in\N$ that
		\begin{equation} \label{eq:density-theorem-apply}
		\mu_h(B_{2^{-n}}(w) \cap {\cluster{z}{t}}) \leq  \alpha \times m \times \mu_h(B_{2^{-n}}(w))  
		\end{equation} 
		where $\alpha$ is as in Lemma \ref{lemma:very-good-annuli-satisfy-harnack-type-estimate} and $m$ is as in the definition~\eqref{eq:annulus-ball-comparison} of $G_n(w; m)$.
		
		By Lemma \ref{lemma:mu_h-doubling-and-harnack}, it is a.s.\ the case that for $\mu_h$-a.e\ $w\in B_1$, there are arbitrarily large values of $n$ such that $\overline E_{n}(w)$ and $G_{n}(w)$ occur. 
		Hence, a.s.\ for each $t>0$, it holds for $\mu_h$-a.e.\ $w\in\partial {\cluster{z}{t}}$ that there are arbitrarily large values of $n$ such that $\overline E_n(w) \cap G_n(w)$ occurs and
		\begin{align*}
		\mu_h(\A_{2^{-n-1}, 2^{-n}}(w) \cap {\cluster{z}{t}}) &\leq \mu_h(B_{2^{-n}}(w) \cap {\cluster{z}{t}})   \\
		&\leq \alpha \times m \times \mu_h(B_{2^{-n}}(w)) \qquad \mbox{(by \eqref{eq:density-theorem-apply})} \\
		&\leq \alpha \mu_h(\A_{2^{-n-1}, 2^{-n}}(w) \cap {\cluster{z}{t}}) \qquad \mbox{(since $G_n(w)$ holds)}. 
		\end{align*}
		Since $\overline E_{n}(w)$ holds, this implies by \eqref{eq:good-annulus-implies-harnack} that
		\[
		{\cluster{z}{t}} \cap B_{2^{-n-1}}(w) = \emptyset, 
		\]
		which shows that $w\notin \partial {\cluster{z}{t}}$. Hence, we have shown that a.s., it holds for each $t>0$ that $\mu_h$-a.e.\ $w\in \partial {\cluster{z}{t}}$ does not belong to $\partial{\cluster{z}{t}}$, which means that $\mu_h(\partial{\cluster{z}{t}}) = 0$.  
	\end{proof}

	\section{Uniqueness of harmonic balls} \label{sec:uniqueness}
	
	In this section we show that there is only one family of harmonic balls satisfying the conditions given by Theorem \ref{theorem:harmonic-balls}. 
	
	\subsection{Uniqueness of subharmonic balls}
	
	The uniqueness of subharmonic balls (defined in \eqref{eq:subharmonic-ball}) for the Lebesgue measure is well known, see, \eg, \cite{sakai1984solutions},
	and the proof extends verbatim to $\gamma$-LQG subharmonic balls.
	The idea is that every subharmonic ball generates a supersolution to the obstacle problem.
	\begin{lemma}[Theorem 2.1 in \cite{shahgholian2013harmonic}, \cite{sakai1984solutions}] \label{lemma:obstacle-comp}
		Let~$R > 0$,~$z \in B_R$ and let $A$ be a domain strictly contained which contains~$z$ and let
		\begin{equation} \label{eq:obstacle-comp}
		f(x) = \mu_h(A) G_{B_R}(x,z) - \int_{A} G_{B_R}(x,y)  \mu_h(d y).
		\end{equation}
		\begin{enumerate}
			\item If $A$ is a subharmonic ball centered at~$z$, then $f \geq 0$ on $B_R$. 
			\item If $A$ is a harmonic ball centered at~$z$, then $f = 0$ on $B_R \backslash A$. 
		\end{enumerate}
	\end{lemma} 
	\begin{proof}
		{\it (1)} \\
		For all $x \in B_R$, the function $g(w) := G_{B_R}(w, x)$
		is superharmonic in $A$, so $-g$ is subharmonic in $A$. Then, since $A$ is a subharmonic ball centered at~$z$, 
		\begin{equation}
		\int_{A} g(w) \mu_h(dw) \leq \mu_h(A) g(z),
		\end{equation}
		and so 
		\begin{equation}
		\int_{A} G_{B_R}(w, x) \mu_h(dw) \leq \mu_h(A) G_{B_R}(x,z) ,\quad \forall x \in B_R .
		\end{equation} 
		
		{\it (2)} \\ 
		In this case, the function $g(w) = G_{B_R}(w, x)$ is harmonic in $A$ for all $x \in B_R \backslash A$. 
		Hence, since $A$ is a harmonic ball, 
		\begin{equation}
		\int_{A} G_{B_R}(w, x) \mu_h(dw) = \mu_h(A) G_{B_R}(x,z) ,\quad \forall x \in B_R \backslash A .
		\end{equation}
	\end{proof}
	
	The prior lemma implies uniqueness (up to sets of $\mu_h$-measure zero) of subharmonic balls. 
	\begin{prop} \label{prop:sub-harmonic-ball-uniqueness} 
		The following holds on an event of probability 1. Let $A \Subset B_R$ be a subharmonic ball centered at~$z \in B_R$ with $\mu_h(A)=t$ for some $t, R > 0$. Then, ${\regcluster{B_R;z}{t}} \subset A$ and $\mu_h(A \backslash {\regcluster{B_R;z}{t}}) = 0$.  
	\end{prop}
	
	\begin{proof}
		Let $f$ be given by \eqref{eq:obstacle-comp} and note that by Lemma \ref{lemma:obstacle-comp} (and since every subharmonic ball is a harmonic ball), 
		$f  \geq 0$ in $B_R$ and $f = 0$ on $B_R \backslash A$.  Since $f \geq 0$ and $\Delta f \leq -\mu_h(A) \delta_z + \mu_h$ on $B_R$, we have that $(f(\cdot)- t G_{B_R}(z,\cdot)) \in \mclS{B_R; z}{t}$ and hence $f \geq \odometer{B_R; z}{t}$. This implies, together with $f = 0$ on $B_R \backslash A$, that  $\cluster{B_R; z}{t} \subset A$. 
		Since $t = \mu_h(A) = \mu_h({\regcluster{B_R; z}{t}})$, this completes the proof. 
	\end{proof}

	We do not use Proposition \ref{prop:sub-harmonic-ball-uniqueness} in this paper but decided to include it explicitly as it may be useful for future work.

	\subsection{Comparing harmonic and subharmonic balls}
	
	We show that any harmonic ball for $\mu_h$ (in the sense of~\eqref{eq:harmonic-ball}) centered at~$z$ must have a boundary contained in the closure of 
	some ${\regcluster{z}{t}}$. Similar arguments have appeared in \cite[Proposition 3.2]{shahgholian2013harmonic},
	\cite[Proposition 2.5]{hedenmalm2002hele}, and \cite[Section 3]{gustafsson1990onquadraturedomains}. 
	\begin{lemma}\label{lemma:contains-boundary}
		Almost surely, every harmonic ball  $A \Subset B_R$ centered at~$z \in B_R$ with $t = \mu_h(A) > 0$  satisfies $\partial A \subset \overline{{\regcluster{}{t}}(z)}$.
	\end{lemma}
	
	\begin{proof}
		
		Let $A$ be a harmonic ball centered at~$z \in B_R$ with $t = \mu_h(A) > 0$ and $A \Subset B_R$. Take $R$ possibly larger so that
		${\regcluster{B_R; z}{t}} \Subset B_R$ and hence by Theorem \ref{theorem:full-theorem-minus-uniqueness},  ${\regcluster{B_R;z}{t}} = {\regcluster{}{t}}(z)$ is a subharmonic ball in $B_R$ centered at~$z$.
		We show that $\partial A \subset \overline{{\regcluster{}{t}}}(z)$ by considering the auxiliary function
		\begin{equation} \label{eq:auxiallary-function}
		u(x) := \int_{A} G_{B_R}(x,y)  \mu_h(d y) - \int_{{\regcluster{}{t}(z)}} G_{B_R}(x,y)  \mu_h(d y)
		\end{equation}
		which satisfies $\Delta u = (1_{{\regcluster{}{t}(z)}} - 1_{A}) \mu_h$ on $B_R$. 
		
		{\it Step 1: $u \geq 0$ in $B_R$.}  \\
		By Lemma \ref{lemma:obstacle-comp}, 	as $A$ is a harmonic ball centered at~$z$
		\begin{equation} \label{eq:harmonic-ball-equality}
		t G_{B_R}(x,z) = \int_{A} G_{B_R}(x,y)  \mu_h(d y) \quad \mbox{on $B_R \backslash A$}.
		\end{equation}
		and as ${\regcluster{}{t}(z)}$ is a subharmonic ball centered at~$z$
		\begin{equation} \label{subharmonic-ball-inequality}
		t G_{B_R}(x,z) \geq \int_{{\regcluster{}{t}(z)}} G_{B_R}(x,y)  \mu_h(d y) \quad \mbox{on $B_R$}.
		\end{equation}
		Combining \eqref{eq:harmonic-ball-equality} and \eqref{subharmonic-ball-inequality} shows 
		\begin{equation} \label{eq:intermediate-harmonic-ball-inequality}
		u \geq 0 \quad \mbox{on $B_R \backslash A$}.
		\end{equation}
		As $u$ is superharmonic on $A$, \eqref{eq:intermediate-harmonic-ball-inequality} together with the minimum 
		principle shows $u \geq 0$ on $B_R$.
		
		{\it Step 2: $u = 0$ on $B_R \cap {{\regcluster{}{t}(z)}}^c \cap A^c$.} \\
		As ${\regcluster{}{t}}(z)$ is also a harmonic ball, by Lemma \ref{lemma:obstacle-comp}, 
		\begin{equation} \label{eq:subharmonic-ball-equality}
		t G_{B_R}(x,z) = \int_{{\regcluster{}{t}(z)}} G_{B_R}(x,y)  \mu_h(d y) \quad \mbox{on $B_R \backslash {\regcluster{}{t}(z)}$}.
		\end{equation}
		Step 2 follows by combining \eqref{eq:harmonic-ball-equality} and \eqref{eq:subharmonic-ball-equality}
		with the definition of $u$.

		{\it Step 3: Conclude.} \\
		We use  the fact $\Delta u = (1_{{\regcluster{}{t}(z)}} - 1_{A}) \mu_h$ on $B_R$. 
		Suppose for sake of contradiction there is $x_0 \in \partial A \backslash \overline{{\regcluster{}{t}}(z)}$.
		Then, $u$ is superharmonic in a neighborhood of $x_0$ as $x_0 \in (\overline{{\regcluster{}{t}(z)}})^c$.
		By Step 2, $u(x_0) = 0$, which, together with Step 1 and the strong minimum principle, shows that $u$ is identically 0 in a neighborhood of $x_0$. 
		This in turn implies that $u$ is harmonic in a neighborhood of $x_0$. 
		
		However, as $x_0 \in \overline{{\regcluster{}{t}}(z)}^c$, we have that $\Delta u = -1_{A} \mu_h$ in a neighborhood of $x_0$. As $\mu_h$ assigns positive mass to every open set, $A$ is open, and $x_0 \in \partial A$, $\Delta u$ is strictly negative on an open subset of every neighborhood of $x_0$, which supplies the desired contradiction.
	\end{proof}

	\subsection{Strong uniqueness} 
	We first show that two regular open sets
	which coincide $\mu_h$-a.e. are in fact equal.
	This fact is the reason for our assumption that $\mathrm{int}(\overline{\Lambda}_t) = \Lambda_t$ in Theorem \ref{theorem:harmonic-balls}.
	
	\begin{lemma} \label{lemma:regular-open-ae}
		Almost surely, the following is true. Let $X$ and $Y$ be two open subsets of $\C$ such that  $\mathrm{int}(\overline{X}) = X$ and $\mathrm{int}(\overline{Y}) = Y$.
		If $\mu_h(X \backslash Y) = \mu_h(Y \backslash X) = 0$, then $X = Y$. 
	\end{lemma}
	
	\begin{proof}
		
		We use the fact that a.s. $\mu_h$ assigns positive mass to every open set to show that $X \backslash Y = \emptyset$. A symmetric argument shows $Y \backslash X = \emptyset$. 
		
		Suppose for the sake of contradiction there is $x_0 \in X \cap Y^c$. As $X$ is open, $B_r(x_0) \subset X$ for all $r < r_0$, where $r_0$ is some small radius. 
		As $\mu_h(X \cap Y^c) = 0$, $B_r(x_0) \not \subset Y^c$ for all $r < r_0$. Thus, 
		$x_0$ is a limit point of $Y$ and so $x_0 \in \partial Y$ by definition.  
		Since $\mathrm{int}(\overline{Y}) = Y$, $x_0$ cannot be in the interior of $\overline Y$ so $x_0$ must be an accumulation point of $\overline Y^c$.
		Since $\overline Y^c$ is open, this implies that every neighborhood of $x_0$ contains an open subset of $\overline Y^c$. Hence $X\cap \overline Y^c \subset X\cap Y^c$ contains a non-empty open set.
		This implies $B_r(x_0) \cap Y^c$ contains a non-empty open set, and hence that $\mu_h(X \cap Y^c) > 0$, a contradiction. 
	\end{proof}
	We now use Lemma~\ref{lemma:regular-open-ae} to show uniqueness of the family of harmonic balls satisfying the conditions of Theorem \ref{theorem:harmonic-balls}.
	Our proof is inspired by the proof of Theorem 10.13 in \cite{sakai2006quadrature}.

	\begin{proof}[Proof of Proposition \ref{prop:uniqueness}]
		Let $t_0 > 0$ and~$z \in \C$ be given and suppose $A_{t_0} \Subset B_R$ for some $R > 0$. 
		By Lemma \ref{lemma:regular-open-ae} it suffices to show that $\mu_h(A_{t_0} \backslash \regcluster{}{t_0})  = \mu_h( \regcluster{}{t_0}(z) \setminus A_{t_0}) = 0$. As we have assumed $\mu_h(A_{t_0}) = t_0$, we only need to show that $\mu_h(A_{t_0} \backslash \regcluster{}{t_0}(z)) = 0$. 
		
		Suppose for the sake of contradiction that $\mu_h(A_{t_0} \backslash \regcluster{}{t_0}(z)) > 0$. 
		Since $\mu_h(\partial \regcluster{}{t_0}(z)) = 0$, this implies $\mu_h(A_{t_0} \backslash \overline{\regcluster{}{t_0}(z)})  >0$.
		In particular, as $A_{t_0} \backslash \overline{\regcluster{}{t_0}(z)}$ is open, there exists a non-empty ball $B \subset A_{t_0} \backslash \overline{\regcluster{}{t_0}(z)}$
		which lies at positive distance from the origin. We will show that $B$ cannot exist by monotonicity 
		\begin{equation} \label{eq:cluster-monotonicity}
		A_a \subseteq  A_b  \quad \text{and} \quad \regcluster{}{a} \subseteq \regcluster{}{b} , \quad \forall a \leq b 
		\end{equation}
		and Lemma \ref{lemma:contains-boundary},
		\begin{equation} \label{eq:boundary-contained}
		\partial A_{t} \subset \overline{{\regcluster{}{t}(z)}}, \quad \forall t > 0.
		\end{equation}

		As the family $\{A_t\}_{t > 0}$ continuously decreases to~$\{z\}$ as $t \to 0$, 
		there exists some $0 < s_0 < t_0$ for which 
		\begin{equation} \label{eq:ball-contained-in-initial-complement}
		B \subset A_{s_0}^c.
		\end{equation}
		Now let, $s_0 < s_1 < \cdots < s_m = t_0$ be a sequence of points satisfying, 
		\begin{equation}
		s_{n} < s_{n+1} \leq s_n + \mu_h(B)/2  \quad \forall n \in \{0,\dots,m-1\}.
		\end{equation}
		We show by induction on $n$ that for each $0 \leq n \leq m$, 
		\begin{equation} \label{eq:b-contained-in-next-cluster}
		B \subset A_{s_n}^c,
		\end{equation}
		which contradicts $B \subset A_{t_0}$. The base case $n = 0$ is established by \eqref{eq:ball-contained-in-initial-complement}. 
		
		Now assume \eqref{eq:b-contained-in-next-cluster} holds for $n \in \{0,\dots,m-1\}$. 
		We show that for 
		\begin{equation} \label{eq:time-interval}
		s_{n+1} \in \left[ s_n ,  \min(s_n + \mu_h(B)/2, t_0) \right]
		\end{equation}
		we must have 
		\begin{equation} \label{eq:b-contained-in-next-cluster-induct}
		B \subset A_{s_{n+1}}^c.
		\end{equation}

		First note that the interval \eqref{eq:time-interval} is non-empty as $B \subset A_{t_0}$, $B \subset A_{s_n}^c$, 
		and $\mu_h(A_t) = t$ for all $t > 0$. 
		Moreover, as $B \subset  (\overline{\regcluster{}{t_0}(z)})^c$ and $t_0 \geq s_{n+1}$, by monotonicity~\eqref{eq:cluster-monotonicity}, $B \subset  (\overline{\regcluster{}{s_{n+1}(z)}})^c$.
		This together with \eqref{eq:boundary-contained} implies  $B \cap \partial A_{s_{n+1}} = \emptyset$; equivalently 
		\begin{equation} \label{eq:two-cases}
		B \subset A_{s_{n+1}} \quad \mbox{or} \quad B \subset \overline{A_{s_{n+1}}}^c. 
		\end{equation}
		The former case in \eqref{eq:two-cases} is impossible as 
		\begin{align*}
		\mu_h(A_{s_{n+1}} \cap B) &\leq \mu_h(A_{s_{n} + \mu_h(B)/2} \cap B)  \quad \mbox{($s_{n+1} \leq s_{n} + \mu_h(B)/2$)} \\
		&= \mu_h(A_{s_{n}} \cap B) + \mu_h((A_{s_{n} + \mu_h(B)/2} \backslash A_{s_{n}} ) \cap B) \quad \mbox{(additivity of measure)} \\
		&\leq \mu_h(A_{s_{n}} \cap B) + \mu_h(B)/2  \quad \mbox{(monotonicity and $\mu_h(A_t) = t, \forall t$)} \\
		&= \mu_h(B)/2 \quad \mbox{(by \eqref{eq:b-contained-in-next-cluster})}
		\end{align*}
		which shows \eqref{eq:b-contained-in-next-cluster-induct}.
	\end{proof}

	\section{Boundary curves of Harmonic balls} \label{sec:boundary-curves}
	
	In this section we show that harmonic balls have boundaries which are simple loops. 
	We first consider clusters ${\cluster{}{t}} \Subset B_1$ and then rescale to achieve the result
	for all $\{{\regcluster{}{t}}\}_{t > 0}$.  
	
	Our main result is essentially a consequence of the following lemma, which limits 
	how many times clusters cross annuli --- see Figure \ref{fig:cannot-cross}.

	\begin{figure}
		\begin{center}
			\includegraphics[width=0.25\textwidth]{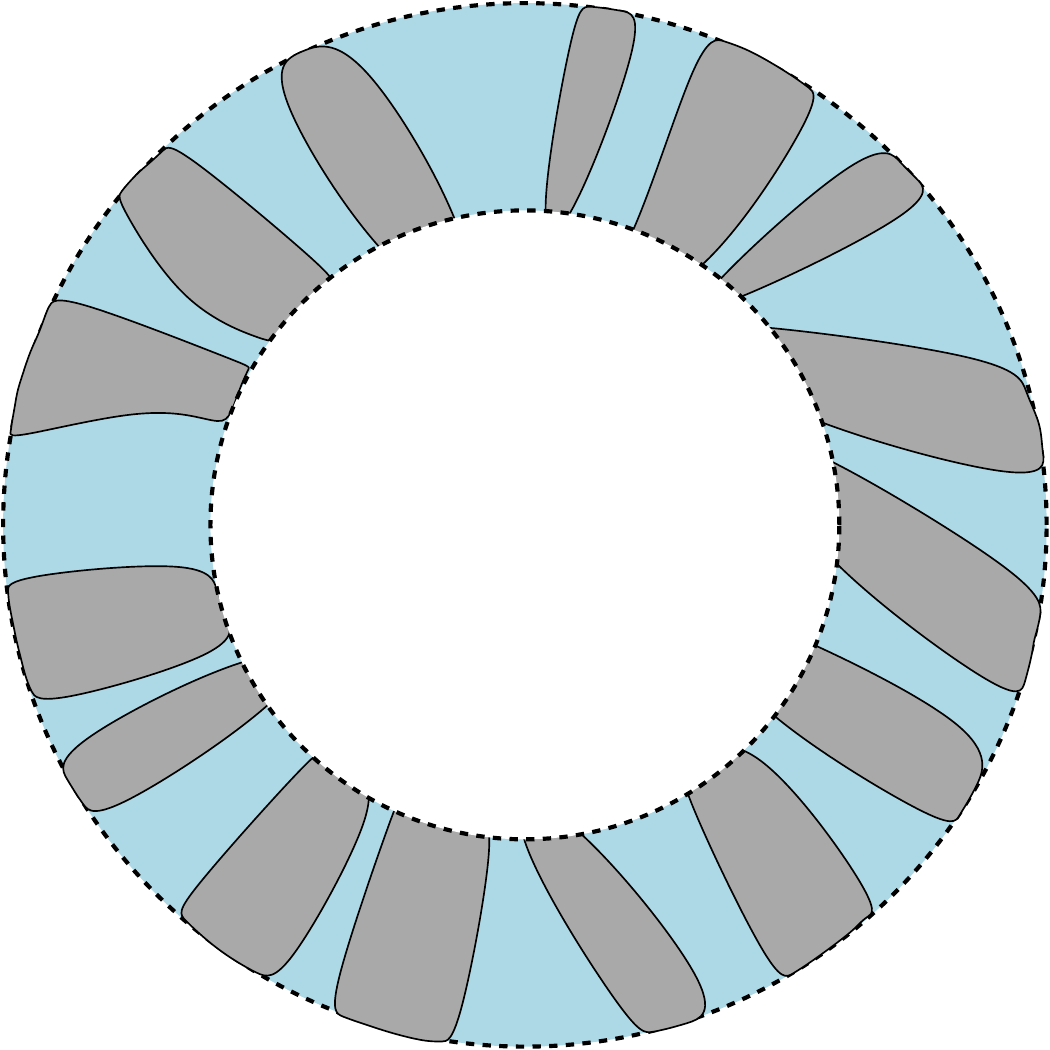}
			\caption{A situation ruled out by Lemma \ref{lemma:cannot-cross}. The annulus $\A_{\rho/3, 2 \rho}$ is in light blue and the connected components of ${\cluster{}{t}} \cap \A_{\rho/3, 2 \rho}$
				which cross the annulus are in gray.}
			\label{fig:cannot-cross}
		\end{center}
		
	\end{figure}
	
	\begin{lemma}\ \label{lemma:cannot-cross} Almost surely, for each small enough $\epsilon > 0$
		it holds for each $w \in B_1 \backslash B_{20 \sqrt{\epsilon}}$ that there is a $\rho \in [\epsilon, \epsilon^{1/2}]$
		(depending on $w$) such that for every~$t > 0$ and~$z \in \C$, we have that ${\cluster{}{t}(z)}$ does not cross $\A_{\rho/3, 2\rho}(w)$ more than  $K := \lceil 1/\alpha \rceil$ times for all ${\cluster{}{t}(z)} \Subset B_1$
		where $\alpha = \alpha(\gamma) \in (0,1)$ is as in Proposition \ref{prop:harnack-type-stronger-property}. 
		That is, there are at most $K$
		connected components of ${\cluster{}{t}(z)} \cap B_{2 \rho}(w)$
		whose closures intersect both $\partial B_{\rho/3}(w)$ and $\partial B_{2\rho}(w)$.
	\end{lemma}
	\begin{proof}
		Let $t>0$ such that ${\cluster{}{t}(z)} \Subset B_1$. By Proposition \ref{prop:harnack-type-property} and the Borel-Cantelli lemma, a.s.\ there exists a random $M_0 \in \N$ and a deterministic $\alpha \in (0,1)$ such that for all $\epsilon \in \{2^{-n}\}_{n \geq M_0}$ we have that for each $x_0 \in (B_{1 +\epsilon} \backslash B_{10 \sqrt{\epsilon}}) \cap \frac{\epsilon}{100} \Z^2$
		there is a $\rho \in [\epsilon, \epsilon^{1/2}]$ for which $\overline E_\rho(x_0)$ occurs.

		Fix $\epsilon \in \{2^{-n}\}_{n \geq M_0}$ and let $w \in B_1 \backslash B_{20 \sqrt{\epsilon}}$. Let $x_0$ be a point of $(B_{1 + \epsilon} \backslash B_{10 \sqrt{\epsilon}}) \cap \frac{\epsilon}{100} \Z^2$
		with $|x_0 - w| < \epsilon/50$ and let $\rho$ be a radius in $[\epsilon, \epsilon^{1/2}]$ such that $\overline E_\rho(x_0)$ occurs, as in the statement of Proposition \ref{prop:harnack-type-property}. 
		Note that since $\rho > \epsilon$ and $|x_0-w| < \epsilon/50$, 
		\begin{equation} \label{eq:annulus-grid-contains}
		B_{\rho}(x_0) \subset B_{2 \rho}(w).
		\end{equation}
		As ${\cluster{}{t}(z)}$ is open and connected (Lemma \ref{lemma:non-coincidence-open-harmonic}), and contains~$\{z\}$ (Proposition \ref{prop:lower-bound}), \eqref{eq:annulus-grid-contains} shows that any non-empty connected component of ${\cluster{}{t}(z)} \cap B_{2 \rho}(w)$
		with closure intersecting $\partial B_{\rho/3}(w)$ and $\partial B_{2 \rho}(w)$
		must intersect $B_{\rho}(x_0)$ and have closure intersecting $\partial B_{\rho}(x_0)$ and $\partial B_{\rho/2}(x_0)$. 
		
		Also note that as $\overline E_\rho(x_0)$ occurs, Proposition \ref{prop:harnack-type-stronger-property} implies that
		\begin{equation} \label{eq:harnack-contrapositive}
		\begin{aligned}
		\overline{\tilde \Lambda_t} \cap \partial B_{\rho/2}(x_0) \neq \emptyset  &\implies \mu_h(\tilde \Lambda_t \cap \A_{\rho/2, \rho}(x_0)) \geq \alpha \mu_h(\A_{\rho/2, \rho}(x_0)) \\
		&\quad \forall  \tilde \Lambda_t, \mbox{ connected component of ${\cluster{}{t}(z)} \cap B_{\rho}(x_0)$}.
		\end{aligned}
		\end{equation}		
		
		By \eqref{eq:annulus-grid-contains}, every connected component of ${\cluster{}{t}(z)} \cap B_{2 \rho}(w)$
		with closure intersecting $\partial B_{\rho/3}(w)$ 
		intersected with $B_{\rho}(x_0)$ decomposes into a union of connected components of 
		${\cluster{}{t}(z)} \cap B_{\rho}(x_0)$ each of which has closure intersecting $\partial B_{\rho}(x_0)$ and $\partial B_{\rho/2}(x_0)$. 
		Thus, the relations \eqref{eq:annulus-grid-contains} and \eqref{eq:harnack-contrapositive} together imply that
		\begin{equation} \label{eq:harnack-type-contrapositive-consequence}
		\begin{aligned}
		\mu_h(\hat \Lambda_t \cap \A_{\rho/2, \rho}(x_0)) &\geq \alpha \mu_h(\A_{\rho/2, \rho}(x_0)) \\
		&\qquad \forall \hat \Lambda_t, \mbox{ connected component of ${\cluster{}{t}(z)} \cap B_{2 \rho}(w)$} \\
		&\quad \mbox{such that $\overline{\hat \Lambda_t} \cap \partial B_{2 \rho}(w) \neq \emptyset$ and $\overline{\hat \Lambda_t} \cap \partial B_{\rho/3}(w) \neq \emptyset$}.
		\end{aligned}
		\end{equation}
		
		By \eqref{eq:annulus-grid-contains}, this is a lower bound on the mass of each connected component of ${\cluster{}{t}(z)} \cap B_{2 \rho}(w)$. 
		In particular, \eqref{eq:harnack-type-contrapositive-consequence} implies that the number of such connected components with closures which intersect
		$\partial B_{\rho/3}(w)$ and $\partial B_{2 \rho}(w)$ is at most $K := \lceil 1/\alpha \rceil$. 
		This implies ${\cluster{}{t}(z)}$ can cross $\A_{\rho/3, 2 \rho}(w)$ (as in the statement of Lemma \ref{lemma:cannot-cross}) at most $K$ times. 
	\end{proof}

	We recall the definition of a loop.
	\begin{definition}
		A set $\Gamma \subset \C$ is a \emph{loop} if $\Gamma = \{\varphi(\zeta): \zeta \in \mathbb T\}$
		for a continuous function $\varphi$ from the unit circle $\mathbb T$ to $\C$. The set $\Gamma$ is an \emph{arc} if $\Gamma = \{\varphi(\zeta): \alpha \leq \zeta \leq \beta\}$.
		$\Gamma$ is a \emph{simple loop [arc]} if $\Gamma$ is a loop [arc] and $\varphi$ is also injective. 
		
	\end{definition}

	We recall the definition of a locally connected set.
	\begin{definition}
		A set $X \subset \C$ is \emph{locally connected} if every
		neighborhood of each $x \in X$ with respect to $X$ contains a connected neighborhood of $x$. 
	\end{definition}
	
	We also recall the definition of cut points. 
	\begin{definition}
		Let $A$ be a compact, connected, and locally connected set. A point $a \in A$ is a \emph{cut point} 
		if $A \backslash \{a\}$ is no longer connected.
	\end{definition}

	\begin{figure}
		\includegraphics[width=0.5\textwidth]{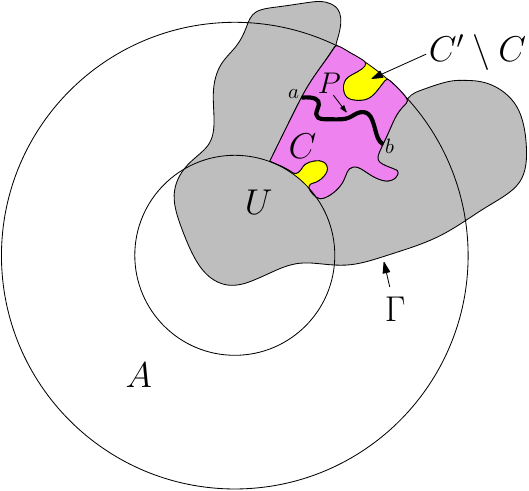}
		\caption{A visual aid to the proof of Lemma \ref{lemma:crossing-connected}. The annulus $A$ is transparent, $U$ is the gray simply connected domain, 
			$C$ is the connected violet domain, the components of $(C' \backslash C) \cap A$ are in yellow, and $P$ is the path with two boundary points $a$ and $b$; both $a$ and $b$ are in $\Gamma$, the boundary of $U$.}
		\label{fig:crossing-argument}
	\end{figure}

	We will use Lemma \ref{lemma:cannot-cross} together with some basic properties of the cluster to show
	that the boundaries of the complementary connected components of the cluster are simple loops. Before doing so, 
	we prove a topological lemma. Recall from the statement of Lemma \ref{lemma:cannot-cross} that a set $X$ {\it crosses} an annulus $\A_{s_1, s_2}(z)$ if 
	$\overline{X} \cap \partial B_{s_1}(z) \neq \emptyset$ and $\overline{X} \cap \partial B_{s_2}(z) \neq \emptyset$.
	\begin{lemma} \label{lemma:crossing-connected}
		Let $A \Subset B_1$ be an annulus,~$t > 0$ and~$z \in B_1$ such that~${\cluster{}{t}(z)} \Subset B_1$, and $U$ a simply connected component of ${\overline{{{\cluster{}{t}(z)}}}}^c$ 
		with $\Gamma = \partial U \Subset B_1$. Every connected component of $U^c \cap A$ which crosses
		$A$ contains a connected component of ${\cluster{}{t}(z)} \cap A$ which crosses $A$. 
	\end{lemma}
	
	\begin{proof}
		See Figure \ref{fig:crossing-argument}. 
		
		Let $C$ be a connected component of $U^c \cap A$ which crosses $A$. The `filling' of $C$, $C'$, is the union of $C$ and the regions which are disconnected 
		in $A$ from the inner boundary of the annulus, the outer boundary of the annulus, or both by $C$. Note that $C'$ is a topological rectangle
		with two boundary segments which are part of the inner and outer boundaries of the annulus respectively, and two boundary segments contained in $\Gamma$. 
		
		By, \eg, the Poincar\'e-Miranda theorem~\cite{poincaremiranda}, there is either a simple arc in ${\cluster{}{t}(z)} \cap C'$ between the inner and outer boundaries 
		of the annulus, or there is a simple arc in ${\cluster{}{t}(z)}^c \cap C'$ between the two boundary segments contained in $\Gamma$. 
		In the former case, ${\cluster{}{t}(z)}$ crosses $A$ and thus there is a connected component of ${\cluster{}{t}(z)} \cap A$ in $C$ which crosses $A$. Hence it suffices to rule out the latter case, which we now do. 
		
		Suppose for sake of contradiction there is a simple open arc $P$ in ${\cluster{}{t}(z)}^c \cap C'$ between the two boundary segments contained in $\Gamma$. That is, $P$ is the interior of a simple arc contained in ${\cluster{}{t}(z)}^c \cap C'$ and $\overline{P} = P \cup \{a\} \cup \{b\}$
		with $\{a, b\} \subset \Gamma$.
		
		As $\mathrm{int}(U^c)$ is simply connected, by \cite[Proposition 2.12]{pommerenke2013boundary},  
		$\mathrm{int}(U^c) \backslash P$ has exactly two components $G_0$ and $G_1$ and these satisfy
		\begin{equation}
		\mathrm{int}(U^c) \cap \partial G_0 = \mathrm{int}(U^c) \cap \partial G_1 = P.
		\end{equation}
		As $\partial G_0$ and $\partial G_1$ contain points in $\Gamma$, the boundary of a simply connected component of 
		${\cluster{}{t}(z)}^c$, by definition of component, both $G_0$ and $G_1$ must contain points in ${\cluster{}{t}(z)}$.  However, this contradicts the fact ${\cluster{}{t}(z)}$ is connected, Lemma \ref{lemma:non-coincidence-open-harmonic}. 
		Indeed, 
		\[
		\partial G_0 \cup \partial G_1 \subset \partial U \cup P \subset {\cluster{}{t}(z)}^c
		\]
		which implies, as $G_0 \cap  {\cluster{}{t}(z)}$ and $G_1 \cap {\cluster{}{t}(z)}$ are open and disjoint,
		that they must lie in different connected components of ${\cluster{}{t}(z)}$.
	\end{proof}

	\begin{prop} \label{prop:boundary-simple-loop}
		Almost surely, for all $t > 0$ and~$z \in B_1$ such that ${\cluster{}{t}(z)} \Subset B_1$, each of the connected components of $B_1 \backslash \overline{{{\cluster{}{t}(z)}}}$
		has a boundary that is a simple loop.
	\end{prop}
	\begin{proof}
		
		Suppose ${\cluster{}{t}(z)} \Subset B_1$ and let $U$ be a component of ${\overline{{{\cluster{}{t}(z)}}}}^c$.
		Note that $U$ and ${\cluster{}{t}(z)}$ are disjoint and $\partial U \subseteq \partial {\cluster{}{t}(z)} \Subset B_1$.
		
		Let $\Gamma = \partial U \Subset B_1$. 
		By Caratheodory's theorem ~\cite[Theorem 2.6]{pommerenke2013boundary} to show $\Gamma$ is a Jordan loop
		it suffices to show that $\Gamma$ is locally connected and has no cut points. 
		
		{\it Step 1: Locally connected.} \\
		As we will show, this follows from Lemma \ref{lemma:cannot-cross} and the definition of locally connected:
		
		Suppose for the sake of contradiction that $\Gamma$ is not locally connected.
		By~\cite[Theorem 2.1]{pommerenke2013boundary}, this implies $U^c$ is not locally connected. 
		By definition, this implies the existence
		of a point $z \in U^c$ and $s > 0$ so that for every sub-neighborhood $V \subset B_{s}(z)$ containing 
		$z$, the set $V \cap U^c$ is not connected.

		As $U^c$ is connected, the closure of every component of $U^c \cap B_s(z)$ has non-empty intersection with $\partial B_s(z)$. 
		Since $U^c$ is closed, every such component not containing $z$ must lie at positive distance from $z$.
		Hence, for each $\epsilon \in (0,s)$ the number of such components intersecting $B_\epsilon(z)$ must be infinite, as
		otherwise we could take $V$ to be $B_{\epsilon}(z)$ minus the other components which do not contain $z$ which intersect $B_{\epsilon}(z)$.
		
		This implies that for all $\rho > 0$ sufficiently small there is a ball $B_{2\rho}(z)$ with infinitely many distinct components of $U^c \cap B_{2 \rho}(x_0)$ with closures intersecting $\partial B_{2\rho}(z)$ and $\partial B_{\rho/3}(z)$. 
		Every connected component of $U^c \cap B_{2 \rho}(z)$ intersected with $\A_{\rho/3, 2 \rho}(z)$ decomposes into a union of connected components 
		of $U^c \cap \A_{\rho/3, 2 \rho}(z)$. Hence, by Lemma \ref{lemma:crossing-connected} applied to each connected
		component of $U^c \cap \A_{\rho/3, 2 \rho}(z)$, every connected component of $U^c \cap B_{2 \rho}(z)$ which crosses
		$\A_{\rho/3, 2 \rho}(z)$ contains a connected component of ${\cluster{}{t}(z)} \cap \A_{\rho/3, 2 \rho}(z)$ which crosses $\A_{\rho/3, 2 \rho}(z)$.
		
		The previous paragraph implies there is a ball $B_{2\rho}(z)$ with infinitely many distinct components of ${\cluster{}{t}(z)} \cap B_{2 \rho}(x_0)$ with closures intersecting $\partial B_{2\rho}(z)$ and $\partial B_{\rho/3}(z)$, contradicting Lemma \ref{lemma:cannot-cross}.

		{\it Step 2: No cut points.} \\
		Let $\psi : B_1 \to U$ be a conformal map, which exists since $U$ is connected with connected complement, so is simply connected. Since $\Gamma$ is locally connected,~\cite[Theorem 2.1]{pommerenke2013boundary} implies that $\psi$ extends to a continuous map $\overline B_1 \to \overline U = U \cup \Gamma$. 
		
		Now, assume by way of contradiction that $\Gamma$ has a cut point $a \in \Gamma$. By~\cite[Proposition 2.5]{pommerenke2013boundary}, $\# \psi^{-1}(a) \geq 2$ (in principle $\#\psi^{-1}(a) $ could be infinite, even uncountable).
		Furthermore, if $\mathcal I$ is the set of connected components of $\partial B_1 \setminus \phi^{-1}(a)$, then the set of connected components of $\Gamma \setminus \{a\}$ is $\{\psi(I) : I\in \mathcal I\}$. 
		
		Fix some $I\in\mathcal I$ and let $J$ be equal to $\partial B_1 \setminus I$ minus its endpoints. 
		Then $I$ and $J$ are disjoint open arcs of $\partial B_1$ and their common endpoints are distinct points of $\psi^{-1}(a)$.  
		Furthermore, the preceding paragraph implies that $\psi(\bar I) \cap \psi(J) = \emptyset$ and  $\psi(I) \cap \psi(\bar J) = \emptyset$
		
		Since $\Gamma$ disconnects $0$ from $y$, the homotopy class of the loop $\psi|_{\partial B_1}$ in $(\C \cup \infty) \setminus \{0,y\}$ is non-trivial.
		Since $\psi$ maps the endpoints of $I$ and $J$ to $a$, each of $\psi|_I$ and $\psi|_J$ is a loop in $\C$, and $\psi|_{\partial B_1}$ is the concatenation of these two loops. 
		The concatenation of two homotopically trivial loops is also homotopically trivial.
		Therefore, one of $\psi|_I$ or $\psi|_J$ is not homotopic to a point in $(\C \cup \infty)\setminus \{0,y\}$. 
		This implies that one of $\psi(\bar I)$ or $\psi(\bar J)$ disconnects $0$ from $y$.   
		
		Assume without loss of generality that $\psi(\bar I)$ disconnects $0$ from $y$. 
		Since ${\cluster{}{t}(z)} \ni 0$ is connected (Lemma \ref{lemma:non-coincidence-open-harmonic}) and $U \ni y$ is connected by definition, ${\cluster{}{t}(z)}$ and $U$ are contained in different connected components of $\C\setminus \psi(\bar I)$. 
		But, every point of $\psi(J) \subset \Gamma$ is an accumulation point of both ${\cluster{}{t}(z)}$ and $U$, so $\psi(J) \subset \psi(\bar I)$. 
		Since $J$ is non-empty by construction, this contradicts the fact that $\psi(\bar I) \cap \psi(J) = \emptyset$.
		We conclude that $\Gamma$ has no cut points.
	\end{proof}
	
	The desired claim follows immediately from a scaling argument. 
	\begin{prop} \label{prop:boundary-simple-loop-full-cluster}
		Almost surely, for all $t > 0$ and~$z \in \C$ each of the connected components of $\C \backslash \overline{{\regcluster{}{t}}(z)}$
		has a boundary that is a simple loop.
	\end{prop}
	
	\begin{proof}
		Combine Lemma \ref{lemma:scale-invariance-flow} together with Lemma \ref{lemma:cannot-cross},  a union bound, 
		and the relation between ${\regcluster{}{t}}$ and $\{\cluster{B_r}{t}\}_{r > 0}$ given in Theorem \ref{theorem:full-theorem-minus-uniqueness}. 
	\end{proof}

	\section{Novelty of Harmonic balls} \label{sec:novelty}

	In this section we show that typical harmonic balls are too rough to have Lipschitz boundaries 
	yet differ in a quantitative way from LQG-metric balls.

	\subsection{Not Lipschitz} \label{subsec:not-lipschitz}
	In this section we show that a `typical' harmonic ball is not a Lipschitz domain.
	Roughly, a Lipschitz domain is a domain whose boundary can be locally represented
	by a graph of a Lipschitz function. Note that as every convex function is locally Lipschitz, see, for example, 
	\cite[Lemma 1.1.6]{gutierrez}, every convex domain, \eg,  a ball or polygon, is a Lipschitz domain.
	
	\begin{definition} \label{def:lipschitz}
		Let $A$ be a non-empty connected open set. $A$ is a  \emph{Lipschitz domain}
		if for every point $x_0 \in \partial A$, there exists $r > 0$ and a Lipschitz function
		$f: \R \to \R$ such that --- upon relabeling and reorienting the coordinate axes if necessary ---
		we have 
		\[
		U \cap B_r(x_0) = \{ z \in B_r(x_0) : \Im(z) > f(\Re(z)) \}.
		\]
	\end{definition}

	In this section we prove the following. 
	\begin{prop} \label{prop:not-balls}
		Almost surely, ${\regcluster{}{t}}$ is not a Lipschitz domain for Lebesgue-a.e.\ $t$. 
	\end{prop}
	
	To that end, we show that `typical' points on the boundary of the cluster
	do not satisfy the `cone condition'. We define these terms. 
	\begin{definition}
		A \emph{cone} $\mathcal Q \subset \C$ is a non-empty open set strictly contained in $\C$ which can be written as $\mathcal Q = \{c v + d w: c, d > 0\}$ for \emph{extremal directions} $v,w \in \C \setminus \{0\}$. We also define complements of cones to be cones. 
	\end{definition}
	
	Note that any cone $\mathcal Q$ is scale invariant, that is, for any $\theta > 0$, $\theta \mathcal Q = \mathcal Q$. 
	For a cone $\mathcal Q$ and $z \in \C$, we write $\mathcal Q(z) := \mathcal Q + z$. 
	
	\begin{definition} \label{def:cone-condition}
		A domain $A \subset \C$ satisfies the \emph{interior} (respectively, \emph{exterior}) \emph{cone condition} at $z \in \partial A$
		if there is a radius $R > 0$ and a cone $\mathcal Q$ such that $\mathcal Q(z) \cap B_R(z) \subset A$ (respectively, $\mathcal Q(z) \cap B_R(z) \subset A^c$).
		$A$ satisfies the \emph{cone condition} at $z \in \partial A$ if it satisfies both the interior and exterior cone conditions at $z$.
	\end{definition}
	
	\begin{figure}
		\begin{center}
			\includegraphics[width=0.5\textwidth]{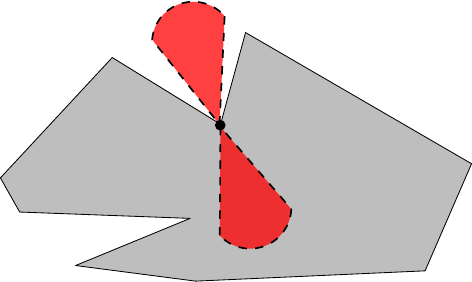}
			\caption{A domain which satisfies the cone condition of Definition \ref{def:cone-condition}. 
				The domain is in gray and the two cones are displayed in red. }
		\end{center}
	\end{figure}

	The main input in the proof of Proposition~\ref{prop:not-balls} is the following lemma.
	\begin{lemma} \label{lemma:no-exterior-cone-condition}
		Almost surely, for $\mu_h$-a.e.\ $z \in \C$, if $t > 0$ such that $z \in \partial {\regcluster{}{t}}$, then
		${\regcluster{}{t}}$ does not satisfy the cone condition at $z$. 
	\end{lemma}

	To show that that ${\regcluster{}{t}}$ does not satisfy the cone condition, we study an event 
	concerning the oscillation of the Liouville measure across sectors of annuli. 
	Before doing so, we motivate the event we consider by reformulating some of the harmonic comparison lemmas
	from Section \ref{subsec:odometer-cluster-comparison} into the following. 
	\begin{lemma} \label{lemma:harmonic-comparison-chain}	
		There exists a universal constant $C > 0$ such that the following holds for each $r > 0$ and $z \in \C$ such that 
		$\overline B_{10 r}(z) \cap \{0\} = \emptyset$.
		If  
		$\A_{4 r , 5r }(z) \cap {{\regcluster{}{t}}}^c \neq \emptyset$, we have
		\[
		\mu_h(\A_{r,2 r}(z) \cap {\regcluster{}{t}}) \leq C \times \SGrho{3r, 6r}{z}.
		\]
		where $\SGrho{3r, 6r}{z}$ is as in \eqref{eq:alternate-size-2}. 
		
	\end{lemma}
	\begin{proof}
		We use Lemma \ref{lemma:odometer-annuli} and Lemma \ref{lemma:odometer-bound-lqg-mass} which were both 
		stated for $\cluster{B_1}{t}$, $\odometer{B_1}{t}$ but whose proofs and statements extend verbatim to $\cluster{B_R}{t}$, $\odometer{B_R}{t}$ for any $R > 0$. 
		In the following chain of inequalities $C$ refers to a universal constant 
		which may change from line to line,
		\begin{align*}
		&\mu_h({\regcluster{}{t}} \cap \A_{r, 2r}(z)) \\
		&= \mu_h(\cluster{B_R}{t} \cap \A_{r, 2 r}(z)) \qquad \mbox{(as for some $R > 0$, $\mu_h({\regcluster{}{t}} \triangle \cluster{B_R}{t})=0$ by Theorem \ref{theorem:full-theorem-minus-uniqueness})} \\
		&\leq C \times \sup_{\A_{0.5 r, 4 r}(z)} \odometer{B_R}{t} \qquad \mbox{(by Lemma \ref{lemma:odometer-annuli} with $(s_1,s_2,s_3,s_4)=(0.5, 1, 2,4)$)} \\
		&\leq C \times \sup_{\partial B_{4 r}(z)} \odometer{B_R}{t} \qquad \mbox{($\odometer{B_R}{t}$ is subharmonic in $B_{10 r}(z)$)} \\
		&\leq C \times \SGrho{3r, 6r}{z} \qquad \mbox{(by Lemma \ref{lemma:odometer-bound-lqg-mass} with $(s_1,s_2,s_3,s_4)=(3,4,5,6)$)}.
		\end{align*}
		
	\end{proof}

	For a cone $\mathcal Q$, $M > 0$,  $z \in \C$, and $r > 0$, consider the event
	\begin{equation} \label{eq:mu-h-wild}
	\tilde G_{r}(z) = \tilde G_{r}(z; M, \mathcal Q) := \left\{\mu_{h}(\A_{r, 2r}(z) \cap \mathcal Q(z)) \geq  M \times \SGrho{3r, 6r}{z} \right\}
	\end{equation}
	and observe that $\tilde G_{r}(z) \in \sigma(h |_{\A_{r, 6r}(z)})$.
	We abbreviate
	\begin{equation}
	\tilde G_n(z) = \tilde G_{2^{-n}}(z), \quad \forall n\in\N 
	\end{equation}
	and prove an analogue of Lemma \ref{lemma:mu_h-doubling-and-harnack} for rare events.  
	\begin{lemma} \label{lemma:mu_h-bad-and-harnack}
		Let $\mathcal Q$ be a cone and let $M > 0$. There exists $\delta = \delta(M) \in (0,1)$ such that a.s.\ for $\mu_h$-a.e.\ $z\in  \C$, it holds for each large enough $N\in \N$ (depending on $z$) that
		\begin{equation} \label{eq:mu_h-bad}
		\#\left\{ n \in [N+1,2N] \cap \Z  : \mbox{$\tilde G_n(z; M, \mathcal Q)$ occurs} \right\}  \geq \delta N.
		\end{equation}  
	\end{lemma}
	In order to prove this lemma, we first prove the following.
	\begin{lemma} \label{lemma:positive-probability}
		Let $\mathcal Q$ be a cone and let $M > 0$. There exists a constant $p_M \in (0,1)$ so  that
		\begin{equation} \label{eq:desired-claim}
		P[ \tilde G_n(z; M, \mathcal Q)] \geq p_M, \quad \mbox{$\forall z$ s.t. $\dist(\A_{2^{-n}, 6 \times 2^{-n}}(z), \{0\}) > 2^{-n}/100$,  $\forall n \geq N(z)$.}
		\end{equation}
	\end{lemma}

	\begin{proof}
		
		Recall from \eqref{eq:gff} that $h = h^\C - \boldsymbol{\alpha}_0 \log |\cdot|$. 
		We first show \eqref{eq:desired-claim} in the case $\boldsymbol{\alpha}_0 = 0$, 
		and then use Weyl scaling to get the general case.

		In the following, we write $SG_{3r ,6r}^h(z)$ instead of $\SGrho{3r,6r}{z}$ to indicate the dependence on the underlying field $h$. 
		Note that $SG^h_{3r ,6r }(z)$ depends only on $h$ restricted to $\A_{3r, 6r}(z)$. 
		
		{\it Step 1: $\boldsymbol{\alpha}_0 = 0$.} \\
		By definition, if $\boldsymbol{\alpha}_0 = 0$ then $h$ is a whole-plane GFF.
		We have chosen the event $\tilde G_n$ so that it is a.s.\ determined by $h $ viewed modulo additive constant. Hence, by the scale and translation invariance of the law of the whole-plane GFF~\eqref{eq:gff-scaling} it suffices to bound $\P[\tilde G_1(0)]$ from below. This lower bound is achieved via the `adding a bump function' technique. 
		
		Since the random variables involved are finite and positive, there are $\gamma$-dependent constants $C_1$ and $C_2$ so that
		\begin{equation} \label{eq:l-and-r-bounds}
		\P\left[\mu_h(\A_{1,2} \cap \mathcal Q) \geq C_1 \mbox{ and }  SG^h_{3,6} \leq C_2\right] \geq 1/2.
		\end{equation}
		Let $\phi$ be a smooth, non-negative bump function 
		which is identically equal to $\gamma^{-1} \log \frac{M C_2}{C_1}$ on $\A_{1,2}$ and identically equal to 0 on $\A_{0.5,2.5}^c$. On the event in \eqref{eq:l-and-r-bounds}
		\begin{align*}
		\mu_{h+\phi}(\A_{1,2} \cap \mathcal Q) &= M \times \mu_h(\A_{1,2} \cap \mathcal Q) \times  \frac{C_2}{C_1} \qquad \mbox{(by Weyl scaling)} \\
		&\geq M \times C_2 \qquad \mbox{($\mu_h(\A_{1,2} \cap \mathcal Q) \geq C_1$ by the event)} \\
		&\geq M \times SG^h_{3,6}  \qquad \mbox{($ C_2 \geq SG^h_{3,6}$ by the event \eqref{eq:l-and-r-bounds})} \\
		&= M \times SG^{h+\phi}_{3,6} \qquad \mbox{($\phi \equiv 0$ on $\A_{3, 6}$)}.
		\end{align*}	
		Since the laws of $h$ and $h+\phi$ are mutually absolutely continuous viewed modulo additive constant~\cite[Proposition 2.9]{miller2017imaginary} this implies
		\[
		\P[\tilde G_1(0)] := p_M > 0,
		\]		
		completing the proof of Step 1.

		{\it Step 2: $\boldsymbol{\alpha}_0 \in (-\infty, Q)$.} \\
		Write $r := 2^{-n}$.
		In the general case, we fix $z$ and take $n \geq N(z)$ sufficiently large so that 
		\[
		\dist(\A_{r, 6 r}(z), \{0\}) > r/100.
		\]
		By Weyl scaling, Fact \ref{fact:lqg-measure}, 
		\[
		\frac{\mu_h(\A_{r,2r}(z) \cap \mathcal Q)}{SG^h_{3r,6r}(z)} \geq C \frac{\mu_{h^\C}(\A_{r,2r}(z) \cap \mathcal Q)}{SG^{h^\C}_{3r,6r}(z)}, \quad \mbox{for a universal constant $C > 0$}. 
		\]
		By Step 1, with probability $p_{M/C}$, 
		\[
		\frac{\mu_{h^\C}(\A_{r,2r} \cap \mathcal Q)}{SG^{h^\C}_{3r,6r}} \geq M/C,  
		\]
		completing the proof. 
	\end{proof}
	
	This leads to a proof of Lemma \ref{lemma:mu_h-bad-and-harnack}.
	\begin{proof}[Proof of Lemma \ref{lemma:mu_h-bad-and-harnack}]
		The proof is nearly identical to that of Lemma \ref{lemma:mu_h-doubling-and-harnack}, the only difference
		being the event under consideration only has positive probability meaning we get $\delta N$ instead of $(1-\zeta) N$ `good' scales.
		
		Specifically, we may carry out the proof of Lemma \ref{lemma:doubling-alpha} and substitute Lemma \ref{lemma:positive-probability}
		as the bound in \eqref{eq:doubling-and-harnack-prob}. 
		This shows that there exists $\delta = \delta(M) \in (0,1)$ such that it holds for each large enough $N\in \N$ 
		\begin{equation} \label{eq:log-singularity-event}
		\#\left\{ n \in [N+1,2N]\cap \Z   : \mbox{$\tilde G_n(0; M, \mathcal Q)$ occurs} \right\}  \geq \delta N.
		\end{equation}  
		Then the argument in the proof of Lemma \ref{lemma:mu_h-doubling-and-harnack} together with 
		\eqref{eq:log-singularity-event} implies a.s.\ the lemma statement holds for $\mu_h$-a.e.\ $z\in B_1$ and then by scaling
		for $\mu_h$-a.e.\ $z\in \C$.
	\end{proof}

	We use Lemma \ref{lemma:mu_h-bad-and-harnack} to show that `$\mu_h$-typical points' on the boundaries of clusters do not satisfy the cone condition. 
	
	\begin{proof}[Proof of Lemma \ref{lemma:no-exterior-cone-condition}]
		Take $M = 2 C$ where $C$ is the universal constant from Lemma \ref{lemma:harmonic-comparison-chain}.
		Condition on $h$ and sample $z$ from $\mu_h$. 
		
		By Lemma \ref{lemma:mu_h-bad-and-harnack}, a.s.\ for every cone $\mathcal Q$ with rational extremal directions $v,w$ there exist arbitrarily large $n \in \N$  (depending on $z$ and $\mathcal Q$) for which the event $\tilde G_n(z; M, \mathcal{Q})$ occurs.  
		Since every cone not equal to all of $\C$ is contained within a cone with rational extremal directions, we get that a.s.\ for \emph{every} cone $\mathcal Q$ there exist arbitrarily large $n \in \N$  (depending on $z$ and $\mathcal Q$) for which the event $\tilde G_n(z; M, \mathcal{Q})$ occurs.  
		
		Suppose $t>0$ is such that $z \in \partial {\regcluster{}{t}}$.
		and let $\mathcal Q(z)$ be a cone with apex at $z$. Let $n$ be such that $\tilde G_n(z; M, \mathcal{Q})$ holds
		and $B_{10 \times 2^{-n} }(z) \cap \{0\} = \emptyset$. Write $r = 2^{-n}$. By the contrapositive of Lemma \ref{lemma:harmonic-comparison-chain}, 
		\begin{equation} \label{eq:large-mass-implies-full}
		\mu_h(A_{r, 2r}(z) \cap {\regcluster{}{t}}) > C \times \SGrho{3r,6r}{z} \implies \A_{4r, 5r}(z) \cap {{\regcluster{}{t}}}^c = \emptyset.
		\end{equation}
		Since the event $\tilde G_n(z; M, \mathcal{Q})$ occurs with $M = 2C$, 
		\begin{equation} \label{eq:mu_h-is-bad}
		\mu_h(A_{3r, 5r}(z) \cap \mathcal{Q}(z)) \geq M \times \SGrho{3r,6r}{z} >  C \times \SGrho{3r,6r}{z}.
		\end{equation}
		
		We, however, also have the following implications of the cone condition: 
		\begin{equation} \label{eq:interior-cone-implies-full}
		B_R(z) \cap \mathcal{Q}(z) \subset {\regcluster{}{t}}    \implies \mu_h(A_{r, 2r}(z) \cap {\regcluster{}{t}}) \geq 		\mu_h(A_{r, 2r}(z) \cap \mathcal{Q}(z)), \quad \forall R > 2 r
		\end{equation}
		and for any cone $\mathcal{Q}'$,
		\begin{equation} \label{eq:exterior-cone-implies-empty}
		B_R(z) \cap \mathcal{Q}'(z) \subset {{\regcluster{}{t}}}^c    \implies \A_{4 r, 5 r}(z) \cap {\regcluster{}{t}}^c \neq \emptyset, \quad \forall R > 5 r.
		\end{equation}
		For any $R > 0$, we can take $n$ sufficiently large so that $R > 5\times 2^{-n}=  5 r$. Thus, the  interior cone condition \eqref{eq:interior-cone-implies-full} together with \eqref{eq:mu_h-is-bad} implies by 
		\eqref{eq:large-mass-implies-full} that  $\A_{4r, 5r}(z) \cap {{\regcluster{}{t}}}^c = \emptyset$ which is incompatible
		with the exterior cone condition \eqref{eq:exterior-cone-implies-empty}.
	\end{proof}

	We use this to show that for Lebesgue a.e.\ $t$, ${\regcluster{}{t}}$ is not a Lipschitz domain. 
	
	\begin{proof}[Proof of Proposition \ref{prop:not-balls}]
		The set ${\regcluster{}{t}}$ is parameterized so that $\mu_h({\regcluster{}{t}})  = t$, so if $A$ is a Lebesgue measurable subset of $[0, \infty)$, then
		\begin{equation}\label{eq:mass-times}
		\mu_h(\{z : z \in \partial {\regcluster{}{t}} \mbox{ for some $t \in A$}\}) = \mbox{Leb}(A) ,
		\end{equation} 
		where $\mbox{Leb}$ denotes one-dimensional Lebesgue measure. This follows from the standard machine. Indeed, \eqref{eq:mass-times} holds for intervals,  $A = [a,b]$,
		\[
		\mu_h(\{z : z \in \partial {\regcluster{}{t}} \mbox{ for some $t \in [a,b]$}\}) = \mu_h( \overline{\regcluster{}{b}} \backslash 
		\regcluster{}{a}) = 
		(b-a),
		\]
		by Theorem \ref{theorem:full-theorem-minus-uniqueness}.
		By approximation, this implies \eqref{eq:mass-times} holds for all Lebesgue measurable subsets of $[0,\infty)$.

		By~\eqref{eq:mass-times} applied to the set 
		\[
		A = \{t > 0 : \mbox{ ${\regcluster{}{t}}$ satisfies the cone condition at each $z\in \partial {\regcluster{}{t}}$}\}
		\]
		together with Lemma \ref{lemma:no-exterior-cone-condition}, we get that a.s.\ the Lebesgue measure of the set of $t>0$ for which ${\regcluster{}{t}}$ satisfies the cone condition at each boundary point is zero. One easily gets from Definition~\ref{def:lipschitz} that every Lipschitz domain satisfies the cone condition at each of its boundary points. Hence, a.s.\ ${\regcluster{}{t}}$ is not a Lipschitz domain for a.e.\ $t>0$.
	\end{proof}
	
	\subsection{Small diameter and large LQG mass} \label{subsec:small-diamter-large-lqg-mass}
	
	Recall from Section \ref{subsec:lqg} that $D_h$ denotes the $\gamma$-LQG metric associated with $h$. 
	We will eventually show that ${\regcluster{}{t}}$ is not an LQG metric ball for a.e.\ $t>0$ by showing that LQG metric balls do not satisfy the Harnack-type condition of Section~\ref{sec:harnack-type-estimate}. For this purpose we will need to force an LQG metric ball to contain certain sets of large $\mu_h$-mass. In order to do this, we will need the following proposition, which we prove in this subsection. 
	
	\begin{prop} \label{prop:mass-diam-gen}
		Let $U \subset V \subset W\subset  \C$ be bounded, connected open sets such that $\overline V\subset W$ and $W$ does not intersect the unit circle $\partial  B_1(0)$.
		For each $C>\epsilon > 0$, it holds with positive probability (depending on $U,V,W,C,\epsilon$) that
		\begin{equation} \label{eq:mass-diam-gen}
		\sup_{u,v\in V} D_h\left( u , v ; W \right) \leq \epsilon \quad \text{and} \quad \mu_h(U) \geq C ,
		\end{equation}
		where here we use the notation for the internal metric from~\eqref{eq:internal-metric}.
	\end{prop}
	It is not obvious how to apply the `adding a bump function' technique used in Section \ref{subsec:not-lipschitz} to prove Proposition~\ref{prop:mass-diam-gen} since if the bump function $\phi$ is positive, then adding $\phi$ tends to increase both $D_h$ and $\mu_h$, and the reverse is true if $\phi$ is negative. So, some work is needed to simultaneously make the $D_h$-diameter small and the $\mu_h$-mass large. 
	
	We first prove a version of Proposition~\ref{prop:mass-diam-gen} for dyadic squares.
	\begin{figure}[ht!]
		\begin{center}
			\includegraphics[width=.4\textwidth]{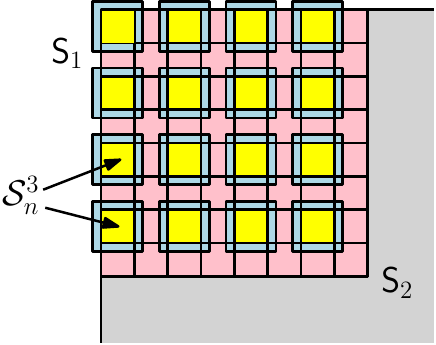}  
			\caption{\label{fig-mass-diam} Illustration of the proof of Lemma~\ref{lemma:mass-diam}. The squares $\mathsf S_1 \subset \mathsf S_2$ are shown in pink and gray, respectively (only part of $\mathsf S_2$ is shown). The smaller dyadic squares $S\in\mathcal  S_n^3$ are shown in yellow and the corresponding larger squares $\hat S$ are shown in light blue. The set $G$ is equal to $\mathsf S_2$ minus the squares $\hat S$ for $S\in \mathcal  S_N^i$, for an appropriate deterministic choice of $N \geq n_0$ and $i\in \{1,2,3,4\}$. 
			}
		\end{center}
		\vspace{-3ex}
	\end{figure}
	\begin{lemma} \label{lemma:mass-diam}
		Fix two closed dyadic squares $\mathsf S_1 \subset \mathsf S_2$ such that $\mathsf S_2$ does not intersect the Euclidean unit circle $\partial  B_1(0)$.
		For each $C >\epsilon > 0$, it holds with positive probability (depending on $\mathsf S_1, \mathsf S_2, C,\epsilon$) that
		\begin{equation} \label{eq:mass-diam}
		\sup_{u,v\in \mathsf S_2 } D_h\left( u , v ; \mathsf S_2 \right) \leq \epsilon \quad \text{and} \quad \mu_h(\mathsf S_1) \geq C .
		\end{equation}
	\end{lemma}
	\begin{proof}
		See Figure~\ref{fig-mass-diam} for an illustration. 
		Let $n_0 \in \mathbb Z$ be chosen so that the side length of $\mathsf S_1$ is $2^{-n_0}$. 
		For $n \geq n_0$, let $\mathcal  S_n$ be the set of closed $2^{-n  } \times 2^{-n }$ squares $S$ which are contained in $\mathsf S_1$. 
		Since $\mathsf S_1$ is dyadic, $\mathsf S_1$ is the union of the squares in $\mathsf S_1$. 
		Furthermore, $\mathcal  S_n$ is the disjoint union of the following four sets of squares:
		\begin{align*}
		\mathcal  S_n^1 &:= \left\{[(k-1) 2^{-n} ,k 2^{-n}] \times [(m-1) 2^{-n} , m 2^{-n}] \in \mathcal  S_n : \text{$k$ is even, $m$ is even} \right\} \\
		\mathcal  S_n^2 &:= \left\{[(k-1) 2^{-n} ,k 2^{-n}] \times [(m-1) 2^{-n} , m 2^{-n}] \in \mathcal  S_n : \text{$k$ is even, $m$ is odd} \right\} \\
		\mathcal  S_n^3 &:= \left\{[(k-1) 2^{-n} ,k 2^{-n}] \times [(m-1) 2^{-n} , m 2^{-n}] \in \mathcal  S_n : \text{$k$ is odd, $m$ is even} \right\} \\
		\mathcal  S_n^4 &:= \left\{[(k-1) 2^{-n} ,k 2^{-n}] \times [(m-1) 2^{-n} , m 2^{-n}] \in \mathcal  S_n : \text{$k$ is odd, $m$ is odd} \right\} . 
		\end{align*}
		
		As $\mu_h(\mathsf S_1)$ is a strictly positive random variable, there exists a constant $C_1 > 0$ so that $\P[\mu_h(\mathsf S_1) \geq C_1] \geq 1/2$. 
		Let $\phi : [0,1] \to \mathbb C$ be a smooth bump function which is identically equal to 1 on $\mathsf S_1$ and which is identically equal to zero outside 
		a neighborhood of $\overline{\mathsf S_1}$ and let $\tilde h := h + \frac{\phi}{\gamma} \log(4 C/C_1)$. By the Weyl scaling property of $\mu_h$, Fact \ref{fact:lqg-measure}, on the positive-probability event $\{\mu_h(\mathsf S_1) \geq C_1\}$, we have
		$\{\mu_{\tilde h}(\mathsf S_1) \geq 4 C\}$. Therefore, by absolute continuity, there is $p= p(\mathsf S_1, C) > 0$ such that with probability at least $p$,
		\begin{equation} \label{eqn-square-mass-lower}
		\mu_h(\mathsf S_1) \geq 4C .
		\end{equation}

		For $S\in\mathcal  S_n$, define the larger square
		\begin{equation}
		\hat S := \left( \text{square of side length $2^{-n}  + 2^{-n-2}$ with same center as $S$} \right). 
		\end{equation}
		By, \eg,~\cite[Lemma 3.19]{dubedat2020weak} and a union bound over all $S\in\mathcal  S_n$, the supremum over all $S\in \mathcal  S_n$ of the $D_h(\cdot,\cdot ; \hat S \cap \mathsf S_2 )$ diameter of $\hat S \cap \mathsf S_2$ tends to zero in probability as $n\to\infty$. Consequently, we can find a deterministic $N \geq n_0$ such that with probability at least $p/2$, \eqref{eqn-square-mass-lower} holds and also 
		\begin{equation} \label{eqn-small-square-diam}
		\sup_{S\in\mathcal  S_N} \sup_{u,v\in  \hat S \cap \mathsf S_2  } D_h(u,v ; \hat S \cap \mathsf S_2  ) \leq \frac{\epsilon}{2} .
		\end{equation}
		Since $\mathsf S_1$ is the union of $\bigcup_{S\in \mathcal  S_n^i} S$ for $i=1,2,3,4$, on the event that~\eqref{eqn-square-mass-lower} holds there exists $i\in \{1,2,3,4\}$ such that $\mu_h\left(  \bigcup_{S\in \mathcal  S_n^i} S \right) \geq C$. Since~\eqref{eqn-square-mass-lower} and~\eqref{eqn-small-square-diam} hold simultaneously with probability at least $p/2$, we can find a deterministic choice of $i\in \{1,2,3,4\}$ such that with probability at least $p/8$, 
		\begin{equation} \label{eqn-small-square-mass-diam}
		\mu_h\left(  \bigcup_{S\in \mathcal  S_N^i} S \right) \geq C \quad \text{and} \quad \sup_{S\in\mathcal  S_N^i} \sup_{u,v\in  \hat S \cap \mathsf S_2 } D_h(u,v ; \hat S \cap \mathsf S_2 ) \leq \frac{\epsilon}{2}  .
		\end{equation}
		Any two squares in $S_N^i$ lie at Euclidean distance at least $2^{-n}$ from each other, so the set 
		\begin{equation}
		G := \mathsf S_2 \setminus \bigcup_{S\in\mathcal  S_N^i} \hat S
		\end{equation}
		is connected. Since $\overline G$ is a finite union of closed Euclidean squares, it follows from, for example, Lemma 3.9 in \cite{dubedat2020weak} that a.s.\ the $D_h(\cdot,\cdot;\overline G)$ diameter of $\overline G$ is finite. Hence, we can find a deterministic $A > 0$ such that with probability at least $p/16$, \eqref{eqn-small-square-mass-diam} holds and also
		\begin{equation} \label{eqn-exterior-diam} 
		\sup_{u,v\in \overline G} D_h(u,v;\overline G) \leq A .
		\end{equation}
		
		Let $\phi : \mathbb C\to [0,1]$ be a smooth compactly supported bump function which is identically equal to 1 on $\overline G$ and which is identically equal to zero on $\partial  B_1(0) \cup \bigcup_{S \in \mathcal  S_N^i} S$. Let
		\begin{equation}
		\tilde h := h - \frac{\phi}{\xi} \log(2 A /  \epsilon). 
		\end{equation}
		By the Weyl scaling properties of $\mu_h$ and $D_h$, if~\eqref{eqn-small-square-mass-diam} and~\eqref{eqn-exterior-diam} hold (which happens with probability at least $p/16$), then 
		\begin{equation} \label{eq:use-weyl}
		\mu_{\tilde h}\left(  \bigcup_{S\in \mathcal  S_N^i} S \right) \geq C ,\quad 
		\sup_{S\in\mathcal  S_N^i} \sup_{u,v\in  \hat S \cap \mathsf S_2  } D_{\tilde h}(u,v ; \hat S \cap \mathsf S_2   ) \leq \frac{\epsilon}{2} , \quad \text{and}\quad 
		\sup_{u,v\in \overline G} D_{\tilde h}(u,v;\overline G) \leq \frac{\epsilon}{2} .
		\end{equation}
		The second and third conditions in~\eqref{eq:use-weyl} together with the triangle inequality imply that
		\begin{equation} \label{eqn-shifted-field-weyl}
		\sup_{u,v\in \mathsf S_2 } D_{\tilde h}\left( u , v ;\mathsf S_2 \right) \leq \epsilon .
		\end{equation}
		Furthermore, the first condition in~\eqref{eq:use-weyl} implies that $\mu_{\tilde h}(\mathsf S_1) \geq C$. 
		
		Since $\phi\equiv 0$ on $\partial  B_1(0)$, the average of $\tilde h$ over $\partial  B_1(0)$ is zero. By standard absolute continuity results for the GFF (see, e.g.,~\cite[Proposition 2.9]{miller2017imaginary}), the laws of $h$ and $\tilde h$ are mutually absolutely continuous. The previous paragraph tells us that with probability at least $p/16$, \eqref{eq:mass-diam} holds with $\tilde h$ in place of $h$. Therefore, \eqref{eq:mass-diam} holds with positive probability for $h$.
	\end{proof}

	\begin{proof}[Proof of Proposition~\ref{prop:mass-diam-gen}]
		Since $U$ is open, we can find deterministic closed dyadic squares $\mathsf S_1 \subset \mathsf S_2 \subset U$ with the property that $\mathsf S_1$ is contained in the interior of $\mathsf S_2$. By Lemma~\ref{lemma:mass-diam}, it holds with positive probability that
		\begin{equation} \label{eqn-use-mass-diam}
		\sup_{u,v\in \mathsf S_2 } D_h\left( u , v ; \mathsf S_2 \right) \leq \frac{\epsilon}{2}  \quad \text{and} \quad \mu_h(\mathsf S_1) \geq C .
		\end{equation}
		
		Our hypotheses on $V,W$ and $\mathsf S_1,\mathsf S_2$ imply that the closure of $V\setminus \mathsf S_2$ is contained in the interior of $W\setminus \mathsf S_1$. Hence, we can find an intermediate open set $O$ such that 
		\begin{equation}
		\overline{V \setminus \mathsf S_2} \subset  O \quad \text{and} \quad \overline O \subset  W\setminus \mathsf S_1 .
		\end{equation}
		Since $D_h$ induces the Euclidean topology and $V\setminus \mathsf S_2$ is connected, there exists a deterministic $A > 0$ such that with positive probability,~\eqref{eqn-use-mass-diam} holds and also
		\begin{equation} \label{eqn-exterior-diam-gen}
		\sup_{u,v\in V\setminus \mathsf S_2} D_h(u,v; O) \leq A .
		\end{equation}
		
		We now use a `subtracting a bump function' argument similar to the one at the end of the proof of Lemma~\ref{lemma:mass-diam}. 
		Let $\phi : [0,1] \to \mathbb C$ be a smooth bump function which is identically equal to 1 on $O$ and which is identically equal to zero outside of $W\setminus \mathsf S_1$. 
		Let
		\begin{equation}
		\tilde h := h - \frac{\phi}{\xi} \log(2 A / \epsilon) .
		\end{equation}
		By the Weyl scaling properties of $\mu_h$ and $D_h$, on the positive-probability event that~\eqref{eqn-use-mass-diam} and~\eqref{eqn-exterior-diam-gen} hold,
		\begin{equation} \label{eq:use-weyl-gen}
		\sup_{u,v\in \mathsf S_2 } D_{\tilde h}\left( u , v ; \mathsf S_2 \right) \leq \frac{\epsilon}{2} ,
		\quad  \mu_{\tilde h}(\mathsf S_1) \geq C  ,   \quad \text{and} \quad
		\sup_{u,v\in V\setminus \mathsf S_2} D_{\tilde h}(u,v; O) \leq \frac{\epsilon}{2} .
		\end{equation}
		By the triangle inequality, \eqref{eq:use-weyl-gen} implies~\eqref{eq:mass-diam-gen} with $\tilde h$ in place of $h$. Since the laws of $h$ and $\tilde h$ are mutually absolutely continuous~\cite[Proposition 2.9]{miller2017imaginary}, we conclude the proof. 
	\end{proof}

	\subsection{Not LQG metric balls} \label{subsec:not-lqg-metric-ball}

	In this subsection we prove the following. 
	
	\begin{prop} \label{prop:not-metric-balls}
		Almost surely, ${\regcluster{}{t}}$ is not an LQG-metric ball for Lebesgue-a.e.\ $t > 0$. 
	\end{prop}

	We follow a strategy similar to that of Section \ref{subsec:not-lipschitz}
	although the arguments are slightly more complicated. To that end, 
	let $U \Subset T$ and $Q' \Subset Q$ be connected open sets as shown in Figure \ref{fig:pinch-point}. We require that $U$ is a ball contained in $\A_{1,2}$, that $Q'$ and $Q$ are contained in $\A_{4,5} \setminus \overline T$, and that $T$ is a keyhole-shaped region contained in $\A_{1,5}$ with $\partial B_{5} \subset \partial T$, as shown in the figure.
	For a set $A$, $r>0$, and $z\in\C$, we write
	\[
	A_r(z) := r A +  z.
	\]

	\begin{figure}
		\begin{center}
			\includegraphics[width=0.75\textwidth]{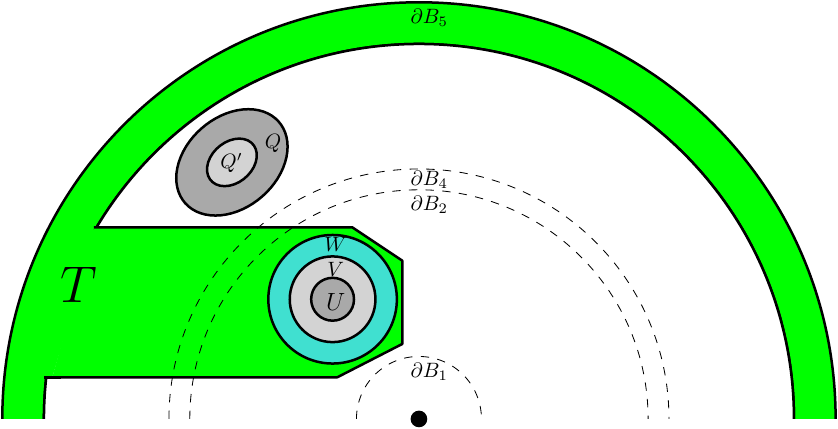}  
			\caption{Sets $U,V,W,T,Q, Q'$ used in the proof of Lemma \ref{lemma:mu_h-pinch-positive-probability} and to define the event \eqref{eq:pinch-point-event}.
				The origin is the black dot; the set $T$ is the green key-hole shaped set;
				$Q' \Subset Q$ are gray ellipses; and $U \Subset V \Subset W$ are gray concentric balls.
				The picture is not drawn to scale and only part of the set $T$ is shown. 
				The sets $V$ and $W$ are only used in the proof of Lemma \ref{lemma:mu_h-pinch-positive-probability}. 
			}
			\label{fig:pinch-point}
		\end{center}
	\end{figure}

	For $z\in\C$ and $M>0$, let $\hat G_r(z) = \hat G_r(z; M) $ denote the event that the following holds:
	\begin{equation}\label{eq:pinch-point-event}
	\begin{aligned}
	D_h(Q'_r(z), \partial Q_r(z)) &> \sup_{x \in \partial B_{5 r}(z)} D_h(z, x) \\
	D_h(z, \partial B_{5 r}(z)) &>  \sup_{x, y \in T_r(z)} D_h(x, y) \\
	\mu_h(U_r(z)) &> M \times \SGrho{3r,6r}{z} .
	\end{aligned}
	\end{equation}
	Write 
	\[
	\hat G_n(z) := \hat G_{2^{-n}}(z) .
	\]
	Showing that this event occurs with positive probability is somewhat technical and hence the proof of the following will be postponed to the end of this subsection.  
	\begin{lemma} \label{lemma:mu_h-pinch-positive-probability}
		In the case $\boldsymbol{\alpha}_0 = 0$, as in \eqref{eq:gff}, $\P[\tilde G_1(0)] > 0$. 
	\end{lemma}
	The preceding lemma implies the following.

	\begin{lemma} \label{lemma:mu_h-pinch-bad}
		Let $M > 0$. There exists $\delta \in (0,1)$ such that a.s.\ for $\mu_h$-a.e.\ $z\in  \C$, it holds for each large enough $N\in \N$ (depending on $z$) that
		\begin{equation} \label{eq:mu_h-pinch-bad}
		\#\left\{ n \in [N+1,2N]\cap \Z  : \mbox{$\hat G_n(z; M)$ occurs} \right\}  \geq \delta N.
		\end{equation}  
	\end{lemma}
	\begin{proof}
		Lemma \ref{lemma:mu_h-pinch-positive-probability} implies, by an argument similar to the proof of Lemma \ref{lemma:doubling-alpha},
		existence of a $\delta = \delta(M) \in (0,1)$ such that \eqref{eq:mu_h-pinch-bad} holds for $z = 0$ for each large enough $N\in \N$.
		The exact same argument outlined in the proof of Lemma \ref{lemma:mu_h-bad-and-harnack} then leads to the lemma statement. 
	\end{proof}

	We now observe a deterministic consequence of the event $\hat G_r(z)$ on the shape of an LQG metric ball whose boundary contains $z$. For the statement, we recall that $\mathcal B_u(0;D_h)$ denotes the open LQG metric ball of radius $u$ centered at 0.

	\begin{lemma} \label{lemma:geometric-event-consequence}
		The following holds a.s.\ for each $M > 0$, each $z \in \C$, and each $r > 0$ such that $B_{10r}(z)$ does not contain the origin.
		If $u > 0$ is such that $\hat G_r(z)$ occurs and $z \in \partial \mathcal{B}_u(0; D_h)$, then
		$Q'_r(z) \subset (\mathcal{B}_u(0; D_h))^c$ and $T_r(z) \subset \mathcal{B}_u(0; D_h)$. 
	\end{lemma}
	\begin{proof}
		The reader is encouraged to refer to Figure \ref{fig:geodesic-decomp} as a visual aid during the proof. 
		For notational convenience, write $\mathcal{B}_u := \mathcal{B}_u(0; D_h)$.  
		Assume that $z\in \partial \mathcal B_u$ and $\hat G_r(z)$ occurs, where $r$ is sufficiently small so that $B_{10r}(z)$ does not contain the origin.
		Since $z \in \partial \mathcal{B}_u$, we have $D_h(0, z) = u$. 
		
		{\it Step 1: $Q'_r(z) \subset (\mathcal{B}_u)^c$} \\	 
		Recall the definition of LQG geodesics between compact sets from just after~\eqref{eq:set-dist}.
		As $B_{10 r}(z)$ does not contain the origin, any geodesic from $0$ to $\overline{Q'_r(z)}$ can be decomposed into 
		geodesics from $0 \to q_1 \in \partial B_{5 r}(z)$, from $q_1$ to $ q_2 \in \partial Q_r(z)$, and from $q_2$ to $q_3 \in \partial Q'_r(z)$.  
		From this and the definition~\eqref{eq:pinch-point-event} of $\hat G_r(z)$, we obtain
		\begin{align*}
		D_h(0, Q'_r(z))  
		&= D_h(0, q_1) + D_h(q_1, q_2) + D_h(q_2, q_3)  \\
		&> D_h(0, q_1) + D_h(q_2, q_3) \qquad \mbox{($D_h(q_1, q_2) > 0$)} \\
		&> D_h(0, q_1) + D_h(q_1, z) \qquad \mbox{($D_h(Q'_r(z), \partial Q_r(z)) > \sup_{x \in \partial B_{5 r}(z)} D_h(z, x)$)} \\
		& \geq D_h(0, z) \qquad \mbox{(triangle inequality)}.
		\end{align*}
		Since $D_h(0,z) = u$, this implies $Q'_r(z) \subset (\mathcal{B}_u)^c$. 
		
		{\it Step 2: $T_r(z) \subset \mathcal{B}_u$} \\
		The proof is similar to Step 1. Let $y \in T_r(z)$. Also let $z_1$ be a point of $\partial B_{5r}(z)$ which is hit by a geodesic from 0 to $z$ (such a point exists since $0\notin B_{10 r}(z)$). Then
		\begin{equation} \label{eq:geodesic3-decomp}
		D_h(0, z) = D_h(0, z_1) + D_h(z_1, z).
		\end{equation} 
		We now use the definition~\eqref{eq:pinch-point-event} of $\hat G_r(z)$ to get
		\begin{align*}
		D_h(0, z) &= D_h(0, z_1) + D_h(z_1, z) \quad \mbox{\eqref{eq:geodesic3-decomp}} \\
		&\geq D_h(0, z_1) + D_h(z, \partial B_{5 r}(z)) \qquad \mbox{($z_1 \in \partial B_{5 r}(z)$)} \\
		&> D_h(0, z_1) + D_h(z_1, y) \\
		&( D_h(z, \partial B_{5 r}(z)) >  \sup_{x, y \in T_r(z)} D_h(x, y) \mbox{ and }\partial T_r(z) \supset \partial B_{5 r}(z) )  \\
		& \geq D_h(0, y) \qquad \mbox{(triangle ineq.)}.
		\end{align*}
		Since $D_h(0,z) = u$, this implies $T_r(z) \subset \mathcal{B}_u$. 
	\end{proof}

	\begin{figure}
		\begin{center}
			\includegraphics[width=0.75\textwidth]{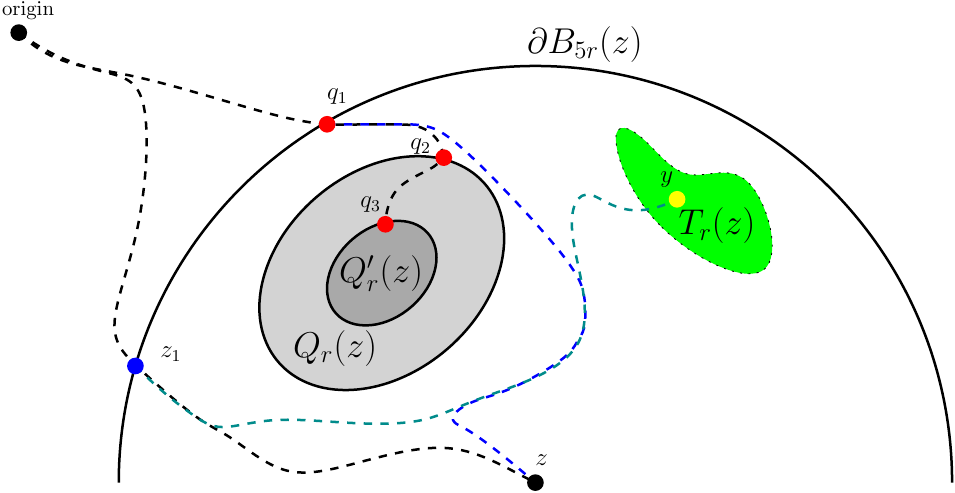}  
			\caption{Geodesic decomposition used in proof of Lemma \ref{lemma:geometric-event-consequence}. Geodesics are drawn as dashed lines. Sets are not drawn to scale
				and not all of $T_r(z)$ is shown. Note that various geodesics in the figure merge into each other. This property is called confluence of geodesics~\cite{gwynne2020confluence} and is not needed for our proofs. }
			\label{fig:geodesic-decomp}
		\end{center}
	\end{figure}
	
	Lemmas \ref{lemma:mu_h-pinch-bad} and \ref{lemma:geometric-event-consequence} immediately lead to a proof of Proposition~\ref{prop:not-metric-balls} 
	via a similar argument as the proof of Lemma \ref{lemma:no-exterior-cone-condition}.
	
	\begin{proof}[Proof of Proposition \ref{prop:not-metric-balls}]
		By the exact same argument as in the beginning of the proof of Proposition \ref{prop:not-balls}, it suffices to show that a.s.\ for $\mu_h$-a.e.\ $z$, no cluster ${\regcluster{}{t}}$
		coincides with a LQG-metric ball containing $z$ on its boundary. 
		
		Throughout the proof we take $M = C$, where $C > 0$ is the universal constant from Lemma \ref{lemma:harmonic-comparison-chain}. Almost surely, the conclusion of Lemma~\ref{lemma:mu_h-pinch-bad} (with this choice of $M$) holds for $\mu_h$-a.e.\ $z$. So, it suffices to consider a $z$ such that the conclusion of Lemma~\ref{lemma:mu_h-pinch-bad} holds and show that a.s.\ no cluster ${\regcluster{}{t}}$ coincides with an LQG metric ball which has $z$ on its boundary.	
		
		By Lemma \ref{lemma:mu_h-pinch-bad}, a.s.\ there exist arbitrarily large $n \in \N$  (depending on $z$) for which the event $\hat G_n(z; C)$ occurs. 
		Consider such an $n$ which is large enough so that $0\notin B_{10 \times 2^{-n}}(z)$ and write $r := 2^{-n}$. 
		
		Since $\hat G_n(z;C)$ occurs, Lemma \ref{lemma:geometric-event-consequence}
		implies that $Q_r'(z) \subset \mathcal{B}^c$ and $U_r (z) \subset T_r(z) \subset \mathcal{B}$. In particular, by the third inequality of the event \eqref{eq:pinch-point-event}, 
		\begin{equation} \label{eq:not-metric-balls-mass}
		\mu_h( \A_{r, 2r}(z) \cap \mathcal{B}) \geq \mu_h(U_r(z)) > C \times \SGrho{3r,6r}{z}.
		\end{equation}
		
		On the other hand, the contrapositive of Lemma \ref{lemma:harmonic-comparison-chain} shows that for each $t>0$, 
		\begin{equation} \label{eq:large-mass-implies-full-2}
		\mu_h(\A_{r, 2r}(z) \cap {\regcluster{}{t}}) > C \times \SGrho{3r,6r}{z} \implies \A_{4r, 5r}(z) \cap {{\regcluster{}{t}}}^c = \emptyset .
		\end{equation}
		Since $Q_r'(z) \Subset \A_{4r, 5r}(z)$, \eqref{eq:large-mass-implies-full-2} implies that it cannot be the case that $\mu_h(\A_{r, 2r}(z) \cap {\regcluster{}{t}}) >  C \times \SGrho{3r,6r}{z}$ and $Q_r'(z) \subset {{\regcluster{}{t}}}^c$. Since $Q_r'(z) \subset \mathcal B^c$ and by~\eqref{eq:not-metric-balls-mass}, we get that a.s.\ $\mathcal B$ is not equal to ${\regcluster{}{t}}$ for any $t>0$. 
	\end{proof}
	
	It remains to prove Lemma \ref{lemma:mu_h-pinch-positive-probability}, which we do in several steps. 
	In the remainder of the subsection, let $T,U,V,W, Q', Q$ be as in Figure \ref{fig:pinch-point}.
	In particular, $U \Subset V \Subset W \Subset T$ and $U,V,W$ are concentric Euclidean balls contained in $\A_{1,2}$.
	
	Fix $\delta_0 \in (0, 1/100)$ sufficiently small so that:
	\begin{align} \label{eq:metric-delta0-choice}
	B_{10 \delta_0}(T) \Subset \A_{1, 6} ,\quad
	B_{10 \delta_0}(U) \Subset V ,\quad
	B_{10 \delta_0}(V) \Subset W ,\notag\\
	B_{10 \delta_0}(W) \Subset \A_{1,2} ,\quad
	B_{10 \delta_0}(Q') \Subset Q ,\quad
	B_{10 \delta_0}(Q) \Subset \A_{4,5} \backslash B_{10 \delta_0}(T) .
	\end{align}  
	Recall the notation for the internal metric from~\eqref{eq:internal-metric}.
	To be succinct, we write, for a set $A$ and $\delta > 0$, 
	\begin{equation}
	\diam^\delta_h(A) := \sup_{x, y \in A} D_h(x,y; B_{\delta}(A)).
	\end{equation} 
	
	To prove Lemma \ref{lemma:mu_h-pinch-positive-probability}, we will show that several auxiliary events occur with positive probability. 
	We let $\tilde E_1$ be the event that
	\begin{equation} \label{eq:first-event}
	\begin{aligned}
	D_h(Q', \partial Q) &> \sup_{x \in \partial B_{5}} D_h(0, x; (B_{2 \delta_0}(Q \cup W))^c) \\
	D_h(0, \partial B_{\delta_0}) &> \diam^{\delta_0}_h(T \backslash B_{3 \delta_0}(W)) .
	\end{aligned}
	\end{equation}
	We let $\tilde E_2 := \tilde E_2(M)$ be the event that
	\begin{equation} \label{eq:second-event}
	\begin{aligned}
	D_h(\partial B_{\delta_0}, \partial B_{2 \delta_0}) &> \diam^{\delta_0}_h(V)  \\
	\mu_h(U) &> M \times \SGrho{3,6}{0} .
	\end{aligned}
	\end{equation} 
	We let $\tilde E_3$ be the event that
	\begin{equation} \label{eq:third-event}
	D_h(\partial B_{2 \delta_0}, \partial B_{3 \delta_0}) > \diam_h^{\delta_0}(B_{3\delta_0}(W) \backslash V).
	\end{equation}
	
	We will successively prove the following lemmas. 
	
	\begin{lemma} \label{lemma:first-event-proof}
		We have $\P[\tilde E_1] > 0$. 
	\end{lemma}
	
	\begin{lemma} \label{lemma:second-event-proof}
		For each $M  > 0$, $\P[\tilde E_1 \cap \tilde E_2(M)] > 0$. 
	\end{lemma}
	
	\begin{lemma} \label{lemma:third-event-proof}
		For each $M  > 0$, $\P[\tilde E_1 \cap \tilde E_2(M) \cap \tilde E_3] > 0$. 
	\end{lemma}
	
	Before we prove these lemmas, we show that Lemma~\ref{lemma:third-event-proof} implies Lemma \ref{lemma:mu_h-pinch-positive-probability}.
	
	\begin{proof}[Proof of Lemma \ref{lemma:mu_h-pinch-positive-probability}]
		Let $M > 0$ be given.
		By Lemma \ref{lemma:third-event-proof}, 
		it then suffices to show that 
		\begin{equation} \label{eq:events-inclusion}
		\{\tilde E_1 \cap \tilde E_2(M) \cap \tilde E_3\} \subset \hat G_1(0; M).
		\end{equation}
		
		Henceforth assume that $\tilde E_1\cap \tilde E_2(M) \cap \tilde E_3$ occurs.
		The first inequality in the definition~\eqref{eq:first-event} of $\tilde E_1$ implies
		\[
		D_h(Q', \partial Q) > \sup_{x \in \partial B_{5}} D_h(0, x; (B_{2 \delta_0}(Q \cup W))^c) \geq \sup_{x \in \partial B_{5}} D_h(0, x)
		\]
		which is the first inequality in the definition~\eqref{eq:pinch-point-event} of the event $\hat G_1(0; M)$.
		We have
		\begin{equation} \label{eq:ball-inclusion}
		T = (T \backslash B_{3 \delta_0}(W)) \cup V \cup  (B_{3 \delta_0}(W) \backslash V).
		\end{equation}
		Thus, the second inequality in the definition~\eqref{eq:first-event} of  $\tilde E_1$, the first inequality in the definition~\eqref{eq:second-event} of $\tilde E_2$, and the definition~\eqref{eq:third-event} of $\tilde E_3$ imply 
		\begin{align*}
		D_h(0, \partial B_{5}) &> D_h(0, \partial B_{\delta_0}) + D_h(\partial B_{\delta_0}, \partial B_{2 \delta_0})
		+ D_h(\partial B_{2 \delta_0}, \partial B_{3 \delta_0}) \quad \text{(triangle ineq.)}
		\\
		&\geq \diam^{\delta_0}_h(T \backslash B_{3 \delta_0}(W))
		+  \diam^{\delta_0}_h(V)
		+  \diam^{\delta_0}_h(B_{3 \delta_0}(W)) 
		\\ 
		&\geq \sup_{x, y \in T} D_h(x,y) \quad \mbox{(\eqref{eq:ball-inclusion} and triangle ineq.)} 
		\end{align*}
		which is the second inequality in the definition of $\hat G_1(0; M)$. The second inequality in the definition~\eqref{eq:second-event} of $\tilde E_2$ is the third and final inequality of $\hat G_1(0; M)$. 
	\end{proof}

	We will now show that $\P[\tilde E_1] > 0$ by adding an appropriate bump function to $h$. 
	
	\begin{proof}[Proof of Lemma \ref{lemma:first-event-proof}]
		Since the random variables involved are strictly positive and finite, 
		there exists positive finite constants $C_1, C_2, C_3, C_4$ 
		so that the event
		\begin{equation} \label{eq:stuff-finite}
		\begin{aligned}
		D_h(Q', \partial Q) &\geq C_1 \\
		\sup_{x \in \partial B_{5}} D_h(0, x; (B_{2 \delta_0}(Q \cup W))^c) &\leq C_2 \\
		D_h(0, \partial B_{\delta_0/2}) &\geq C_3  \\
		\diam^{\delta_0}_h(T \backslash B_{3 \delta_0}(W)) &\leq C_4
		\end{aligned}
		\end{equation}
		satisfies	$\P[\eqref{eq:stuff-finite}] > 0$. Henceforth assume that the event in~\eqref{eq:stuff-finite} occurs.
		
		Take smooth compactly supported bump functions $\phi_1, \phi_2 : \mathbb C\to [0,1]$
		so that
		\[
		\phi_1 \equiv \begin{cases}
		1 &\mbox{on $Q$} \\
		0 &\mbox{on $(B_{\delta_0}(Q))^c$}
		\end{cases}
		\]
		and
		\[
		\phi_2 \equiv \begin{cases}
		1 &\mbox{on  $B_{\delta_0/2}$} \\
		0 &\mbox{on  $B_{\delta_0}^c$}.
		\end{cases}
		\]
		With $\xi = \xi(\gamma)$ as in~\eqref{eq:xi}, let
		\begin{equation}
		\tilde h := h + \frac{\phi_1}{\xi} \log\left( (\frac{2 C_2}{C_1} \vee 1) \times (\frac{2 C_4}{C_3} \vee 1)\right) + \frac{\phi_2}{\xi} \log\left(\frac{2 C_4}{C_3} \vee 1\right).
		\end{equation}
		Suppose \eqref{eq:stuff-finite} holds.
		Then 
		\begin{align} \label{eq:first-event-pf1}
		D_{\tilde h}(0, B_{\delta_0}) &\geq D_{\tilde h}(0, B_{\delta_0/2}) \qquad \mbox{(positivity of length)} \notag \\
		&\geq \frac{2 C_4}{C_3} D_{h}(0, B_{\delta_0/2}) \qquad \mbox{(Weyl scaling, Fact~\ref{fact:lqg-metric})}  \notag\\
		&> \diam^{\delta_0}_h(T \backslash B_{3 \delta_0}(W)) \qquad \mbox{(the event \eqref{eq:stuff-finite})}  \notag \\ 
		&= \diam^{\delta_0}_{\tilde h}(T \backslash B_{3 \delta_0}(W)) \qquad \mbox{($\phi_1 + \phi_2 \equiv 0$ on $B_{\delta_0}(T)$)}.
		\end{align}
		
		Also note that by Weyl scaling and since the bump function $\phi_2$ has support contained in $(B_{2 \delta_0}(Q \cup W))^c$,
		\begin{equation} \label{eq:weyl-scaling-leak}
		\begin{aligned}
		&\sup_{x \in \partial B_{5}} D_{\tilde h}(0, x; (B_{2 \delta_0}(Q \cup W))^c)  \\
		&\qquad\qquad\leq (\frac{2 C_4}{C_3} \vee 1) \times  \sup_{x \in \partial B_{5}} D_{h}(0, x; (B_{2 \delta_0}(Q \cup W))^c).
		\end{aligned}
		\end{equation}
		Thus, 
		\begin{align}  \label{eq:first-event-pf2}
		D_{\tilde h}(Q', \partial Q) &\geq \frac{2 C_2}{C_1} \times (\frac{2 C_4}{C_3} \vee 1) \times D_{h}(Q', \partial Q) \qquad \mbox{(Weyl scaling)}  \notag \\
		&> (\frac{2 C_4}{C_3} \vee 1) \times \sup_{x \in \partial B_{5}} D_{h}(0, x; (B_{2 \delta_0}(Q \cup W))^c)  \qquad \mbox{(the event \eqref{eq:stuff-finite})}  \notag \\
		&\geq \sup_{x \in \partial B_{5}} D_{\tilde h}(0, x; (B_{2 \delta_0}(Q \cup W))^c) \qquad \mbox{(by \eqref{eq:weyl-scaling-leak})}.
		\end{align}
		
		By~\eqref{eq:first-event-pf1} and~\eqref{eq:first-event-pf2} and since $\P[\eqref{eq:stuff-finite}] > 0$, the event $\tilde E_1$ occurs with positive probability with $\tilde h$ in place of $h$. 
		By~\cite[Proposition 2.9]{miller2017imaginary}, the laws of $h$ and $\tilde h$ are mutually absolutely continuous, viewed modulo additive constant. By Weyl scaling the occurrence of $\tilde E_1$ is unaffected by adding a constant to $h$.
		Thus, the fact that $\tilde E_1$ occurs with positive probability with $\tilde h$ instead of $h$ implies that $\P[\tilde E_1] > 0$.   
	\end{proof}

	We next show that $\P[\tilde E_1] > 0 \implies \P[\tilde E_1 \cap \tilde E_2(M)] > 0$ using the domain Markov property
	with a set selected to be disjoint from the domain of dependence of $\tilde E_1$. 
	
	\begin{proof}[Proof of Lemma \ref{lemma:second-event-proof}]
		By the domain Markov property of the GFF, \cite[Proposition 2.8]{miller2017imaginary}, we can decompose
		\begin{equation} \label{eq:use-gff-markov}
		h = h_0 + \varphi
		\end{equation}
		where $h_0$ is a zero-boundary GFF on $W$, $\varphi$ is harmonic on $W$,
		and $h_0$ and  $\varphi$ are independent. 
		
		Since the random variables involved are strictly positive and finite and $\P[\tilde E_1] > 0$, 
		there exists positive finite constants $C_1, C_2$ so that the event
		\begin{equation} \label{eq:stuff-finite-2}
		\begin{aligned}
		D_h(\partial B_{\delta_0}, \partial B_{2 \delta_0}) &\geq C_1 \\
		M \times \SGrho{3,6}{0} &\leq C_2
		\end{aligned}
		\end{equation}
		satisfies	$\P[\tilde E_1 \cap \eqref{eq:stuff-finite-2}] > 0$.
		
		We will also need to consider the event
		\begin{equation} \label{eq:large-mass-low-diam}
		\begin{aligned}
		\diam_{h_0}^{\delta_0}(V) &\leq e^{-\xi \sup_{W} \varphi} C_1 \\
		\mu_{h_0}(U) &\geq e^{-\gamma \inf_{U} \varphi} C_2 .
		\end{aligned}
		\end{equation} 
		By the locality properties of $\mu_h$ and $D_h$ (Facts~\ref{fact:lqg-measure} and~\ref{fact:lqg-metric}), the events \eqref{eq:stuff-finite-2} and $\tilde E_1$ are both measurable with respect to the restriction of $h$ to
		\[
		Q \cup B_{2 \delta_0}(Q \cup W)^c \cup B_{2 \delta_0} \cup B_{2 \delta_0}(W)^c  \cup \A_{3,6} ,
		\] 
		which is a compact subset of $W^c$. 
		Furthermore, the event \eqref{eq:large-mass-low-diam} is measurable with respect to the restriction of $h$ to $B_{\delta_0}(V)$ and the function $\phi$ (which is measurable with respect to $h|_{W^c}$).

		By standard absolute continuity results 
		for the GFF (see, \eg, \cite[Proposition 3.4]{miller2016imaginary}) together with~\eqref{eq:use-gff-markov},
		the conditional law of $h|_{B_{\delta_0}(V)}$ given $h|_{W^c}$
		is mutually absolutely continuous with respect to its marginal law. 
		From this and Lemma~\ref{lemma:mass-diam}, we obtain
		\[
		\P[\eqref{eq:large-mass-low-diam} \,|\,  \eqref{eq:stuff-finite-2} \cap \tilde E_1] > 0 .
		\]
		Since $\P[\tilde E_1 \cap \eqref{eq:stuff-finite-2}] > 0$, we thus have
		\[
		\P[\eqref{eq:large-mass-low-diam} \cap \eqref{eq:stuff-finite-2} \cap \tilde E_1] > 0.
		\]
		
		We will now conclude the proof by showing that $\eqref{eq:large-mass-low-diam} \cap \eqref{eq:stuff-finite-2} \subset \tilde E_2$. 
		Assume that $\eqref{eq:large-mass-low-diam} \cap \eqref{eq:stuff-finite-2}$ occurs. Then
		\begin{align*}
		\diam^{\delta_0}_h(V) &\leq   e^{\xi \sup_{B_{\delta_0}(V)} \varphi}  \diam^{\delta_0}_{h_0}(V) \quad \mbox{(Weyl scaling)} \\
		&\leq D_h(\partial B_{\delta_0}, \partial B_{2 \delta_0}) \quad \mbox{(the events \eqref{eq:large-mass-low-diam} and \eqref{eq:stuff-finite-2})}
		\end{align*}
		and
		\begin{align*}
		\mu_h(U)&\geq  e^{\gamma \inf_{U} \varphi}  \mu_{h_0}(U) \quad \mbox{(Weyl scaling)} \\
		&\geq M \times \SGrho{3,6}{0} \quad \mbox{(the events \eqref{eq:large-mass-low-diam} and \eqref{eq:stuff-finite-2})}
		\end{align*}
		which is exactly the event $\tilde E_2$. 
	\end{proof}
	
	We finally show that $\P[\tilde E_1 \cap \tilde E_2(M) ] > 0 \implies \P[\tilde E_1 \cap \tilde E_2(M) \cap \tilde E_3] > 0$. 
	The proof involves adding a bump function to make $\tilde E_3$ to occur and then checking that the events
	$\tilde E_1 \cap \tilde E_2(M)$ still occur after adding the bump function.

	\begin{proof}[Proof of Lemma \ref{lemma:third-event-proof}]
		Let $\phi : \mathbb C\to [0,1]$ be a smooth compactly supported bump function such that
		\[
		\phi \equiv \begin{cases}
		1 &\mbox{on $B_{3 \delta_0}(W) \backslash B_{\delta_0}(U)$  } \\
		0 &\mbox{on $(U \cup B_{4 \delta_0}(W))^c$}.
		\end{cases}
		\]
		Since we know that $\P[\tilde E_1\cap \tilde E_2(M)] > 0$ (Lemma~\ref{lemma:second-event-proof}) and the quantities involved are a.s.\ finite and positive, we can find finite positive constants $C_1, C_2$ (depending on $M$) such that the event
		\begin{equation} \label{eq:stuff-finite-3}
		\begin{aligned}
		D_h(\partial B_{2 \delta_0}, \partial B_{3 \delta_0}) &\geq C_1 \\
		\diam^{\delta_0}_h(B_{3 \delta_0}(W) \backslash V) &\leq C_2
		\end{aligned}
		\end{equation}
		satisfies	$\P[\tilde E_1 \cap \tilde E_2 \cap \eqref{eq:stuff-finite-3}] > 0$.
		
		Let 
		\begin{equation}
		\tilde h = h + \frac{\phi}{\xi} \log(\frac{C_2}{C_1} \wedge 1) .
		\end{equation}
		Recall from~\eqref{eq:metric-delta0-choice} that $B_{\delta_0}(U) \subset V$. Hence, on $\tilde E_1 \cap \tilde E_2 \cap \eqref{eq:stuff-finite-3}$,
		\begin{align*}
		\diam^{\delta_0}_{\tilde h}(B_{3 \delta_0}(W) \backslash V) &\leq \frac{C_2}{C_1} \diam^{\delta_0}_{h}(B_{3 \delta_0}(W) \backslash V) \quad \mbox{(Weyl scaling)} \\
		&\leq D_h(\partial B_{2 \delta_0}, \partial B_{3\delta_0}) \quad \mbox{(the event \eqref{eq:stuff-finite-3})} \\
		&= D_{\tilde h}(\partial B_{2 \delta_0}, \partial B_{3\delta_0}) \quad \mbox{($\phi \equiv 0$ on $\overline B_{3 \delta_0}$)},
		\end{align*}
		which is the event $\tilde E_3$ with $\tilde h$ in place of $h$. Hence, this event has positive probability.
		
		As we will see below, by Weyl scaling,  the fact $\log(\frac{C_2}{C_1} \wedge 1) \leq 0$, and since $\phi \equiv 0$ on $U \cup \A_{3,6} \cup B_{3 \delta_0}$,
		on the event $\tilde E_1 \cap \tilde E_2(M) \cap \eqref{eq:stuff-finite-3}$ the event $\tilde E_1 \cap \tilde E_2(M)$ occurs with $\tilde h$ in place of $h$. By~\cite[Proposition 2.9]{miller2017imaginary}, the laws of $h$ and $\tilde h$, viewed modulo additive constant, are mutually absolutely continuous. Since the occurrence of the event $\tilde E_1 \cap \tilde E_2(M) \cap \tilde E_3$ is unaffected by adding a constant to $h$, we conclude that $\P[\tilde E_1\cap \tilde E_2(M) \cap \tilde E_3] > 0$, as required. 
		
		that adding the bump function did not change the occurrence of the events $\tilde E_1$ and $\tilde E_2$,
		defined in \eqref{eq:first-event} and \eqref{eq:second-event} respectively.
		The first inequality in $\tilde E_1$ is 
		\begin{align*}
		D_{\tilde h}(Q', \partial Q) &= D_h(Q', \partial Q) \quad \mbox{($\phi \equiv 0$ on $\overline Q$)} \\
		&> \sup_{x \in \partial B_{5}} D_h(0, x; (B_{2 \delta_0}(Q \cup W))^c) \quad \mbox{(event $\tilde E_1$ for $h$)} \\
		&\geq e^{- \log(\frac{C_2}{C_1} \wedge 1)} \sup_{x \in \partial B_{5}} D_{\tilde h}(0, x; (B_{2 \delta_0}(Q \cup W))^c) \quad \mbox{(Weyl scaling)} \\
		&\geq  \sup_{x \in \partial B_{5}} D_{\tilde h}(0, x; (B_{2 \delta_0}(Q \cup W))^c) \quad \mbox{($\log(\frac{C_2}{C_1} \wedge 1) \leq 0$)}.
		\end{align*}
		The second inequality in $\tilde E_1$ and the first inequality in $\tilde E_2$ are checked in a similar fashion, using 
		$\log(\frac{C_2}{C_1} \wedge 1) \leq 0$, and $\phi \equiv 0$ on $\overline{B_{2 \delta_0}}$.
		Since $\phi \equiv 0$ on $U \cup \A_{3,6}$, the last inequality in $\tilde E_2$ is preserved. 
		Hence, on the event $\tilde E_1 \cap \tilde E_2(M) \cap \eqref{eq:stuff-finite-3}$ the event $\tilde E_1 \cap \tilde E_2(M)$ occurs with $\tilde h$ in place of $h$.
	\end{proof}

	\appendix
	
	\section{Obstacle problem for Radon measures} \label{sec:obstacle-appendix}
	
	In this appendix we provide the proofs which were omitted in Section \ref{subsec:basic-properties}. 
	For clarity, we prove these results for any Radon measure $\mu$ satisfying, for some $R > 0$,
	\begin{equation} \label{eq:volume-growth}
	r^{\beta^-} \leq \mu(B_r(z)) \leq r^{\beta^-} \quad \mbox{for all $z \in B_R$},
	\end{equation}
	for some exponents $\beta^+, \beta^- > 0$, for all $r$ sufficiently small. This implies the results in Section \ref{subsec:basic-properties} as the Liouville measure is a.s.\ a Radon measure which satisfies \eqref{eq:volume-growth}. Indeed, this follows from Lemma~\ref{lemma:volume-growth} and the scaling properties of $h$ and $\mu_h$, namely~\eqref{eq:h-coordinate-change} and~\eqref{eq:measure-coord}.
	
	For $R > 0$, let 
	\begin{equation} \label{eq:expected-exit-time}
	q_{B_R}(\cdot) = \int_{B_R} G_{B_R}(y,\cdot) \mu(dy),
	\end{equation} 
	where $G_{B_R}$ is the Green's function for $B_R$. 
	Under the condition \eqref{eq:volume-growth}, the function $q_{B_R}$ satisfies the following properties: 
	\begin{enumerate}
		\item Continuous: $q_{B_R}$ is H\"{o}lder continuous in $\overline{B_R}$ and finite;
		\item Potential: $q_{B_R}$ is superharmonic and $\Delta q_{B_R} = -\mu$  in $B_R$;
		\item Zero boundary: $q_{B_R}(z) = 0$ for $z \in \partial B_R$;
		\item Positive: $q_{B_R}(x) > 0$ for $x \in B_R$.
	\end{enumerate}
	The first property follows by the same argument outlined in the proof of Proposition \ref{prop:lbm-exit-time}; 
	the second, by, for example, ~\cite[Theorem 4.3.8]{armitage2000classical}; the third as $G_{B_R}(0, \cdot) = 0$ on $\partial B_R$;
	and the fourth by the strong maximum principle. 
	
	For notational simplicity, we consider $R = 1$ in all but the last subsection. We will also only consider the case where~$z$ is the origin.
	
	\subsection{Definition} \label{subsec:obstacle-problem}
	For $t > 0$,  denote the {\it obstacle} $\beta_t: \overline{B_1} \to \R \cup \{\infty\}$ as 
	\begin{equation} \label{eq:obstacle}
	\beta_t(x) = -t G_{B_1}(0,x) + q_{B_1}(x).
	\end{equation}
	The set of  {\it supersolutions} 
	is
	\begin{equation} \label{eq:shifted-supersolutions}
	\mclSshift{}{t} = \{ w \in C(\overline{B_1}) :  \Delta w \leq 0 \mbox{ in $B_1$}  \mbox{ and }  w \geq   \beta_t \mbox{ in $\overline{B_1}$}\},
	\end{equation}
	where $C(\overline{B_1})$ denotes the set of continuous functions on the closed unit ball.
	
	Consider the {\it least supersolution} or {\it least superharmonic majorant} 
	as the pointwise infimum of all functions in $\mclSshift{}{t}$
	\begin{equation} \label{eq:appendix-least-supersolution}
	\lssshift{}{t} = \inf\{ w \in \mclSshift{}{t} \} 
	\end{equation}
	and the {\it odometer}
	\begin{equation} \label{eq:appendix-limit-odometer}
	v_t = \lssshift{}{t}  - \beta_t.
	\end{equation}
	Note that $\mclSshift{}{t}$ is non-empty as it contains $q_{B_1}$ --- thus $\lssshift{}{t}$ always exists.
	Denote the {\it non-coincidence set} by
	\begin{equation} \label{eq:appendix-non-co-set}
	\mathit{\Lambda_t} = \{ x \in B_1 : \lssshift{}{t} > \beta_t \}.
	\end{equation}
	
	Note that the least supersolution in \eqref{eq:appendix-least-supersolution} is related
	to \eqref{eq:least-super-solution} by $\lss{}{t}= \lssshift{}{t} - q_{B_1}$. In particular, $\mathit{\Lambda_t}$
	coincides with ${\cluster{}{t}}$ and $v_t$ with $\odometer{}{t}$  --- each of the lemmas in Section \ref{subsec:basic-properties} will follow via this substitution.
	We choose to work with $\mclSshift{}{t}$ as this allows us to directly cite results concerning superharmonic functions.

	\subsection{Existence} 
	
	We first verify that the solution to the obstacle problem is non-degenerate in the following sense, this implies
	Lemma \ref{lemma:basic-properties}. 
	\begin{lemma}\label{lemma:appendix-basic-properties}
		For all $t > 0$, $\lssshift{}{t}$ is finite, continuous, and an element of $\mclSshift{}{t}$. 
	\end{lemma}
	\begin{proof}
		Let 
		\begin{equation} \label{eq:lsc-leastsuperharmonic}
		f_t  := \inf \{ g: g \mbox{ is superharmonic in $B_1$ and $g \geq \beta_t$ in $\overline{B}_1$} \},
		\end{equation}
		where, as before, the infimum is pointwise. 
		Note that this definition differs from $\lssshift{}{t}$ in that admissible superharmonic functions need only be lower semicontinuous. 
		It suffices to show that $f_t$ satisfies the desired properties. Indeed, by definition $f_t \leq  \lssshift{}{t}$
		and the reverse inequality follows from $f_t \in \mclSshift{}{t}$.
		
		{\it Step 1: Finiteness.} \\
		If $g \geq \beta_t$ is superharmonic, then, by the minimum principle, $g \geq 0$ on $B_1$ as $g \geq \beta_t = 0$ on $\partial B_1$. 
		As this holds for all such $g$, $f_t \geq 0$. This together with $\infty > q_{B_1} \geq f_t$ ($q_{B_1}$ is admissible in \eqref{eq:lsc-leastsuperharmonic}) shows finiteness.
		
		{\it Step 2: Superharmonicity.} \\
		We use ~\cite[Theorem 3.7.5]{armitage2000classical}
		which we recall for the reader's convenience. Let $O$ be a bounded open set and let $f:O \to [-\infty, \infty]$.
		The {\it lower semicontinuous regularization} of $f$ is defined by 
		\begin{equation} \label{eq:lower-semicontinuous}
		\hat{f}(x) = \min \{ f(x), \liminf_{y \to x} f(y) \}.
		\end{equation}
		\cite[Theorem 3.7.5]{armitage2000classical} states that if $f > -\infty$ is the infimum of a family of superharmonic functions on $O$
		then $\hat{f}$ is superharmonic on $O$ and $\hat{f}(x) = \liminf_{y \to x} f(y)$. 
		
		Since $f_t$ is finite, we may use this to see that its lower semicontinuous regularization, $\hat{f}_t$, is superharmonic on $B_1$ and satisfies $\hat{f}_t(x) = \liminf_{y \to x} f_t(y) \leq f_t(x)$.
		In fact, \eqref{eq:lsc-leastsuperharmonic} implies $f_t$ is equal to its lower semicontinuous regularization. Indeed, $\hat{f}_t$ is superharmonic and 
		\[
		\hat{f}_t(x) = \liminf_{y \to x} f_t(y) \geq \liminf_{ y \to x} \beta_t(y) =  \beta_t(x)
		\]
		as the obstacle, $\beta_t$, is continuous, implying that $\hat{f}_t \geq f_t$ by \eqref{eq:lsc-leastsuperharmonic}.
		
		{\it Step 3: Continuity.} \\
		By Step 2, $f_t$ is lower semicontinuous. It remains to verify upper semicontinuity. 
		Let $\epsilon > 0$ and $x_0 \in B_1$ be given. By continuity of $\beta_t$, there exists $\delta > 0$ sufficiently small so that
		\[
		\beta_t(x) +\epsilon/2 \geq \beta_t(y), \quad \forall x,y \in B_{\delta}(x_0).
		\]
		and hence, since $f_t(x) \geq \beta_t(x)$,
		\begin{equation} \label{eq:lower-bound-continuity}
		\epsilon + f_t(x) \geq \beta_t(x_0) + \epsilon/2 \geq \beta_t(x), \quad \forall x \in B_{\delta}(x_0).
		\end{equation}
		
		Let $g_1$ be the unique function which is harmonic in $B_{\delta}(x_0)$ and coincides with $\epsilon + f_t$ on $\partial B_{\delta}(x_0)$. 
		Note that since $\epsilon + f_t - g_1$ is superharmonic in $B_{\delta}(x_0)$ and equal to 0 on $\partial B_{\delta}(x_0)$, 
		\begin{equation} \label{eq:superharmonic-poisson-inequality}
		\epsilon + f_t \geq g_1, \quad \mbox{on $B_{\delta}(x_0)$}.
		\end{equation}
		Define the function $g$ to be $g_1$ on $B_{\delta}(x_0)$ and $\epsilon + f_t$ on $B_1 \backslash B_{\delta}(x_0)$.
		One may check, using the super-mean-value property and \eqref{eq:superharmonic-poisson-inequality}, that $g$ is superharmonic in $B_1$. Hence, by \eqref{eq:lower-bound-continuity}, $g$ is admissible in \eqref{eq:lsc-leastsuperharmonic} and, in turn, $g \geq f_t$.
		Therefore, 
		\begin{align*}
		\limsup_{x \to x_0} f_t(x) &\leq \limsup_{x \to x_0} g(x) \quad \mbox{($g \geq f_t$ on $B_1$)} \\
		&= g(x_0)\quad \mbox{($g$ is harmonic and hence continuous in a neighborhood of $x_0$)} \\
		&\leq \epsilon + f_t(x_0) \quad \mbox{(by \eqref{eq:superharmonic-poisson-inequality} and definition of $g$)}.
		\end{align*}
		We conclude by observing the prior inequality holds for any $\epsilon > 0$.
	\end{proof}
	
	We next check that the least supersolution is harmonic on the non-coincidence set, this together with Lemma \ref{lemma:appendix-basic-properties} implies
	Lemma \ref{lemma:non-coincidence-open-harmonic}. 
	\begin{lemma} \label{lemma:appendix-non-coincedence-open-harmonic}
		The non-coincidence set, ${\cluster{}{t}}$, is open and connected and
		\[
		\Delta \lssshift{}{t} = 0 \quad \mbox{on ${\cluster{}{t}}$}.
		\]
	\end{lemma}
	
	\begin{proof}
		{\it Step 1: ${\cluster{}{t}}$ is open.} \\
		As $\lssshift{}{t}$ and $\beta_t$ are continuous and the disk $B_1$ is open, the set ${\cluster{}{t}}$
		is open (this is the topological definition of a continuous function). 
		
		{\it Step 2: $\lssshift{}{t}$ is harmonic on ${\cluster{}{t}}$.} \\
		If ${\cluster{}{t}}$ is empty, we are done, so suppose not. (This never happens but is proved later in Proposition \ref{prop:lower-bound}.) 
		Further suppose for sake of contradiction that $\lssshift{}{t}$ is not harmonic on ${\cluster{}{t}}$. Since we know that $\lssshift{}{t}$ is superharmonic (Lemma~\ref{lemma:appendix-basic-properties}), this means that $\lssshift{}{t}$ is not subharmonic on ${\cluster{}{t}}$.  
		
		The idea of the rest of the proof is the following. Since $\lssshift{}{t} > \beta_t$ on ${\cluster{}{t}}$, there is some extra room to `lower' $\lssshift{}{t}$. If $\Delta \lssshift{}{t}(z) < 0$ at some $z \in {\cluster{}{t}}$, then we can decrease $\lssshift{}{t}$ around $z$ by bending it up just enough to not break superharmonicity. This contradicts the minimality of $\lssshift{}{t}$. 
		We can't carry out this strategy literally since $\lssshift{}{t}$ is a priori not differentiable, so we instead use one of the equivalent definitions of subharmonic. 
		
		Here are the details. Since we are assuming that $\lssshift{}{t}$ is not subharmonic, by, \eg,
		~\cite[Theorem 3.2.2]{armitage2000classical}, there is some $z \in {\cluster{}{t}}$ such that for every $R > 0$, there is a closed ball 
		$\overline{B_r}(z) \subset {\cluster{}{t}}$ of radius $r < R$ and a function $H:\overline{B_r}(z) \to \R$ which is continuous in $\overline{B_r}(z)$ and harmonic in $B_r(z)$ such that 
		\begin{equation}  \label{eq:smaller}
		H \geq \lssshift{}{t} \quad \mbox{on $\partial B_r(z)$}  
		\end{equation}
		but
		\begin{equation} \label{eq:below-w}
		\lssshift{}{t}(x_0) > H(x_0)
		\end{equation}
		for some $x_0 \in B_r(z)$. 
		
		Next, since $\lssshift{}{t}$ and $\beta_t$ are continuous and $\overline{B_R}(z) \subset {\cluster{}{t}}$, for $R$ sufficiently small,
		\begin{equation} \label{eq:lies-way-above}
		\inf_{x \in \partial B_r(z)} \lssshift{}{t}(x) > \sup_{y \in \overline{B_r}(z)} \beta_t(y), \quad\forall r < R .
		\end{equation} 
		Fix $r > 0$ small so that \eqref{eq:smaller}, \eqref{eq:below-w}, and \eqref{eq:lies-way-above} hold.  As $H$ is harmonic, 
		\[
		\inf_{x \in B_r(z)} H(x) \geq \inf_{y \in \partial B_r(z)} H(y) \geq  \inf_{x \in \partial B_r(z)} \lssshift{}{t}(x) > \sup_{y \in \overline{B_r}(z)} \beta_t(y),
		\]
		in particular,
		\begin{equation} \label{eq:above-obstacle}
		H > \beta_t \quad \mbox{in $\overline{B_r}(z)$}. 
		\end{equation}
		The above inequalities allow us to `lower' $\lssshift{}{t}$ using $H$.  Indeed, take the function $\psi: \overline{B_r}(z) \to \R$ defined by
		\begin{equation}
		\psi := \min(H, \lssshift{}{t}) 
		\end{equation}
		and note that by continuity and \eqref{eq:smaller}, $\psi = \lssshift{}{t}$ in a neighborhood of $\partial B_r(z)$. In particular, 
		we may continuously extend $\psi$ to all of $\overline{B_1}$ by defining $\psi = \lssshift{}{t}$ on $\overline{B_1} \backslash B_r(z)$.
		As $H$ is harmonic in $B_r(z)$ and $\lssshift{}{t}$ is superharmonic, this extension $\psi$ is superharmonic. Also, by \eqref{eq:above-obstacle} $\psi \geq \beta_t$. 
		This shows that $\psi \in \mclSshift{}{t}$. However, by \eqref{eq:below-w},  $\lssshift{}{t}(x_0) > \psi(x_0)$, 
		contradicting the minimality of $\lssshift{}{t}$. 
		
		{\it Step 3: ${\cluster{}{t}}$ is connected.} \\
		Otherwise there is a connected component of ${\cluster{}{t}}$ not containing the origin
		upon which $\odometer{}{t}$ is non-zero, subharmonic, and 0 on its boundary --- this violates the strong maximum principle.
	\end{proof}

	\subsection{Monotonicity}
	We now check monotonicity, this proves Lemma \ref{lemma:monotonicity}.
	\begin{lemma} \label{lemma:appendix-monotonicity}
		If $t_1 \leq t_2$, then $ \cluster{}{t_1} \subseteq  \cluster{}{t_2}$.
	\end{lemma}
	\begin{proof}
		Recall that the odometer can be expressed as $\odometer{}{t} = \lssshift{}{t} - \beta_t$. 
		Showing monotonicity is equivalent to verifying $\odometer{}{t_1} \leq \odometer{}{t_2}$. 
		Unpack the difference to see that
		\begin{align*}
		\odometer{}{t_2} -  \odometer{}{t_1} &= \lssshift{}{t_2} - \lssshift{}{t_1} + \beta_{t_1} - \beta_{t_2} \\
		&= \lssshift{}{t_2} - \lssshift{}{t_1} + (-t_1+t_2) G_{B_1}(0,\cdot).
		\end{align*}
		This motivates considering the superharmonic function 
		\[
		\tilde{s} := \lssshift{}{t_2} + (t_2 - t_1) G_{B_1}(0,\cdot).
		\]
		In particular, we have 
		\[
		\odometer{}{t_2} - \odometer{}{t_1} = \tilde{s}-\lssshift{}{t_1}, 
		\]
		thus it suffices to show 
		\begin{equation} \label{eq:lower-bound}
		\tilde{s} \geq \lssshift{}{t_1}. 
		\end{equation}
		This inequality follows from the obstacle problem. Indeed, 
		as
		\[
		\lssshift{}{t_2} \geq \beta_{t_2} =  q_{B_1} - t_2 G_{B_1}(0,\cdot)
		\]
		we have, after plugging in the definition of $\tilde{s}$, 
		\[
		\tilde{s} \geq q_{B_1} - t_1 G_{B_1}(0,\cdot) = \beta_{t_1}.
		\]
		Therefore, $\tilde{s} \in \mclSshift{}{t_1}$ and by minimality of $\lssshift{}{t_1}$
		we have \eqref{eq:lower-bound}. 
	\end{proof}
	
	\subsection{Conservation of mass}
	In this section we prove that no mass comes in from the boundary, that is, $\mu_h({\cluster{}{t}}) \leq t$. 
	This establishes Lemma \ref{lemma:conservation-of-mass}. 
	To that end, we observe that the odometer is 0 on the boundary of the domain. 
	
	\begin{lemma} \label{lemma:appendix-zero-boundary}
		For all $t > 0$, $\odometer{}{t} = 0$ on $\partial B_1$.
	\end{lemma}
	
	\begin{proof}
		This is immediate from $q_{B_1} \in  \mclSshift{}{t}$ and $q_{B_1} = G_{B_1}(0,\cdot) = 0$ on $\partial B_1$
	\end{proof}

	We then use this together with the definition of weak normal derivative to prove the desired claim. 
	\begin{lemma} \label{lemma:appendix-conservation-of-mass}
		For all $t > 0$, $\mu({\cluster{}{t}}) \leq t$. Moreover, if $\overline{{{\cluster{}{t}}}} \subset B_1$ and $\mu(\partial {\cluster{}{t}}) = 0$, then $\mu_h({\cluster{}{t}}) = t$.
	\end{lemma}

	\begin{proof}
		Fix $t > 0$ and recall that ${\cluster{}{t}}$ is an open set. By Lemma \ref{lemma:appendix-zero-boundary}, Lemma \ref{lemma:appendix-non-coincedence-open-harmonic},
		the superharmonicity of $\lssshift{}{t}$, and the definition of $\odometer{}{t}$, 
		\begin{equation} \label{eq:laplacian-of-v}
		\begin{cases}
		\odometer{}{t} = 0  &\mbox{on $\partial B_1$} \\
		\Delta \odometer{}{t} = -t \delta_0 + \mu |_{_{{{\cluster{}{t}}}}} + \nu |_{_{\partial {\cluster{}{t}}}} &\mbox{on $B_1$}
		\end{cases}
		\end{equation}
		where $0 \leq \nu \leq \mu$ is a Radon measure which is absolutely continuous with respect to $\mu$. Hence, 
		$\odometer{}{t}$ solves a linear Dirichlet problem on $B_1$, and so, by ~\cite[Proposition 7.3]{ponce2016elliptic},
		there exists a weak normal derivative $\frac{ \partial \odometer{}{t}}{\partial n}$ so that 
		\begin{equation} \label{eq:weak-integration-by-parts}
		-t + \mu({\cluster{}{t}}) + \nu(\partial {\cluster{}{t}}) = \int_{\partial B_1} \frac{ \partial \odometer{}{t}}{\partial n} d \sigma,
		\end{equation}
		where $d \sigma$ denotes integration with respect to surface measure. 
		Since $\odometer{}{t} \geq 0$ on $B_1$ and $\odometer{}{t} = 0$ on $\partial B_1$, by ~\cite[Lemma 12.15]{ponce2016elliptic},
		\begin{equation}
		\frac{ \partial \odometer{}{t}}{\partial n} \leq 0
		\end{equation}
		almost everywhere with respect to the surface measure. Moreover,
		\begin{equation}
		\overline{{{\cluster{}{t}}}} \subset B_1 \implies \frac{ \partial \odometer{}{t}}{\partial n} = 0, 
		\end{equation}
		as the weak normal derivative coincides with the classical normal derivative if it exists. 
		In particular, by \eqref{eq:weak-integration-by-parts}, 
		\begin{equation}
		-t + \mu({\cluster{}{t}}) + \nu(\partial {\cluster{}{t}}) \leq 0,
		\end{equation}
		and
		\begin{equation}
		\overline{{{\cluster{}{t}}}} \subset B_1 \implies -t + \mu({\cluster{}{t}}) + \nu(\partial {\cluster{}{t}}) = 0, 
		\end{equation}
		completing the proof as $0 \leq \nu \leq \mu$.
	\end{proof}

	\subsection{Compatibility}
	We prove Lemma \ref{lemma:cluster-compatibility} in this section.
	As previously mentioned, the results proved so far in this appendix apply to $\lss{B_R}{t}, \cluster{B_R}{t}, \odometer{B_R}{t}$ as long as \eqref{eq:volume-growth}
	is satisfied for $\mu$ in $B_R$. That is, if \eqref{eq:volume-growth} is satisfied, then $\lss{B_R}{t} \in \mclS{B_R}{t}$ and
	\begin{equation} \label{eq:laplacian-of-v-general}
	\begin{cases}
	\odometer{B_R}{t} = 0  &\mbox{on $\partial B_R$} \\
	\Delta \odometer{B_R}{t} = -t \delta_0 + \mu |_{_{\cluster{B_R}{t}}} + \nu |_{\partial \cluster{B_R}{t}} &\mbox{on $B_R$}
	\end{cases}
	\end{equation}
	for a Radon measure $0 \leq \nu \leq \mu$ which is absolutely continuous with respect to $\mu$.

	\begin{lemma} \label{lemma:appendix-cluster-compatibility}
		Suppose \eqref{eq:volume-growth} is satisfied for fixed $R > 0$. 
		For all $s_1 \leq R$, if, for some $s_2 \in [s_1, R]$, we have $\overline{\cluster{B_{s_2}}{t}} \subset B_{s_1}$, 
		then  $\cluster{B_{s}}{t} = \cluster{B_{s_2}}{t}$ for all $s \in [s_1, R]$.
	\end{lemma}
	
	\begin{proof}
		Fix $R > 0$. Note that if \eqref{eq:volume-growth} is satisfied for $R$, then it is satisfied for all $s \leq R$.
		meaning $\lss{B_s}{t} \in \mclS{B_s}{t}$ and \eqref{eq:laplacian-of-v-general} holds for $\odometer{B_s}{t}$ for all $s \leq R$.  Let $s_1 \leq R$ be given and fix $s_2 \in [s_1, R]$ for which we have $\overline{\cluster{B_{s_2}}{t}} \subset B_{s_1}$.
		
		We first claim that
		\begin{equation} \label{eq:monotonicity-domains}
		\odometer{B_{r_1}}{t} \leq \odometer{B_{r_2}}{t} \quad \mbox{in $B_{r_1}$}, \quad \forall r_1 \leq r_2 \leq R.
		\end{equation}
		To prove~\eqref{eq:monotonicity-domains}, fix $r_1 \leq r_2 \leq R$ and write
		\begin{equation} \label{eq:intermediate-inequality-monotonicity}
		\odometer{B_{r_1}}{t} - \odometer{B_{r_2}}{t} =  \lss{B_{r_1}}{t} -  \lss{B_{r_2}}{t} + t G_{B_{r_1}}(0,\cdot) - t G_{B_{r_2}}(0,\cdot) = \lss{B_{r_1}}{t} - \hat{w}_t
		\end{equation}
		where
		\begin{equation}
		\hat{w}_t := \lss{B_{r_2}}{t} - t (G_{B_{r_1}}(0,\cdot) - G_{B_{r_2}}(0,\cdot))  .
		\end{equation}
		Since $\lss{B_{r_2}}{t} \in \mclS{B_{r_2}}{t}$ and $\Delta G_{B}(0,\cdot) = -\delta_0$ in $B$ for any ball $B$, 
		\begin{equation}
		\Delta \hat{w}_t = \Delta \lss{B_{r_2}}{t} - t (\Delta G_{B_{r_1}}(0,\cdot) - \Delta G_{B_{r_2}}(0,\cdot)) \leq \mu  \quad \mbox{in $B_{r_1}$}
		\end{equation}
		and 
		\begin{equation}
		\hat{w}_t = (\lss{B_{r_2}}{t}  + t G_{B_{r_2}}(0,\cdot)) - t G_{B_{r_1}}(0,\cdot)  \geq - t G_{B_{r_1}}(0,\cdot) \quad \mbox{in $B_{r_1}$}.
		\end{equation}
		Therefore, $\hat{w}_t \in \mclS{B_{r_1}}{t}$, which shows $\lss{B_{r_1}}{t} \leq \hat{w}_t$ and hence \eqref{eq:monotonicity-domains} by \eqref{eq:intermediate-inequality-monotonicity}.
		
		For the other direction, we use the hypothesis $\overline{\cluster{B_{s_2}}{t}} \subset B_{s_1}$. This together with \eqref{eq:monotonicity-domains}
		implies $\odometer{B_{s_1}}{t}$ is identically zero in a neighborhood of $\partial B_{s_1}$ and so can be extended
		by 0 to be harmonic in $B_R \backslash B_{s_1}$. Also, observe that \eqref{eq:monotonicity-domains} implies $\cluster{B_{s_1}}{t} \subset \cluster{B_R}{t}$ by definition.  Therefore, by \eqref{eq:laplacian-of-v-general} for $B_{s_1}$ and $B_R$, $\odometer{B_R}{t} - \odometer{B_{s_1}}{t}$ is subharmonic in $\cluster{B_R}{t}$ and 0 on its boundary
		which shows 
		\begin{equation} \label{eq:monotonicity-domains-two}
		\odometer{B_R}{t} \leq \odometer{B_{s_1}}{t}.
		\end{equation}
		Combining \eqref{eq:monotonicity-domains} and \eqref{eq:monotonicity-domains-two} shows that 
		\begin{equation}
		\odometer{B_{R}}{t} = \odometer{B_{s_1}}{t} \leq \odometer{B_{s}}{t} \leq \odometer{B_{r}}{t} \leq \odometer{B_R}{t}, \quad \forall s_1 \leq s \leq r \leq R,
		\end{equation}
		completing the proof by the definition of ${\cluster{}{t}}$. 
	\end{proof}

	\bibliographystyle{alpha}
	\bibliography{refs.bib}

\end{document}